\documentclass[french]{amsart}
\usepackage[T1]{fontenc}
\usepackage[applemac]{inputenc}
\usepackage{geometry}
\usepackage{babel}
\makeatletter
\addto\extrasfrench{%
   \providecommand{\og}{\leavevmode\flqq~}%
   \providecommand{\fg}{\ifdim\lastskip>\z@\unskip\fi~\frqq}%
}

\makeatother
\usepackage{longtable}
\usepackage{units}
\usepackage{mathrsfs}
\usepackage{mathtools}
\usepackage{amstext}
\usepackage{amsthm}
\usepackage{amssymb}
\usepackage[unicode=true]
 {hyperref}

\makeatletter

\providecommand{\tabularnewline}{\\}

\numberwithin{equation}{section}
\numberwithin{figure}{section}
\theoremstyle{plain}
\newtheorem{thm}{\protect\theoremname}
  \theoremstyle{plain}
  \newtheorem{prop}[thm]{\protect\propositionname}
  \theoremstyle{definition}
  \newtheorem{defn}[thm]{\protect\definitionname}
  \theoremstyle{remark}
  \newtheorem{rem}[thm]{\protect\remarkname}
  \theoremstyle{plain}
  \newtheorem{cor}[thm]{\protect\corollaryname}
  \theoremstyle{definition}
  \newtheorem{example}[thm]{\protect\examplename}
  \theoremstyle{plain}
  \newtheorem{lem}[thm]{\protect\lemmaname}

\usepackage{amscd}

\makeatother

\usepackage{babel}\addto\extrasfrench{\providecommand{\og}{\leavevmode\flqq~}\providecommand{\fg}{\ifdim\lastskip>\z@\unskip\fi~\frqq}}
\usepackage{tikz}

\usepackage{dsfont}

\usepackage[T1]{fontenc}
\usepackage{frcursive}
\usepackage{diagbox}

\allowdisplaybreaks[4]

\makeatother

  \providecommand{\corollaryname}{Corollaire}
  \providecommand{\definitionname}{Définition}
  \providecommand{\examplename}{Exemple}
  \providecommand{\lemmaname}{Lemme}
  \providecommand{\propositionname}{Proposition}
  \providecommand{\remarkname}{Remarque}
\providecommand{\theoremname}{Théorème}

\begin{document}
\begin{center}
\textsf{\textbf{\LARGE{}Identités pondérées Peirce-évanescentes}}\medskip{}
\par\end{center}

\begin{center}
\textbf{Richard Varro} \medskip{}
 
\par\end{center}

{\footnotesize{}Institut Montpelliérain Alexander Grothendieck,
Université de Montpellier, CNRS, Place Eugène Bataillon - 35095
Montpellier, France.}{\footnotesize \par}

\emph{\footnotesize{}E-mail}{\footnotesize{}: richard.varro@umontpellier.fr}{\footnotesize \par}

{\Large{}\medskip{}
}{\Large \par}

\textbf{Abstract} : {\small{}Peirce-evanescent baric identities
are polynomial identities verified by baric algebras such that
their Peirce polynomials are the null polynomial. In this paper
procedures for constructing such homogeneous and non homogeneous
identities are given. For this we define an algebraic system
structure on the free commutative nonassociative algebra generated
by a set $\mathcal{T}$ which provides for classes of baric algebras
satisfying a given set of identities similar properties to those
of the varieties of algebras. Rooted binary trees with labeled
leaves are used to explain the Peirce polynomials. It is shown
that the mutation algebras satisfy all Peirce-evanescent identities,
it results from this that any part of the field $K$ can be the
Peirce spectrum of a $K$-algebra satisfying a Peirce-evanescent
identity. We end by giving methods to obtain generators of homogeneous
and non-homogeneous Peirce-evanescent identities that are applied
in several univariate and multivariate cases. }{\small \par}

\medskip{}

\emph{Key words} : {\small{}Baric algebras, polynomial identities,
algebraic systems, variety of algebraic systems, $T$-ideal,
labeled rooted binary trees, altitude of a polynomial, Peirce
polynomial, mutation algebras.}{\small \par}

\emph{\small{}2010 MSC}{\small{} : Primary : 17D92, Secondary
: 17A30}.\medskip{}

\section{Introduction}

Les algèbres non associatives sont des algèbres non nécessairement
associatives dans lesquelles l'identité d'associativé $\left(xy\right)z-x\left(yz\right)=0$
est remplacée par une ou plusieurs identités polynomiales\footnote{\og Without associativity, rings and algebras are not in general
well enough behaved to have much of a structure theory. For this
reason, the nonassociative algebraists normally studies the class
of rings which satisfy some particular identity or set of identities. \fg{}
M. Osborn \cite{Osborn-86}.}. Ces identités polynomiales sont à une ou plusieurs indéterminées,
à coefficients constants ou variables. Dans les cas où ces algèbres
admettent un idempotent, un outil fondamental pour leur étude
est la décomposition de Peirce obtenue à partir du polynôme de
Peirce qui est un polynôme annulateur de l'opérateur de multiplication
à gauche $L_{e}:x\mapsto ex$, où $e\neq0$ est un idempotent.\medskip{}

Cependant il existe des algèbres définies par des identités polynomiales
pour lesquelles le polynôme de Peirce est nul. Illustrons cette
situation par un exemple, soit $A$ une $K$-algèbre commutative
sur un corps $K$ de caractéristique $\neq2$ vérifiant l'identité
\begin{equation}
x^{2}x^{2}-\alpha\omega\left(x\right)x^{3}-\left(1-\alpha\right)\omega\left(x\right)^{2}x^{2}=0\label{eq:Ret_alpha}
\end{equation}
 où $\omega:A\rightarrow K$ est un morphisme d'algèbre non nul.
On suppose qu'il existe dans $A$ un élément idempotent $e\neq0$,
de $e^{2}=e$ on déduit que $\omega\left(e\right)^{2}=\omega\left(e\right)$
d'où $\omega\left(x\right)\in\left\{ 0,1\right\} $, on a $\omega\left(e\right)\neq0$
sinon en posant $x=e$ dans l'identité \ref{eq:Ret_alpha} on
aurait $e=0$. La première linéarisation de l'identité \ref{eq:Ret_alpha}
est: 
\[
4x^{2}\left(xy\right)-\alpha\left[\omega\left(y\right)x^{3}+\omega\left(x\right)\left(x^{2}y+2x\left(xy\right)\right)\right]-2\left(1-\alpha\right)\left[\omega\left(x\right)\omega\left(y\right)x^{2}+\omega\left(x\right)^{2}xy\right]=0
\]
pour $y\in\ker\omega$ cette relation devient 
\[
4x^{2}\left(xy\right)-\alpha\omega\left(x\right)\left(x^{2}y+2x\left(xy\right)\right)-2\left(1-\alpha\right)\omega\left(x\right)^{2}xy=0
\]
en spécialisant cette identité pour $x=e$ on obtient
\[
\left(2-\alpha\right)\left(2e\left(ey\right)-ey\right)=0.
\]

Par conséquent, si $\alpha\neq2$ le polynôme de Peirce est $P\left(X\right)=2X^{2}-X$,
le spectre de $L_{e}$ est $\left\{ 0,\frac{1}{2}\right\} $
et on a $A=Ke\oplus A\left(0\right)\oplus A\left(\frac{1}{2}\right)$
où $A\left(\lambda\right)=\ker\left(L_{e}-\lambda id\right)$.
La seconde linéarisation de l'identité \ref{eq:Ret_alpha} aboutit
à $4\left(ey\right)\left(ez\right)+\left(2-\alpha\right)e\left(yz\right)-\alpha\left(\left(ey\right)z+y\left(ez\right)\right)-\left(1-\alpha\right)yz=0$
avec $y,z\in\ker\omega$ ce qui permet d'établir que $A\left(\frac{1}{2}\right)^{2}\subset A\left(0\right)$,
$A\left(0\right)A\left(\frac{1}{2}\right)\subset A\left(\frac{1}{2}\right)$
et $A\left(0\right)^{2}=\left\{ 0\right\} $ si $\alpha\neq0,1$
ou $A\left(0\right)^{2}\subset A\left(\frac{1}{2}\right)$ si
$\alpha=0$ et $A\left(0\right)^{2}\subset A\left(0\right)$
si $\alpha=1$. \medskip{}

En revanche si $\alpha=2$, l'identité \ref{eq:Ret_alpha} s'écrit
\begin{equation}
x^{2}x^{2}-2\omega\left(x\right)x^{3}+\omega\left(x\right)^{2}x^{2}=0\label{eq:Ret_id}
\end{equation}

Les algèbres vérifiant l'identité \ref{eq:Ret_id} sont dites
de rétrocroisement (backcrossing algebras) à cause de leur interprétation
génétique (cf. \cite{M-V-05}), elles sont apparues pour la première
fois dans \cite{M-V-94} et par la suite dans plusieurs autres
articles (voir les références dans \cite{M-V-??}). La linéarisation
et la spécialisation pour $x=e$ de \ref{eq:Ret_id} aboutit
à un polynôme de Peirce nul et de ce fait ne fournit aucune information
sur le spectre de $L_{e}$.

\medskip{}

Dans \cite{Tka-18}, V. Tkachev a appelé dégénérées ces identités
dont le polynôme de Peirce est nul, dans ce travail on préfère
les nommer évanescentes\footnote{Evanescent vient du participe présent \emph{evanescens} du verbe
latin \emph{evanescere} qui signifie “disparaître''. En effet
on observe que les termes du polynôme de Peirce disparaissent
au fur et à mesure du calcul.}\medskip{}

Ce papier est organisé comme suit. A la section 2 on munit le
groupoide commutatif $\mathfrak{M}\left(\mathcal{T}\right)$
engendré par un ensemble au plus dénombrable $\mathcal{T}$ d'une
structure de système algébrique, alors l'algèbre libre commutative
non associative engendrée par $\mathcal{T}$ qui en découle permet
d'obtenir pour les classes d'algèbres pondérées satisfaisant
un ensemble donné d'identités des propriétés analogues à celles
des variétés d'algèbres. A la section 3 on définit les linéarisées
des identités définies à partir des éléments de l'algèbre libre
obtenue à la section 2 et les polynômes de Peirce de ces identités,
on montre comment calculer les polynômes de Peirce à l'aide des
arbres binaires enracinés à feuilles étiquetées. On définit les
notions de polynômes et d'identités évanescents, on montre que
les algèbres de mutation vérifient toutes les identités évanescentes
et on en tire des conséquences sur le spectre de Peirce. On termine
à la section 4 en exposant des méthodes pour obtenir les générateurs
des polynômes évanescents homogènes et non homogènes, on applique
ces méthodes à plusieurs cas, on obtient ainsi un peu plus de
250 identités évanescentes.

\section{Variétés pour les algèbres pondérées.}

Dans tout ce travail, $K$ est un corps commutatif de caractéristique
$\neq2$ et les $K$-algèbres sont supposées commutatives.\medskip{}

Soit $\mathcal{T}=\left\{ t_{n};n\geq1\right\} $ un ensemble
dénombrable de symboles, on note $\mathfrak{M}\left(\mathcal{T}\right)$
le groupoïde commutatif engendré par $\mathscr{\mathcal{T}}$
muni de l'opération binaire, notée $\cdot$, et vérifiant pour
tout $t_{i},t_{j}\in\mathcal{T}$ et $u,v\in\mathfrak{M}\left(\mathcal{T}\right)$:
\begin{align*}
t_{i}\cdot t_{j} & =t_{i}t_{j}, & t_{i}\cdot u & =t_{i}\left(u\right), & v\cdot t_{i} & =\left(v\right)t_{i}, & u\cdot v & =\left(u\right)\left(v\right).
\end{align*}

Les éléments de $\mathfrak{M}\left(\mathcal{T}\right)$ sont
appelés des mots (ou monômes) non associatifs.

Pour $w\in\mathfrak{M}\left(\mathcal{T}\right)$, le degré de
$w$ en $t_{i}\in\mathscr{T}$, noté $\left|w\right|_{t_{i}}$
ou $\left|w\right|_{i}$, est le nombre d'occurrence de $t_{i}$
dans le monôme $w$, le degré de $w$ noté $\left|w\right|$
est la longueur du monôme $w$ autrement dit $\left|w\right|=\sum_{i\geq1}\left|w\right|_{i}$
et le type de $w$ est $\left[\left|w\right|_{1},\ldots,\left|w\right|_{n},\ldots\right]$.
Soient $\mathfrak{M}\left(\mathcal{T}\right)_{d}$ l'ensemble
des monômes de degré $d$ et $\mathfrak{M}\left(\mathcal{T}\right)_{\left[n_{1},\ldots,n_{m},\ldots\right]}$
l'ensemble des monômes de type $\left[n_{1},\ldots,n_{m},\ldots\right]$,
on a: 
\begin{align}
\mathfrak{M}\left(\mathcal{T}\right) & =\coprod_{d\geq1}\mathfrak{M}\left(\mathcal{T}\right)_{d},\label{eq:dec_M(T)}
\end{align}

et 
\begin{equation}
\mathfrak{M}\left(\mathcal{T}\right)_{d}=\coprod_{n_{1}+\cdots+n_{m}+\cdots=d}\mathfrak{M}\left(\mathcal{T}\right)_{\left[n_{1},\ldots,n_{m},\ldots\right]}.\label{eq: dec_M(T)_d}
\end{equation}

On a aussi le résultat suivant qui sera utilisé par la suite.
\begin{prop}
\emph{\label{prop:dec_ds_M(t)}{[}Proposition 2 in }\cite{Zh-Sl-Sh-82}\emph{{]}}
Every nonassociative word $w$ with $\left|w\right|\geq2$ has
a unique representation in the form of a product of two nonassociative
words of lesser length.
\end{prop}

\medskip{}

Pour ce qui suit on munit le groupoide commutatif $\left(\mathfrak{M}\left(\mathcal{T}\right),\cdot\right)$
d'une structure de système algébrique (\cite{Cohn-91}, chap
1) en définissant sur $\mathfrak{M}\left(\mathcal{T}\right)$
une loi de multiplication non commutative notée $\star$ (cette
loi est utilisée dans \cite{BGM-97} pour définir les algèbres
pondérées généralisées) et vérifiant pour tout $u,v,u',v',w\in\mathfrak{M}\left(\mathcal{T}\right)$,
les relations:
\begin{align}
\left(u\star v\right)\star w & =u\star\left(v\star w\right)=\left(u\cdot v\right)\star w;\label{eq:Ax_1}\\
\left(u\star v\right)\cdot w & =u\star\left(v\cdot w\right)=v\cdot\left(u\star w\right);\label{eq:Ax_2}\\
\left(u\star v\right)\cdot\left(u'\star v'\right) & =\left(u\cdot u'\right)\star\left(v\cdot v'\right);\label{eq:Ax_3}\\
\left(u\star v\right)\star\left(u'\star v'\right) & =\left(u\cdot v'\right)\star\left(u'\star v'\right).\label{eq:Ax_4}
\end{align}

On notera $\left(\mathfrak{M}\left(\mathcal{T}\right),\cdot,\star\right)$
ce système algébrique et pour allèger les notations on écrira
$\mathfrak{M}\left(\mathcal{T}\right)$ pour $\left(\mathfrak{M}\left(\mathcal{T}\right),\cdot\right)$
quand il n'y a pas de risque de confusion. Pour tout $u,v\in\left(\mathfrak{M}\left(\mathcal{T}\right),\cdot,\star\right)$
et tout $i\geq1$ on définit récursivement les degrés des éléments
de $\left(\mathfrak{M}\left(\mathcal{T}\right),\cdot,\star\right)$
par $\left|u\cdot v\right|_{i}=\left|u\star v\right|_{i}=\left|u\right|_{i}+\left|v\right|_{i}$.

\medskip{}

On a l'analogue de la proposition \ref{prop:dec_ds_M(t)} pour
les éléments de $\left(\mathfrak{M}\left(\mathcal{T}\right),\cdot,\star\right)$.
\begin{prop}
\label{prop:Dec_ds_Mt*}Tout élément $w$ de $\left(\mathfrak{M}\left(\mathcal{T}\right),\cdot,\star\right)$
se décompose de manière unique sous la forme $w=w_{1}w_{2}$
ou $w=w_{1}\star w_{2}$ avec $w_{1},w_{2}\in\left(\mathfrak{M}\left(\mathcal{T}\right),\cdot,\star\right)$
tels que $\left|w_{1}\right|,\left|w_{2}\right|<\left|w\right|$.
\end{prop}

\begin{proof}
C'est une conséquence de la proposition \ref{prop:dec_ds_M(t)}.
En effet, on peut considérer que tout élément de $\left(\mathfrak{M}\left(\mathcal{T}\right),\cdot,\star\right)$
s'obtient à partir d'un élément de $\mathfrak{M}\left(\mathcal{T}\right)$
en remplaçant dans celui-ci certaines opérations $\cdot$ par
des opérations $\star$.
\end{proof}
\begin{prop}
On a $\left(\mathfrak{M}\left(\mathcal{T}\right),\cdot,\star\right)=\mathfrak{M}\left(\mathcal{T}\right)\cup\left(\mathfrak{M}\left(\mathcal{T}\right)\star\mathfrak{M}\left(\mathcal{T}\right)\right)$.
\end{prop}

\begin{proof}
Montrons que pour tout $w\in\left(\mathfrak{M}\left(\mathcal{T}\right),\cdot,\star\right)$
tel que $w\notin\mathfrak{M}\left(\mathcal{T}\right)$ avec $\left|w\right|\geq2$,
il existe $u,v\in\mathfrak{M}\left(\mathcal{T}\right)$ tels
que $w=u\star v$. Par récurrence sur $\left|w\right|$. Si $\left|w\right|=2$,
il existe $t_{i},t_{j}\in\mathcal{T}$ tels que $w=t_{i}\star t_{j}$.
Si $\left|w\right|\geq3$, on suppose la propriété vraie pour
tout monôme de longueur $<\left|w\right|$. Soit $w\in\left(\mathfrak{M}\left(\mathcal{T}\right),\cdot,\star\right)\setminus\mathfrak{M}\left(\mathcal{T}\right)$
de degré $n$, d'après la proposition \ref{prop:Dec_ds_Mt*},
$w$ se décompose de façon unique sous la forme $w=w_{1}w_{2}$
ou $w=w_{1}\star w_{2}$ avec $w_{1}$ ou $w_{2}$ dans $\left(\mathfrak{M}\left(\mathcal{T}\right),\cdot,\star\right)$. 

Dans le cas $w=w_{1}w_{2}$ on a trois situations possibles:\smallskip{}

a) $w_{1}\in\mathfrak{M}\left(\mathcal{T}\right)$ et $w_{2}\in\left(\mathfrak{M}\left(\mathcal{T}\right),\cdot,\star\right)\setminus\mathfrak{M}\left(\mathcal{T}\right)$,
par hypothèse il existe $u_{2},v_{2}\in\mathfrak{M}\left(\mathcal{T}\right)$
avec $w_{2}=u_{2}\star v_{2}$, alors d'après la relation (\ref{eq:Ax_2})
on a $w=w_{1}\left(u_{2}\star v_{2}\right)=u_{2}\star\left(w_{1}v_{2}\right)$
où $w_{1}v_{2}\in\mathfrak{M}\left(\mathcal{T}\right)$;

b) $w_{1}\in\left(\mathfrak{M}\left(\mathcal{T}\right),\cdot,\star\right)\setminus\mathfrak{M}\left(\mathcal{T}\right)$
et $w_{2}\in\mathfrak{M}\left(\mathcal{T}\right)$, avec $w_{1}w_{2}=w_{2}w_{1}$
on est ramené au cas a);

c) $w_{1},w_{2}\in\left(\mathfrak{M}\left(\mathcal{T}\right),\cdot,\star\right)\setminus\mathfrak{M}\left(\mathcal{T}\right)$,
par hypothèse il existe $u_{1},u_{2},v_{1},v_{2}\in\mathfrak{M}\left(\mathcal{T}\right)$
tels que $w_{1}=u_{1}\star v_{1}$ et $w_{2}=u_{2}\star v_{2}$,
alors avec la relation (\ref{eq:Ax_3}) on obtient $w=\left(u_{1}\star v_{1}\right)\left(u_{2}\star v_{2}\right)=\left(u_{1}u_{2}\right)\star\left(v_{1}v_{2}\right)$
où $u_{1}u_{2},v_{1}v_{2}\in\mathfrak{M}\left(\mathcal{T}\right)$.\medskip{}

Dans le cas $w=w_{1}\star w_{2}$ on a quatre situations possibles:\smallskip{}

a) $w_{1},w_{2}\in\mathfrak{M}\left(\mathcal{T}\right)$, le
résultat est immédiat;

b) $w_{1}\in\mathfrak{M}\left(\mathcal{T}\right)$ et $w_{2}\in\left(\mathfrak{M}\left(\mathcal{T}\right),\cdot,\star\right)\setminus\mathfrak{M}\left(\mathcal{T}\right)$,
on a $w_{2}=u_{2}\star v_{2}$ avec $u_{2},v_{2}\in\mathfrak{M}\left(\mathcal{T}\right)$,
en utilisant les relations (\ref{eq:Ax_1}) on trouve $w=w_{1}\star w_{2}=w_{1}\star\left(u_{2}\star v_{2}\right)=\left(w_{1}u_{2}\right)\star v_{2}$;

c) $w_{1}\in\left(\mathfrak{M}\left(\mathcal{T}\right),\cdot,\star\right)\setminus\mathfrak{M}\left(\mathcal{T}\right)$
et $w_{2}\in\mathfrak{M}\left(\mathcal{T}\right)$, on a $w_{1}=u_{1}\star v_{1}$
où $u_{1},v_{1}\in\mathfrak{M}\left(\mathcal{T}\right)$ alors
il résulte aussitôt de (\ref{eq:Ax_1}) que $w=\left(u_{1}\star v_{1}\right)\star w_{2}=\left(u_{1}v_{1}\right)\star w_{2}$;

d) $w_{1},w_{2}\in\left(\mathfrak{M}\left(\mathcal{T}\right),\cdot,\star\right)\setminus\mathfrak{M}\left(\mathcal{T}\right)$,
on a par hypothèse $w_{1}=u_{1}\star v_{1}$ et $w_{2}=u_{2}\star v_{2}$
avec $u_{1},u_{2},v_{1},v_{2}\in\mathfrak{M}\left(\mathcal{T}\right)$
, alors en appliquant successivement les relations (\ref{eq:Ax_4})
et (\ref{eq:Ax_2}) on a $w=\left(u_{1}\star v_{1}\right)\star\left(u_{2}\star v_{2}\right)=\left(u_{1}v_{1}\right)\star\left(u_{2}\star v_{2}\right)=\left(\left(u_{1}v_{1}\right)\star u_{2}\right)\star v_{2}=\left(\left(u_{1}v_{1}\right)u_{2}\right)\star v_{2}$.
\end{proof}
Soit $K\left(\mathscr{\mathcal{T}}\right)$ la $K$-algèbre libre
commutative et non associative engendrée par $\mathfrak{M}\left(\mathcal{T}\right)$
(voir \cite{Zh-Sl-Sh-82}). Les éléments de $K\left\langle \mathcal{T}\right\rangle $
sont les polynômes non associatifs, ils sont de la forme $f=\sum_{k\geq1}\alpha_{k}w_{k}$
avec $w_{k}\in\mathfrak{M}\left(\mathcal{T}\right)$, $\alpha_{k}\in K$,
alors le degré de $f$, noté $\left|f\right|$, est $\left|f\right|=\max\left\{ \left|w_{k}\right|;\alpha_{k}\neq0\right\} $
et pour tout $t_{i}\in\mathcal{T}$, le degré de $f$ en $t_{i}$
est défini par $\left|f\right|_{i}=\max\left\{ \left|w_{k}\right|_{i};\alpha_{k}\neq0\right\} $.
On dit que $f$ est homogène si pour tout $t_{i}\in\mathscr{T}$
et tout $k\geq1$ on a $\left|w_{k}\right|_{i}=\left|f\right|_{i}$
autrement dit, si tous les monômes qui composent $f$ sont de
même type. 

On note $\left(K\left(\mathcal{T}\right),\cdot,\star\right)$
(resp. $K\left(\mathcal{T}\right)^{\star}$) l'algèbre libre
engendrée par $\left(\mathfrak{M}\left(\mathcal{T}\right),\cdot,\star\right)$
(resp. $\left(\mathfrak{M}\left(\mathcal{T}\right),\cdot\right)\star\left(\mathfrak{M}\left(\mathcal{T}\right),\cdot\right)$).
Il résulte de la proposition précédente que l'on a: 
\[
\left(K\left(\mathcal{T}\right),\cdot,\star\right)=K\left(\mathcal{T}\right)\oplus K\left(\mathcal{T}\right)^{\star}
\]
 et donc les éléments de $\left(K\left(\mathcal{T}\right),\cdot,\star\right)$
sont de la forme: 
\[
\sum_{i}\alpha_{i}w_{i}+\sum_{j}\beta_{j}u_{j}\star v_{j},\quad\left(\alpha_{i},\beta_{j}\in K;w_{i},u_{j},v_{j}\in\mathfrak{M}\left(\mathcal{T}\right)\right).
\]

Nous allons appliquer le système algèbrique $\left(K\left(\mathcal{T}\right),\cdot,\star\right)$
aux identités vérifiées par les algèbres pondérées.

Une $K$-algèbre $A$ est pondérée s'il existe un morphisme d'algèbres
non nul $\omega:A\rightarrow K$ appelé une pondération de $A$,
on note ceci $\left(A,\omega\right)$, l'image $\omega\left(x\right)$
d'un élément $x$ de $A$ est appelé le poids de $x$, on note
$H_{\left(A,\omega\right)}$ ou plus simplement $H_{\omega}$,
l'hyperplan affine $\left\{ x\in A:\omega\left(x\right)=1\right\} $
(cf. \cite{Ether-39}, \cite{WB-80}). Pour les algèbres pondérées
on a l'analogue de l'opération de substitution des symboles de
$\mathscr{T}$ par des éléments de l'algèbre (\cite{Osborn-72},
prop. 1.1). 
\begin{prop}
Soient $\left(A,\omega\right)$ une $K$-algèbre pondérée et
$\varrho:\mathcal{T}\rightarrow A$ une application. Il existe
un unique homomorphisme d'algèbres $\widehat{\varrho}$ de $\left(K\left(\mathcal{T}\right),\cdot,\star\right)$
dans $\left(A,\omega\right)$ tel que 
\[
\widehat{\varrho}\left(t_{i}\right)=\varrho\left(t_{i}\right),\quad\widehat{\varrho}\left(t_{i}t_{j}\right)=\varrho\left(t_{i}\right)\varrho\left(t_{j}\right),\quad\widehat{\varrho}\left(t_{i}\star t_{j}\right)=\omega\left(\varrho\left(t_{i}\right)\right)\varrho\left(t_{j}\right).
\]
\end{prop}

\begin{proof}
D'après les hypothèses l'application $\widehat{\varrho}$ est
définie pour tous les monômes de degré 2 de $\left(\mathfrak{M}\left(\mathcal{T}\right),\cdot,\star\right)$.
On suppose qu'elle est définie pour tous les monômes de $\left(\mathfrak{M}\left(\mathcal{T}\right),\cdot,\star\right)$
de degré $<n$, soit $w\in\left(\mathfrak{M}\left(\mathcal{T}\right),\cdot,\star\right)$
de degré $n$, on a $w=w_{1}w_{2}$ ou $w=w_{1}\star w_{2}$
avec $w_{1},w_{2}\in\mathfrak{M}\left(\mathcal{T}\right)$ de
degré $<n$, par hypothèse $\widehat{\varrho}\left(w_{1}\right)$
et $\widehat{\varrho}\left(w_{2}\right)$ sont définies et on
pose $\widehat{\varrho}\left(w\right)=\widehat{\varrho}\left(w_{1}\right)\widehat{\varrho}\left(w_{2}\right)$
si $w=w_{1}w_{2}$ et $\widehat{\varrho}\left(w\right)=\omega\left(\widehat{\varrho}\left(w_{1}\right)\right)\widehat{\varrho}\left(w_{2}\right)$
si $w=w_{1}\star w_{2}$, avec ceci et par unicité de la décomposition
de $w$, l'application $\widehat{\varrho}$ est bien définie
sur $\left(\mathfrak{M}\left(\mathcal{T}\right),\cdot,\star\right)$,
elle se prolonge par linéarité sur $\left(K\left(\mathcal{T}\right),\cdot,\star\right)$
en posant: $\widehat{\varrho}\left(\sum_{k\geq1}\alpha_{k}w_{k}\right)=\sum_{k\geq1}\alpha_{k}\widehat{\varrho}\left(w_{k}\right)$.
\end{proof}
\begin{defn}
\label{def:A_verifie_f}Étant donné $f$ un élément de $\left(K\left(\mathcal{T}\right),\cdot,\star\right)$
tel que $f\neq0$. On dit qu'une $K$-algèbre pondérée $\left(A,\omega\right)$
vérifie l'identité $f$ si on a: 
\begin{equation}
\widehat{\varrho}\left(f\right)=0,\label{eq:bar_id}
\end{equation}
pour toute application de substitution $\varrho:\mathcal{T}\rightarrow A$. 
\end{defn}

\begin{rem}
Plus généralement, le système algébrique $\left(K\left(\mathcal{T}\right),\cdot,\star\right)$
permet de définir la notion de weighted identity introduite dans
\cite{Tka-18}. Soient $A$ une $K$-algèbre commutative et $\varrho:\mathcal{T}\rightarrow A$
une application de substitution. Une application $\phi:A\rightarrow K$
est dite polynomiale si pour tout élément $a$ et $b$ de $A$
l'application $t\mapsto\phi\left(a+tb\right)$ est un polynôme.
Soit $\left\{ \phi_{w};w\in\mathfrak{M}\left(\mathcal{T}\right)\right\} $
une famille d'applications polynomiales donnée, il existe un
unique morphisme d'algèbres $\widehat{\varrho}$ de $\left(K\left(\mathcal{T}\right),\cdot,\star\right)$
dans $A$ tel que $\widehat{\varrho}\left(u\cdot v\right)=\widehat{\varrho}\left(u\right)\widehat{\varrho}\left(v\right)$
et $\widehat{\varrho}\left(u\star v\right)=\phi_{u}\left(\widehat{\varrho}\left(u\right)\right)\widehat{\varrho}\left(v\right)$
pour tout $u,v\in\mathfrak{M}\left(\mathcal{T}\right)$, alors
on dit que l'algèbre $A$ vérifie une weighted identity $f\in\left(K\left(\mathcal{T}\right),\cdot,\star\right)$
si $\widehat{\varrho}\left(f\right)=0$.
\end{rem}

Sous certaine condition, pour montrer qu'une algèbre pondérée
$\left(A,\omega\right)$ vérifie une identité $f$ il suffit
de montrer que $A$ vérifie $f$ pour les éléments de poids 1.
\begin{prop}
\label{prop:Bar-Id_pond1}Soit $f\in\left(K\left(\mathcal{T}\right),\cdot,\star\right)$.
Si le corps $K$ vérifie $\text{\emph{card}}K^{*}>\max\left\{ \left|f\right|_{i};i\geq1\right\} $,
alors une $K$-algèbre pondérée $\left(A,\omega\right)$ vérifie
l'identité $f$ si et seulement si on a:
\[
\widehat{\varrho}\left(f\right)=0,
\]
 pour toute application de substitution $\varrho:\mathcal{T}\rightarrow H_{\omega}$.
\end{prop}

\begin{proof}
La condition nécessaire est immédiate. Montrons que la condition
est suffisante. Soit $f\in K\left(\mathcal{T}\right)$ avec $f=\sum_{r\geq1}\alpha_{r}w_{r}+\sum_{s\geq1}\beta_{s}u_{s}\star v_{s}$
où $\alpha_{r},\beta_{s}\in K$ et $w_{r},u_{s},v_{s}\in\mathfrak{M}\left(\mathcal{T}\right)$
vérifiant $\widehat{\varrho}\left(f\right)=0$ pour toute application
$\varrho:\mathcal{T}\rightarrow H_{\omega}$. 

Pour $1\leq i\leq n$ on pose
\[
R_{i}=\left\{ r;\left|w_{r}\right|_{i}=\left|f\right|_{i}\right\} \;\text{et}\;S_{i}=\left\{ s;\left|u_{s}\star v_{s}\right|_{i}=\left|f\right|_{i}\right\} .
\]

Pour tout $a_{1}\in A$, $\omega\left(a_{1}\right)\neq0$ on
a $\omega\left(a_{1}\right)^{-1}a_{1}\in H_{\omega}$ , soit
$\left(x_{n}\right)_{n\geq2}$ où $x_{n}\in H_{\omega}$, en
prenant $\varrho\left(t_{1}\right)=\omega\left(a_{1}\right)^{-1}a_{1}$
et $\varrho\left(t_{i}\right)=x_{i}$ pour $i\geq2$, la condition
$\widehat{\varrho}\left(f\right)=0$ s'écrit: 
\begin{align}
\lefteqn{\sum_{r\in R_{1}}\alpha_{r}w_{r}\left(a_{1},x_{2},\ldots\right)+\sum_{r\notin R_{1}}\alpha_{r}\omega\left(a_{1}\right)^{\left|f\right|_{1}-\left|w_{r}\right|_{1}}w_{r}\left(a_{1},x_{2},\ldots\right)+}\nonumber \\
 & \qquad\sum_{s\in S_{1}}\beta_{s}v_{s}\left(a_{1},x_{2},\ldots\right)+\sum_{s\notin S_{1}}\beta_{s}\omega\left(a_{1}\right)^{\left|f\right|_{1}-\left|u_{s}\star v_{s}\right|_{1}}v_{s}\left(a_{1},x_{2},\ldots\right)=0.\label{eq:poids1_eq1}
\end{align}
Posons $f_{1}=\sum_{r\in R_{1}}\alpha_{r}w_{r}+\sum_{s\in S_{1}}\beta_{s}v_{s}$,
on a $\left|f_{1}\right|_{1}=\left|f\right|_{1}$ et (\ref{eq:poids1_eq1})
s'écrit: 
\begin{align}
f_{1}\left(a_{1},x_{2},\ldots\right)+\sum_{r\notin R_{1}}\alpha_{r}\omega\left(a_{1}\right)^{\left|f\right|_{1}-\left|w_{r}\right|_{1}}w_{r}\left(a_{1},x_{2},\ldots\right)+\qquad\nonumber \\
\sum_{s\notin S_{1}}\beta_{s}\omega\left(a_{1}\right)^{\left|f\right|_{1}-\left|u_{s}\star v_{s}\right|_{1}}v_{s}\left(a_{1},x_{2},\ldots\right) & =0,\label{eq:poids1_eq2}
\end{align}
autrement dit, on a obtenu $\widehat{\varrho}\left(f_{1}+\sum_{r\notin R_{1}}\alpha_{r}w_{r}+\sum_{s\notin S_{1}}\beta_{s}u_{s}\star v_{s}\right)=0$
ou $\widehat{\varrho}\left(f\right)=0$ pour $\varrho\left(t_{1}\right)=a_{1}$
et $\varrho\left(t_{i}\right)=x_{i}$, ($2\leq i$).

Ensuite pour tout $\left(x_{n}\right)_{n\geq1}$ où $x_{n}\in H_{\omega}$,
tout $z\in\ker\omega$ et $\lambda\in K$, comme $x_{1}+\lambda z\in H_{\omega}$
en prenant $\varrho\left(t_{1}\right)=x_{1}+\lambda z$ et $\varrho\left(t_{i}\right)=x_{i}$
pour $i\geq2$, d'après (\ref{eq:poids1_eq2}) l'identité $\widehat{\varrho}\left(f\right)=0$
s'écrit:
\[
\lambda^{\left|f\right|_{1}}f_{1}\left(z,x_{2},\ldots\right)+\sum_{k=0}^{\left|f\right|_{1}-1}\lambda^{k}g_{1,k}\left(z,x_{1},\ldots\right)=0
\]
où $g_{1,k}\in K\left(\mathcal{T}\right)$ avec $\left|g_{1,k}\right|_{1}=k$.
Par hypothèse on a $\text{card}K>\left|f\right|_{1}$, alors
en remplaçant $\lambda$ par des éléments $\lambda_{0},\ldots,\lambda_{\left|f\right|_{1}}$
de $K$ deux à deux distincts on obtient un système linéaire
homogène de $\left|f\right|_{1}+1$ équations d'inconnues $f_{1},g_{1,0},\ldots,g_{1,\left|f\right|_{1}-1}$
dont le déterminant est non nul, il en résulte que $f_{1}\left(z,x_{2},\ldots\right)=0$
ce qui d'après (\ref{eq:poids1_eq1}) ou (\ref{eq:poids1_eq2})
équivaut à $\widehat{\varrho}\left(f\right)=0$ pour $\varrho\left(t_{1}\right)=z$
et $\varrho\left(t_{i}\right)=x_{i}$ ($2\leq i$). On a donc
établit que $\widehat{\varrho}\left(f\right)=0$ pour $\varrho\left(t_{1}\right)=a_{1}$
et $\varrho\left(t_{i}\right)=x_{i}$ ($2\leq i$) quel que soit
$a_{1}\in A$ et $x_{2},\text{…},x_{n}\in H_{\omega}$.\medskip{}

En prenant $\varrho\left(t_{1}\right)=a_{1}$, $\varrho\left(t_{2}\right)=\omega\left(a_{2}\right)^{-1}a_{2}$
et $\varrho\left(t_{i}\right)=x_{i}$ ($3\leq i$) où $a_{1}\in A$,
$a_{2}\in A$, $\omega\left(a_{2}\right)\neq0$, $x_{n}\in H_{\omega}$
pour $n\geq3$, la condition $\widehat{\varrho}\left(f\right)=0$
conduit à 
\begin{align}
\sum_{r\in R_{2}}\alpha_{r}\omega\left(a_{1}\right)^{\left|f\right|_{1}-\left|w_{r}\right|_{1}}w_{r}\left(\mathbf{x}\right)+{\displaystyle \sum_{r\notin R_{2}}}\alpha_{r}\omega\left(a_{1}\right)^{\left|f\right|_{1}-\left|w_{r}\right|_{1}}\omega\left(a_{2}\right)^{\left|f\right|_{2}-\left|w_{r}\right|_{2}}w_{r}\left(\mathbf{x}\right)+\label{eq:poids1_eq3}\\
\sum_{s\in S_{2}}\beta_{s}\omega\left(a_{1}\right)^{\left|f\right|_{1}-\left|u_{s}\star v_{s}\right|_{1}}v_{s}\left(\mathbf{x}\right)+{\displaystyle \sum_{s\notin S_{2}}}\beta_{s}\omega\left(a_{1}\right)^{\left|f\right|_{1}-\left|u_{s}\star v_{s}\right|_{1}}\omega\left(a_{2}\right)^{\left|f\right|_{2}-\left|u_{s}\star v_{s}\right|_{2}}v_{s}\left(\mathbf{x}\right) & =0\nonumber 
\end{align}
où on a posé $\mathbf{x}=\left(a_{1},a_{2},x_{3},\ldots,x_{n},\ldots\right)$.

Ensuite avec $\varrho\left(t_{1}\right)=a_{1}$, $\varrho\left(t_{2}\right)=x_{2}+\lambda_{j}z$
($0\leq j\leq\left|f\right|_{2}$) et $\varrho\left(t_{i}\right)=x_{i}$
où $a_{1}\in A$, $x_{i}\in H_{\omega}$ ($i\geq3$), $z\in\ker\omega$
et $\lambda_{0},\ldots,\lambda_{\left|f\right|_{2}}\in K$ deux
à deux différents, on obtient 
\begin{equation}
\sum_{r\in R_{2}}\alpha_{i}\omega\left(a_{1}\right)^{\left|f\right|_{1}-\left|w_{r}\right|_{1}}w_{r}\left(a_{1},z,x_{3},\ldots\right)+\sum_{s\in S_{2}}\beta_{s}\omega\left(a_{1}\right)^{\left|f\right|_{1}-\left|u_{s}\star v_{s}\right|_{1}}v_{s}\left(a_{1},z,x_{3},\ldots\right)=0.\label{eq:poids1_eq4}
\end{equation}
 De (\ref{eq:poids1_eq3}) et (\ref{eq:poids1_eq4}) on déduit
que $\widehat{\varrho}\left(f\right)=0$ pour $\varrho\left(t_{1}\right)=a_{1}$,
$\varrho\left(t_{2}\right)=a_{2}$ et $\varrho\left(t_{i}\right)=x_{i}$
($3\leq i$), pour tout $a_{1},a_{2}\in A$, $x_{i}\in H_{\omega}$.

En poursuivant ainsi on obtient par récurrence (\ref{eq:bar_id}).
\end{proof}
L'introduction de la multiplication $\star$ dans $\mathfrak{M}\left(\mathcal{T}\right)$
permet d'obtenir pour les identités vérifiées par les algèbres
pondérées un résultat connu pour les variétés d'algèbres (cf.
\cite{Zh-Sl-Sh-82} Theorem 3).
\begin{prop}
\label{prop:Comp_homog}Soit $f\in\left(K\left(\mathcal{T}\right),\cdot,\star\right)$
une identité vérifiée par une $K$-algèbre pondérée $\left(A,\omega\right)$.
Si le corps $K$ vérifie $\text{\emph{card}}K^{*}>\max\left\{ \left|f\right|_{i};i\geq1\right\} $
alors chaque composante homogène de $f$ est une identité vérifiée
par $A$. 
\end{prop}

\begin{proof}
Soit $f\in\left(K\left(\mathcal{T}\right),\cdot,\star\right)$
une identité vérifiée par $\left(A,\omega\right)$. Pour tout
$d\geq0$ on note $f_{1,d}$ la somme des monômes de degré $d$
en $t_{1}$ de $f$, on a donc $f=\sum_{d=0}^{\left|f\right|_{1}}f_{1,d}$.
A toute application $\varrho:\mathcal{T}\rightarrow H_{\omega}$
on associe $\varrho_{1}:\mathcal{T}\rightarrow A$ telle que
$\varrho_{1}\left(t_{j}\right)=\varrho\left(t_{j}\right)$ si
$j\neq1$ et $\varrho_{1}\left(t_{1}\right)=\lambda\varrho\left(t_{1}\right)$
où $\lambda\in K$, $\lambda\neq0$. Alors de $\widehat{\varrho}_{1}\left(f\right)=0$
il résulte $\sum_{d=0}^{\left|f\right|_{1}}\lambda^{d}\widehat{\varrho}\left(f_{1,d}\right)=0$
pour tout $\lambda\in K$, en particulier en prenant pour $\lambda$
des valeurs non nulles $\lambda_{0},\ldots\lambda_{\left|f\right|_{1}}$
deux à deux distinctes on obtient un système de $\left|f\right|_{1}+1$
équations linéaires d'inconnues $\widehat{\varrho}\left(f_{1,d}\right)$
dont le déterminant de Vandermonde n'est pas nul, par conséquent
on a $\widehat{\varrho}\left(f_{1,d}\right)=0$ pour tout $d\geq0$,
autrement dit les polynômes $f_{1,0},\ldots,f_{1,\left|f\right|_{1}}$
sont des identités vérifiées par $A$.\medskip{}

En appliquant la même procédure pour l'indéterminée $t_{2}$
aux polynômes $f_{1,0},\ldots,f_{1,\left|f\right|_{1}}$ on obtient
des polynômes homogènes en $t_{1}$ et $t_{2}$ qui sont des
identités de $A$. Et en poursuivant ainsi pour toutes les variables
$t_{3},\ldots,t_{n},\ldots$ on établit le résultat. 
\end{proof}
\begin{rem}
\label{rem:Card(K*)}Compte tenu de l'importance des résultats
obtenus dans les propositions \ref{prop:Bar-Id_pond1} et \ref{prop:Comp_homog},
on supposera désormais que le corps $K$ vérifie la condition
énoncée dans ces propositions.
\end{rem}

\medskip{}

Soit $\left(K\left(\mathcal{T}\right),\cdot,\star\right)_{\left[n_{1},\ldots,n_{m},\ldots\right]}$
le sous espace des polynômes homogènes de type $\left[n_{1},\ldots,n_{m},\ldots\right]$,
il résulte de la proposition \ref{prop:Comp_homog} que
\[
\left(K\left(\mathscr{\mathcal{T}}\right),\cdot,\star\right)=\bigoplus_{\left(n_{1},\ldots,n_{m},\ldots\right)}\left(K\left(\mathcal{T}\right),\cdot,\star\right)_{\left[n_{1},\ldots,n_{m},\ldots\right]}.
\]

De la proposition \ref{prop:Comp_homog} on déduit immédiatement
la forme des identités vérifiées par les algèbres pondérées.
\begin{cor}
Les identités vérifiées par une algèbre pondérée $\left(A,\omega\right)$
sont de la forme:
\[
\sum_{k=1}^{m}\alpha_{k}\omega\left(a_{1}\right)^{\left|f\right|_{1}-\left|w_{k}\right|_{1}}\text{…}\;\omega\left(a_{n}\right)^{\left|f\right|_{n}-\left|w_{k}\right|_{n}}w_{k}\left(a_{1},\text{…},a_{n}\right)=0,\quad\forall a_{1},\text{…},a_{n}\in A;
\]
où $\alpha_{k}\in K$, $\alpha_{k}\neq0$ et $w_{k}\in\mathfrak{M}\left(\mathcal{T}\right)$
pour tout $1\leq k\leq m$. 
\end{cor}

\begin{proof}
Soit $f\in\left(K\left(\mathcal{T}\right),\cdot,\star\right)$
une identité de $\left(A,\omega\right)$. L'ensemble $I=\left\{ i;\left|f\right|_{i}\neq0\right\} $
est fini, en effet pour $i>\left|f\right|$ on a $\left|f\right|_{i}=0$.
Soit $n$ le cardinal de l'ensemble $I$, on reindexe les éléments
de $\mathscr{\mathcal{T}}$ pour que $I=\left\{ 1,\ldots,n\right\} $.
Ensuite il suffit de remarquer que pour chaque monôme de $f$
du type $u\star v$ avec $u,v\in\mathfrak{M}\left(\mathscr{\mathcal{T}}\right)$
tels que $\left|u\star v\right|_{i}=\left|f\right|_{i}$, si
$\varrho\left(t_{k}\right)=a_{k}$ ($1\leq k\leq n$) on a $\widehat{\varrho}\left(u\star v\right)=\omega\left(\widehat{\varrho}\left(u\right)\right)\widehat{\varrho}\left(v\right)$
avec $\omega\left(\widehat{\varrho}\left(u\right)\right)=\omega\left(a_{1}\right)^{\left|u\right|_{1}}\cdots\omega\left(a_{n}\right)^{\left|u\right|_{n}}$,
or on a $\left|u\star v\right|_{i}=\left|u\right|_{i}+\left|v\right|_{i}$
d'où $\left|u\right|_{i}=\left|f\right|_{i}-\left|v\right|_{i}$.
\end{proof}
Le point de vue des identités considérées dans le système algébrique
$\left(K\left(\mathcal{T}\right),\cdot,\star\right)$ permet
d'obtenir pour la classe des algèbres pondérées des résultats
analogues à ceux des variétés d'algèbres, ce que ne permet pas
de faire le point de vue restreint à la seule algèbre $K\left(\mathcal{T}\right)$,
par exemple la proposition \ref{prop:Comp_homog} n'est pas vraie
dans $K\left(\mathcal{T}\right)$. Néanmoins, l'utilisation de
l'algèbre $K\left(\mathcal{T}\right)$ est très utile pour écrire
de manière plus commode et manipuler les identités. En effet,
dans $\left(K\left(\mathcal{T}\right),\cdot,\star\right)$ l'écriture
d'une identité vérifiée par une algèbre $\left(A,\omega\right)$
n'est pas unique, par exemple, les polynômes $\left(t_{1}t_{2}\right)\left(t_{1}t_{2}\right)-\left(t_{1}t_{2}\right)t_{1}\star t_{2}$
et $\left(t_{1}t_{2}\right)\left(t_{1}t_{2}\right)-\left(t_{1}t_{1}\right)t_{2}\star t_{2}$
correspondent dans une algèbre $\left(A,\omega\right)$ à l'identité
$\left(xy\right)^{2}-\omega\left(x\right)^{2}\omega\left(y\right)y=0$,
($x,y\in A$), dont l'écriture dans $K\left(\mathcal{T}\right)$
est $\left(t_{1}t_{2}\right)\left(t_{1}t_{2}\right)-t_{2}$. 
\begin{defn}
Soit $f\in K\left(\mathcal{T}\right)$, $f=\sum_{k\geq1}\alpha_{k}w_{k}$,
on appelle homogénéisation de $f$, un polynôme $f^{\star}\in\left(K\left(\mathcal{T}\right),\cdot,\star\right)$
défini par:

\[
f^{\star}=\sum_{k\geq1}\alpha_{k}\widetilde{w}_{k}\star w_{k},
\]
où pour tout $k,i\geq1$ on a $\widetilde{w}_{k}\in\mathfrak{M}\left(\mathcal{T}\right)$
avec $\left|\widetilde{w}_{k}\right|_{i}=\left|f\right|_{i}-\left|w_{k}\right|_{i}$
.
\end{defn}

\begin{prop}
\label{prop:f*_et_f}Soient $\left(A,\omega\right)$ une $K$-algèbre
et $f\in K\left(\mathcal{T}\right)$, les énoncés suivants sont
équivalents

i) $A$ vérifie toutes les homogénéisations $f^{\star}$ de $f$,

ii) A vérifie une homogénéisation $f^{\star}$ de $f$,

iii) on a $\widehat{\varrho}\left(f\right)=0$ pour toute application
$\varrho:\mathcal{T}\rightarrow H_{\omega}$.
\end{prop}

\begin{proof}
L'implication $i)\Rightarrow ii)$ est immédiate. 

Pour la suite on remarque que pour toute application $\varrho:\mathcal{T}\rightarrow H_{\omega}$
et tout $w\in\mathfrak{M}\left(\mathcal{T}\right)$ on a $\widehat{\varrho}\left(w\right)=1$,
on en déduit que $\widehat{\varrho}\left(f^{\star}\right)=\widehat{\varrho}\left(f\right)$
pour tout $f\in K\left(\mathcal{T}\right)$ et tout $\varrho:\mathcal{T}\rightarrow H_{\omega}$. 

$ii)\Rightarrow iii)$ Par conséquent si $f^{\star}$ est une
identité de $A$, on a $\widehat{\varrho}\left(f^{\star}\right)=0$
et donc $\widehat{\varrho}\left(f\right)=\widehat{\varrho}\left(f^{\star}\right)=0$
quel que soit $\varrho:\mathcal{T}\rightarrow H_{\omega}$. . 

$iii)\Rightarrow i)$ Réciproquement si on a $\widehat{\varrho}\left(f\right)=0$
pour toute application $\varrho:\mathcal{T}\rightarrow H_{\omega}$
alors pour toute homogénéisation $f^{\star}$ de $f$ on a $\widehat{\varrho}\left(f^{\star}\right)=\widehat{\varrho}\left(f\right)=0$
ce qui entraîne d'après la proposition \ref{prop:Bar-Id_pond1}
que l'algèbre $A$ vérifie l'identité $f^{\star}$.
\end{proof}
Ce résultat conduit naturellement à poser la définition suivante.
\begin{defn}
\label{def:Id ds K(T)}Soient $\left(A,\omega\right)$ une $K$-algèbre
et $f\in K\left(\mathcal{T}\right)$, on dit que $f$ est une
identité vérifiée par $A$ si l'algèbre $A$ vérifie toute homogénéisation
de $f$.
\end{defn}

Les algèbres pondérées ne vérifient pas nécessairement une identité,
cependant dans certains cas l'existence d'une identité est assurée. 
\begin{prop}
Si $\left(A,\omega\right)$ est de dimension finie alors l'algèbre
$A$ vérifie une identité.
\end{prop}

\begin{proof}
Soit $d+1$ la dimension de $\left(A,\omega\right)$. Le résultat
est vrai pour $d=0$ car dans ce cas $A\simeq Ke$ avec $e^{2}=e$
et $\omega\left(e\right)=1$. Supposons $d\geq1$, soit $\left(e_{1},\ldots,e_{d}\right)$
une base de $\ker\omega$. Pour $z\in\ker\omega$ soit $L_{z}:x\mapsto zx$,
l'application $L_{z}$ est un endomorphisme de $\ker\omega$
et l'ensemble $L=\left\{ L_{z};z\in\ker\omega\right\} $ est
sous-espace de $\text{End}\left(\ker\omega\right)$ engendré
par $\left\{ L_{e_{1}},\ldots,L_{e_{d}}\right\} $ donc (\cite{Zh-Sl-Sh-82},
lemma 5, p. 103) vérifie l'identité
\[
P\left(t_{1},\ldots,t_{d}\right)=\sum_{s\in S_{d}}\left(-1\right)^{\text{sgn}\sigma}t_{\sigma\left(1\right)}\cdots t_{\sigma\left(d\right)}.
\]
De $P\left(L_{z_{1}},\ldots,L_{z_{d}}\right)y=0$ pour tout $z_{1},\ldots,z_{d}\in\ker\omega$
et $y\in A$ on déduit que pour tout $a_{1},\ldots,a_{d}\in H_{\omega}$
on a $P\left(L_{a_{1}^{2}-a_{1}},\ldots,L_{a_{d}^{2}-a_{d}}\right)y=0$
quel que soit $y\in A$, autrement dit $A$ vérifie l'identité:
\[
\sum_{s\in S_{d}}\left(-1\right)^{\text{sgn}\sigma}\left(t_{\sigma\left(1\right)}^{2}-t_{\sigma\left(1\right)}\right)\cdots\left(t_{\sigma\left(d\right)}^{2}-t_{\sigma\left(d\right)}\right)t_{d+1},
\]
d'où le résultat.
\end{proof}
L'exemple qui suit montre qu'en général ce résultat n'est pas
vérifié en dimension infinie.
\begin{example}
Soient $\mathcal{T}=\left\{ t_{n};n\geq1\right\} $ et $\mathfrak{M}\left(\mathcal{T}\right)_{n}$
l'ensemble des monômes de degré $n$. On a $\mathfrak{M}\left(\mathcal{T}\right)=\bigcup_{n\geq1}\mathfrak{M}\left(\mathcal{T}\right)_{n}$
et $\mathfrak{M}\left(\mathcal{T}\right)_{n}\subset\bigcup_{p+q=n}\mathfrak{M}\left(\mathcal{T}\right)_{p}\times\mathfrak{M}\left(\mathcal{T}\right)_{q}$
donc l'ensemble $\mathfrak{M}\left(\mathcal{T}\right)$ est dénombrable.
Soit $\varphi:\mathfrak{M}\left(\mathcal{T}\right)\rightarrow\mathbb{N}$
une énumération bijective. On considère l'algèbre $A$ de base
$\left(e_{\varphi\left(w\right)}\right)_{w\in\mathfrak{M}\left(\mathcal{T}\right)}$
définie par $e_{\varphi\left(u\right)}e_{\varphi\left(v\right)}=e_{\varphi\left(v\right)}e_{\varphi\left(u\right)}=e_{\varphi\left(uv\right)}$
où $u,v\in\mathfrak{M}\left(\mathcal{T}\right)$. L'algèbre $A$
est commutative, non associative car $\left(e_{\varphi\left(t_{1}\right)}e_{\varphi\left(t_{2}\right)}\right)e_{\varphi\left(t_{3}\right)}=e_{\varphi\left(\left(t_{1}t_{2}\right)t_{3}\right)}$
et $e_{\varphi\left(t_{1}\right)}\left(e_{\varphi\left(t_{2}\right)}e_{\varphi\left(t_{3}\right)}\right)=e_{\varphi\left(t_{1}\left(t_{2}t_{3}\right)\right)}$
avec $\left(t_{1}t_{2}\right)t_{3}\neq t_{1}\left(t_{2}t_{3}\right)$,
elle est pondérée par $\omega\left(e_{\varphi\left(w\right)}\right)=1$.
Supposons que $A$ vérifie une identité $f\in K\left(\mathcal{T}\right)$
avec $f=\sum_{k\geq1}\alpha_{k}w_{k}$ alors pour $\varrho:\mathcal{T}\rightarrow A$
telle que $\varrho\left(t_{i}\right)=e_{\varphi\left(t_{i}\right)}$
on a $\widehat{\varrho}\left(f\right)=\sum_{k\geq1}\alpha_{k}e_{\varphi\left(w_{k}\right)}$
donc $\widehat{\varrho}\left(f\right)\neq0$, contradiction.
Ainsi l'algèbre $\left(A,\omega\right)$ ne vérifie aucune identité.
\end{example}

Soit $\left(A,\omega\right)$ une $K$-algèbre pondérée, on note
$Id\left(A\right)$ (resp. $\left(Id\left(A\right),\cdot,\star\right)$)
l'ensemble des éléments de $K\left(\mathcal{T}\right)$ (resp.
$\left(K\left(\mathcal{T}\right),\cdot,\star\right)$) qui sont
des identités vérifiées par $A$. D'après la proposition \ref{prop:f*_et_f}
l'ensemble $Id\left(A\right)$ est non vide si et seulement si
il en est de même de l'ensemble $\left(Id\left(A\right),\cdot,\star\right)$.
Il est clair que $Id\left(A\right)$ (resp. $\left(Id\left(A\right),\cdot,\star\right)$)
est une $K$-algèbre et un idéal bilatère de $K\left(\mathcal{T}\right)$
(resp. $\left(K\left(\mathcal{T}\right),\cdot,\star\right)$).
On a vu à la proposition \ref{prop:Comp_homog} que les éléments
de $\left(Id\left(A\right),\cdot,\star\right)$ sont homogènes.
En revanche d'après la définition \ref{def:Id ds K(T)}, l'idéal
$Id\left(A\right)$ peut contenir à la fois des polynômes homogènes
et non homogènes, mais contrairement à ce qu'on a montré pour
l'idéal $\left(Id\left(A\right),\cdot,\star\right)$, les composantes
homogènes d'un élément non homogène de $Id\left(A\right)$ ne
sont pas toujours des identités de $A$. Si $\mathscr{H}\left(\mathcal{T}\right)$
et $\overline{\mathscr{H}}\left(\mathcal{T}\right)$ dénotent
respectivement l'ensemble des polynômes homogènes et non homogènes,
la partition $K\left(\mathcal{T}\right)=\mathscr{H}\left(\mathcal{T}\right)\sqcup\overline{\mathscr{H}}\left(\mathcal{T}\right)$
induit la partition de $Id\left(A\right)$ en deux sous-ensembles:
$Id\left(A\right)=\mathscr{H}\left(A\right)\sqcup\overline{\mathscr{H}}\left(A\right)$
où $\mathcal{\mathscr{H}}\left(A\right)=Id\left(A\right)\cap\mathscr{H}\left(\mathcal{T}\right)$
et $\overline{\mathscr{H}}\left(A\right)=Id\left(A\right)\cap\overline{\mathscr{H}}\left(\mathcal{T}\right)$.
Dans ce qui suit on étudie les propriétés de $Id\left(A\right)$
et $\left(Id\left(A\right),\cdot,\star\right)$. 
\begin{prop}
\label{prop:H(A)}Soit $\left(A,\omega\right)$ une $K$-algèbre. 

a) Pour tout $f\in Id\left(A\right)$ et tout $f\in\left(Id\left(A\right),\cdot,\star\right)$
on a: $f\left(\mathds{1}\right)=0$, où $\mathds{1}=\left(1,\ldots,1,\ldots\right)$.

b) Si $Id\left(A\right)\neq\textrm{Ø }$ alors $\overline{\mathscr{H}}\left(A\right)\neq\textrm{Ø}$. 

c) Soit $f\in\overline{\mathcal{H}}\left(A\right)$ alors pour
tout $i\geq1$ tel que $\left|f\right|_{i}\neq0$ il existe $g,h\in K\left(\mathscr{\mathcal{T}}\right)$
vérifiant les conditions : $f=g-h$, $\left|g\right|_{i}=\left|f\right|_{i}$,
$\left|g\right|_{i}>\left|h\right|_{i}$ et $g\left(\mathds{1}\right)=h\left(\mathds{1}\right)$.
\end{prop}

\begin{proof}
a) Soit $f\in Id\left(A\right)$, $f=\sum_{k\geq1}\alpha_{k}w_{k}$
où $\alpha_{k}\in K$ et $w_{k}\in\mathfrak{M}\left(\mathcal{T}\right)$,
de $w_{k}\left(\mathds{1}\right)=1$ il vient $f\left(\mathds{1}\right)=\sum_{k\geq1}\alpha_{k}w_{k}\left(\mathds{1}\right)=\sum_{k\geq1}\alpha_{k}$.
Mais d'après la proposition \ref{prop:Bar-Id_pond1} on a $\sum_{k\geq1}\alpha_{k}\widehat{\varrho}\left(w_{k}\right)=0$
pour tout $\varrho:\mathscr{\mathscr{\mathcal{T}}}\rightarrow H_{\omega}$,
or $\omega\left(\widehat{\varrho}\left(w_{k}\right)\right)=1$
et en appliquant la pondération $\omega$ à la relation $\widehat{\varrho}\left(f\right)=0$
on obtient $\sum_{k\geq1}\alpha_{k}=0$. 

Soit $f\in\left(Id\left(A\right),\cdot,\star\right)$, on a $f=\sum_{k\geq1}\alpha_{k}w_{k}+\sum_{l\geq1}\beta_{l}u_{l}\star v_{l}$
où $\alpha_{k},\beta_{l}\in K$ et $w_{k},u_{l},v_{l}\in\mathfrak{M}\left(\mathcal{T}\right)$,
alors $f\left(\mathds{1}\right)=\sum_{k\geq1}\alpha_{k}+\sum_{l\geq1}\beta_{l}$.
Pour toute application $\varrho:\mathcal{T}\rightarrow H_{\omega}$
on a $\omega\left(\widehat{\varrho}\left(w_{k}\right)\right)=1$
et $\omega\left(\widehat{\varrho}\left(u_{l}\star v_{l}\right)\right)=\omega\left(\widehat{\varrho}\left(u_{l}\right)\right)\omega\left(\widehat{\varrho}\left(v_{l}\right)\right)=1$,
alors en appliquant la pondération $\omega$ à la relation $\widehat{\varrho}\left(f\right)=0$
on obtient $\sum_{k\geq1}\alpha_{k}+\sum_{l\geq1}\beta_{l}=0$.

b) Le résultat est immédiat si $\mathcal{H}\left(A\right)=\textrm{Ø}$.
Si $\mathcal{H}\left(A\right)\neq\textrm{Ø}$, soit $f\in\mathcal{H}\left(A\right)$
et $i\geq1$ tel que $\left|f\right|_{i}\neq0$, on a $f\left(t_{1},\ldots,t_{i}^{2}\ldots,t_{n},\ldots\right)\in Id\left(A\right)$
avec $\left|f\left(t_{1},\ldots,t_{i}^{2}\ldots,t_{n},\ldots\right)\right|_{1}>\left|f\left(t_{1},\ldots,t_{i}\ldots,t_{n},\ldots\right)\right|$,
donc $f\left(t_{1},\ldots,t_{i}^{2}\ldots,t_{n},\ldots\right)-f\left(t_{1},\ldots,t_{i}\ldots,t_{n},\ldots\right)\in\overline{\mathcal{H}}\left(A\right)$.

c) \'{E}tant donné $f\in\overline{\mathcal{H}}\left(A\right)$,
$f=\sum_{k\geq1}\alpha_{k}w_{k}$ où $\alpha_{k}\in K$, $w_{k}\in\mathfrak{M}\left(\mathcal{T}\right)$.
Soit $i\geq1$ tel que $\left|f\right|_{i}\neq0$ on pose $I_{i}=\left\{ k;\left|w_{k}\right|_{i}=\left|f\right|_{i}\right\} $
alors $f=\sum_{k\in I_{i}}\alpha_{k}w_{k}+\sum_{k\notin I_{i}}\alpha_{k}w_{k}$
donc en prenant $g=\sum_{k\in I_{i}}\alpha_{k}w_{k}$ et $h=g-f$
on a $\left|g\right|_{i}=\left|f\right|_{i}$, $\left|g\right|_{i}>\left|h\right|_{i}$
et de $f\left(\mathds{1}\right)=0$ il vient $g\left(\mathds{1}\right)-h\left(\mathds{1}\right)=0$.
\end{proof}
Une autre différence entre les les idéaux $\left(Id\left(A\right),\cdot,\star\right)$
et $Id\left(A\right)$ concerne la propriété de \emph{$T$}-idéal.
Un $T$-idéal de $K\left(\mathcal{T}\right)$ est un idéal bilatère
de $K\left(\mathcal{T}\right)$ qui est stable par substitution
des indéterminées $t_{i}$ par tout élément de $K\left(\mathcal{T}\right)$
ou, ce qui est équivalent, stable par tout endomorphisme de $K\left(\mathcal{T}\right)$. 
\begin{prop}
Soit $A$ une algèbre pondérée telle que $Id\left(A\right)\neq\left\{ 0\right\} $,
alors

a) L'idéal $\left(Id\left(A\right),\cdot,\star\right)$ est un
$T$-idéal de $\left(K\left(\mathcal{T}\right),\cdot,\star\right)$.

b) Si $\text{card}K>2\mu$ où $\mu=\min\left(\left\{ \left|f\right|_{i};f\in Id\left(A\right),i\geq1\right\} \setminus\left\{ 0\right\} \right)$,
alors l'idéal $Id\left(A\right)$ n'est pas un $T$-idéal de
$K\left(\mathcal{T}\right)$. 
\end{prop}

\begin{proof}
a) Soit $f\in\left(Id\left(A\right),\cdot,\star\right)$, à toute
famille $\left(f_{n}\right)_{n\geq1}$ d'éléments de $\left(K\left(\mathcal{T}\right),\cdot,\star\right)$
et $\left(a_{n}\right)_{n\geq1}$ d'éléments de $A$ on associe
l'application $\varrho:\mathcal{T}\rightarrow A$, $\varrho\left(t_{i}\right)=f_{i}\left(a_{1},\ldots,a_{n},\ldots\right)$
alors de $\widehat{\varrho}$$\left(f\right)=0$ on déduit que
$f\left(f_{1},\ldots,f_{n},\ldots\right)\in\left(Id\left(A\right),\cdot,\star\right)$.\medskip{}

b) Supposons par l'absurde que $Id\left(A\right)$ est un $T$-idéal.
Soient $f\text{\ensuremath{\in}}\overline{\mathcal{H}}\left(A\right)$,
$f=\sum_{k\geq1}\alpha_{k}w_{k}$ et $i\geq1$ tel que $\left|f\right|_{i}=\mu$.
On pose $I=\left\{ k;\left|w_{k}\right|_{i}=\left|f\right|_{i}\right\} $,
on a donc $\left|w_{k}\right|_{i}<\left|f\right|_{i}$ si $k\notin I$
et $f=\sum_{k\in I}\alpha_{k}w_{k}+\sum_{k\notin I}\alpha_{k}w_{k}$.
De $\text{card}K>2\mu$ on déduit qu'il existe $\alpha\in K$
tel que $\sum_{k\in I}\alpha_{k}\alpha^{2\mu}+\sum_{k\notin I}\alpha_{k}\alpha^{2\left|w_{k}\right|_{i}}\neq0$,
alors en prenant une famille $\left(x_{n}\right)_{n\geq1}$ d'éléments
de $H_{\omega}$ et $\varrho:\mathcal{T}\rightarrow A$ vérifant
$\varrho\left(t_{k}\right)=x_{k}$ si $k\neq i$ et $\varrho\left(t_{i}\right)=\alpha x_{i}$
on a par hypothèse $\widehat{\varrho}\left(f\right)=0$, or $\omega\circ\widehat{\varrho}\left(f\right)=\sum_{k\in I}\alpha_{k}\alpha^{2\mu}+\sum_{k\notin I}\alpha_{k}\alpha^{2\left|w_{k}\right|_{i}}$
d'où une contradiction, on a montré que $f\left(t_{1},\ldots,\alpha t_{i}^{2},\ldots,t_{n},\ldots\right)\notin Id\left(A\right)$.
\end{proof}
L'idéal $Id\left(A\right)$ vérifie une notion affaiblie de $T$-idéal.
\begin{rem}
\label{rem:T-ideal-stoch}Notons $\Delta K\left(\mathcal{T}\right)$
l'ensemble des polynômes $h\in K\left(\mathcal{T}\right)$ tels
que $h\left(\mathds{1}\right)=1$. Un idéal $I$ de $K\left(\mathcal{T}\right)$
est un $T$-idéal stochastique (cf. \cite{BGM-97}, p. 388) si
$I$ est invariant par remplacement des symboles $t_{i}$ par
tout élément $h$ de $\Delta K\left(\mathcal{T}\right)$. Soit
$\left(A,\omega\right)$ une $K$-algèbre, comme pour tout $\left(x_{n}\right)_{n\geq1}$
tel que $x_{n}\in H_{\omega}$ et tout $h\in\Delta K\left(\mathcal{T}\right)$
on a $\omega\bigl(h\left(x_{1},\ldots,x_{n},\ldots\right)\bigr)=h\left(\mathds{1}\right)=1$,
il en résulte que $\widehat{\varrho}h\left(t_{1},\ldots,t_{n},\ldots\right)\in H_{\omega}$
par conséquent si $f\in Id\left(A\right)$ d'après la définition
\ref{def:Id ds K(T)} et la proposition \ref{prop:f*_et_f},
pour toute famille $\left(h_{n}\right)_{n\geq1}$ d'éléments
de $\Delta K\left(\mathcal{T}\right)$ on a $\widehat{\varrho}f\left(h_{1},\ldots,h_{n},\ldots\right)=0$
autrement dit $f\left(h_{1},\ldots,h_{n},\ldots\right)\in Id\left(A\right)$
et $Id\left(A\right)$ est un $T$-idéal stochastique.
\end{rem}

\section{Polynômes de Peirce - Identités Peirce-évanescentes}

Dans toute la suite de ce travail le symbole $t$ est une lettre
n'appartenant pas à l'ensemble $\mathcal{T}$. \medskip{}

Soit $f\in\left(K\left(\mathcal{T}\right),\cdot,\star\right)$,
pour tout $i\geq1$, on a $f\left(t_{1},\ldots,t_{i}+t,\ldots\right)\in\left(K\left(\mathcal{T}\cup\left\{ t\right\} \right),\cdot,\star\right)$
et le développement du polynôme $f\left(t_{1},\ldots,t_{i}+t,\ldots\right)$
peut s'écrire sous la forme 
\[
f\left(t_{1},\ldots,t_{i}+t,\ldots\right)=\sum_{0\leq k\leq\left|f\right|_{i}}\mathscr{L}_{i,k}\left(f\right)\left(t,t_{1},\ldots,t_{i},\ldots\right)
\]
 où pour tout $0\leq k\leq\left|f\right|_{i}$ on a $\mathscr{L}_{i,k}\left(f\right)\in\left(K\left(\mathcal{T}\cup\left\{ t\right\} \right),\cdot,\star\right)$
et $\left|\mathscr{L}_{i,k}\left(f\right)\right|_{t}=k$. 

Le polynôme $\mathscr{L}_{i,k}\left(f\right)$ est appelé la
linéarisation de $f$ de degré $k$ en $t_{i}$. En particulier,
on a $\mathscr{L}_{i,0}\left(f\right)=f$.
\begin{prop}
\label{prop:Lin=000026Identity}Soit $\left(A,\omega\right)$
une $K$-algèbre vérifiant une identité $f\in\left(Id\left(A\right),\cdot,\star\right)$,
alors pour tout $i\geq1$, les linéarisations de $f$ en $t_{i}$
sont des identités vérifiées par $A$.
\end{prop}

\begin{proof}
Soient $t\notin\mathcal{T}$ et $f\in\left(Id\left(A\right),\cdot,\star\right)$.
Pour tout $i\geq1$, tout $\left(a_{k}\right)_{k\geq1}$, $a$
éléments de $A$ et $\lambda\in K^{*}$, on considère les applications
$\varrho,\varrho_{\lambda}:\mathcal{T}\cup\left\{ t\right\} \rightarrow A$
telles que $\varrho\left(t_{i}\right)=\varrho_{\lambda}\left(t_{i}\right)=a_{i}$,
$\varrho\left(t\right)=a$ et $\varrho_{\lambda}\left(t\right)=\lambda a$,
en remarquant que $\widehat{\varrho_{\lambda}}\left(\mathscr{L}_{i,k}\left(f\right)\right)=\lambda^{k}\widehat{\varrho}\left(\mathscr{L}_{i,k}\left(f\right)\right)$,
de $\widehat{\varrho_{\lambda}}\left(f\left(t_{1},\ldots,t_{i}+t,\ldots\right)\right)=0$
on déduit $\sum_{0\leq k\leq\left|f\right|_{i}}\lambda^{k}\widehat{\varrho}\left(\mathscr{L}_{i,k}\left(f\right)\right)=0$.
Alors compte tenu de l'hypothèse faite sur le corps $K$ (cf.
remarque \ref{rem:Card(K*)}), en donnant à $\lambda$ des valeurs
non nulles $\lambda_{0},\ldots\lambda_{\left|f\right|_{i}}$
deux à deux distinctes, on obtient ainsi un système linéaire
composé de $\left|f\right|_{i}+1$ équations d'inconnues $\widehat{\varrho}\left(\mathscr{L}_{i,k}\left(f\right)\right)$
dont le déterminant n'est pas nul, par conséquent on a $\widehat{\varrho}\left(\mathscr{L}_{i,k}\left(f\right)\right)=0$,
autrement dit les polynômes $\mathscr{L}_{i,0}\left(f\right),\ldots,\mathscr{L}_{i,\left|f\right|_{i}}\left(f\right)$
sont des identités vérifiées par $A$.
\end{proof}
\medskip{}

Pour tout $i\geq1$ et tout $h\in\left(K\left(\mathcal{T}\right),\cdot,\star\right)$,
on introduit les analogues des opérateurs de dérivation (\cite{G-K-69},
\cite{Guzzo-00}, \cite{Zh-Sl-Sh-82}) $\Delta_{i,h}:\left(K\left(\mathcal{T}\right),\cdot,\star\right)\rightarrow\left(K\left(\mathcal{T}\right),\cdot,\star\right)$
qui sont les applications linéaires définies par: 
\begin{align}
\Delta_{i,h}\left(t_{j}\right) & =\begin{cases}
0 & \text{si }t_{j}\neq t_{i},\\
h & \text{si }t_{j}=t_{i},
\end{cases}\nonumber \\
\Delta_{i,h}\left(u\cdot v\right) & =\Delta_{i,h}\left(u\right)\cdot v+u\cdot\Delta_{i,h}\left(v\right),\label{eq:Delta_prod_uv}\\
\Delta_{i,h}\left(u\star v\right) & =\Delta_{i,h}\left(u\right)\star v+u\star\Delta_{i,h}\left(v\right),\quad\left(u,v\in\mathfrak{M}\left(\mathcal{T}\right)\right).\label{eq:Delta_prod_u*v}
\end{align}

\medskip{}
\begin{prop}
\label{prop:Lin_de_f}Etant donné $f\in\left(K\left(\mathcal{T}\right),\cdot,\star\right)$,
$f=\sum_{k\geq1}\alpha_{k}w_{k}$ où $w_{k}\in\left(\mathfrak{M}\left(\mathcal{T}\right),\cdot,\star\right)$,
la linéarisée de $f$ de degré $d$ en $t_{i}$ est obtenue par
\begin{align*}
\mathscr{L}_{i,d}\left(f\right)=\sum_{k\geq1}\alpha_{k}\Delta_{i,t}^{d}\left(w_{k}\right).
\end{align*}
\end{prop}

\begin{proof}
Il est clair que $\mathscr{L}_{i,d}\left(f\right)=\sum_{k\geq1}\alpha_{k}\mathscr{L}_{i,d}\left(w_{k}\right)$,
il suffit donc de montrer que $\mathscr{L}_{i,d}\left(w_{k}\right)=\Delta_{i,t}^{d}\left(w_{k}\right)$,
ce que l'on va faire par récurrence sur le degré du monôme $w_{k}$.
Au degré 1 le résultat découle de la définition de l'application
$\Delta_{i,t}$. Supposons le résultat vrai pour tous les monômes
de degré $\leq n$. Soit $w$ un monôme de degré $n+1$, d'après
la proposition \ref{prop:Dec_ds_Mt*} on a $w=u\cdot v$ ou $w=u\star v$.
\smallskip{}

Dans le cas $w=u\cdot v$ on a $\mathscr{L}_{i,d}\left(w\right)=\sum_{p+q=d}\mathscr{L}_{i,p}\left(u\right)\mathscr{L}_{i,q}\left(v\right)$
et avec l'hypothèse de récurrence $\mathscr{L}_{i,d}\left(w\right)=\sum_{p+q=d}\Delta_{i,t}^{p}\left(u\right)\Delta_{i,t}^{q}\left(v\right)$,
or de la relation (\ref{eq:Delta_prod_uv}) on déduit par récurrence
que pour tout entier $d\geq1$ on a $\Delta_{i,t}^{d}\left(u\cdot v\right)=\sum_{p+q=d}\Delta_{i,h}^{p}\left(u\right)\Delta_{i,h}^{q}\left(v\right)$
par conséquent on a obtenu $\mathscr{L}_{i,d}\left(w\right)=\Delta_{i,t}^{d}\left(u\cdot v\right)=\Delta_{i,t}^{d}\left(w\right)$.\smallskip{}

Pour le cas $w=u\star v$ on a $\mathscr{L}_{i,d}\left(w\right)=\sum_{p+q=d}\mathscr{L}_{i,p}\left(u\right)\star\mathscr{L}_{i,q}\left(v\right)$
on en déduit avec l'hypothèse de récurrence que $\mathscr{L}_{i,d}\left(w\right)=\sum_{p+q=d}\Delta_{i,t}^{p}\left(u\right)\star\Delta_{i,t}^{q}\left(v\right)$,
mais de la relation (\ref{eq:Delta_prod_u*v}) on déduit récursivement
que $\Delta_{i,t}^{d}\left(u\star v\right)=\sum_{p+q=d}\Delta_{i,t}^{p}\left(u\right)\star\Delta_{i,t}^{q}\left(v\right)$
et donc $\mathscr{L}_{i,d}\left(u\star v\right)=\Delta_{i,t}^{d}\left(u\star v\right)$.
\end{proof}
\begin{prop}
Soient $\left(A,\omega\right)$ une $K$-algèbre et $f\in\left(K\left(\mathcal{T}\right),\cdot,\star\right)$
un polynôme homogène tel que $f=\sum_{p\geq1}\alpha_{p}w_{p}+\sum_{q\geq1}\beta_{q}u_{q}\star v_{q}$
où $\alpha_{k},\beta_{l}\in K$ et $w_{p},u_{q},v_{q}\in\mathfrak{M}\left(\mathcal{T}\right)$
avec $\left|w_{p}\right|_{i}=\left|u_{q}\star v_{q}\right|_{i}=\left|f\right|_{i}$
pour tout $i\geq1$. Alors pour $t\notin\mathcal{T}$, pour tout
entier $i\geq1$ et toute application $\varrho:\mathcal{T}\cup\left\{ t\right\} \rightarrow A$
on a: 
\begin{equation}
\widehat{\varrho}\mathscr{L}_{i,1}\left(f\right)=\sum_{p\geq1}\alpha_{p}\widehat{\varrho}\Delta_{i,t}\left(w_{p}\right)+\sum_{q\geq1}\beta_{q}\widehat{\varrho}\Delta_{i,t}\left(u_{q}\star v_{q}\right),\label{eq:rho_L(f)}
\end{equation}
où pour tout $q\geq1$ on a:

\begin{align}
\widehat{\varrho}\Delta_{i,t}\left(u_{q}\star v_{q}\right) & =\Bigl(\prod_{r\geq1,r\neq i}\omega\left(\varrho\left(t_{r}\right)\right)^{\left|u_{q}\right|_{r}}\Bigr)\omega\left(\varrho\left(t_{i}\right)\right)^{\left|u_{q}\right|_{i}-1}\times\label{eq:rho_delta_u*v}\\
 & \hspace{2cm}\Bigl[\left|u_{q}\right|_{i}\omega\left(\varrho\left(t\right)\right)\widehat{\varrho}\left(v_{q}\right)+\omega\left(\varrho\left(t_{i}\right)\right)\widehat{\varrho}\Delta_{i,t}\left(v_{q}\right)\Bigr].\nonumber 
\end{align}
\end{prop}

\begin{proof}
D'après la proposition \ref{prop:Lin_de_f} et par linéarité
de l'application $\widehat{\varrho}$ on a:
\[
\widehat{\rho}\mathscr{L}_{i,1}\left(f\right)=\sum_{p\geq1}\alpha_{p}\widehat{\varrho}\Delta_{i,t}\left(w_{p}\right)+\sum_{q\geq1}\beta_{q}\widehat{\varrho}\Delta_{i,t}\left(u_{q}\star v_{q}\right).
\]
Ensuite par définition des applications $\widehat{\varrho}$
et $\Delta_{i,t}$ on a pour tout entier $q\geq1$: 
\[
\widehat{\varrho}\Delta_{i,t}\left(u_{q}\star v_{q}\right)=\omega\left(\widehat{\varrho}\Delta_{i,t}\left(u_{q}\right)\right)\widehat{\varrho}\left(v_{q}\right)+\omega\left(\widehat{\varrho}\left(u_{q}\right)\right)\widehat{\varrho}\Delta_{i,t}\left(v_{q}\right).
\]
Les applications $\omega$ et $\widehat{\varrho}$ étant des
morphismes on a: 
\[
\omega\left(\widehat{\varrho}\left(u_{q}\right)\right)=\prod_{r\geq1}\omega\left(\varrho\left(t_{r}\right)\right)^{\left|u_{q}\right|_{r}}.
\]
Montrons que pour tout monôme $u$ on a:
\[
\omega\left(\widehat{\varrho}\Delta_{i,t}\left(u\right)\right)=\left|u\right|_{i}\Bigl(\prod_{r\geq1,r\neq i}\omega\left(\varrho\left(t_{r}\right)\right)^{\left|u\right|_{r}}\Bigr)\omega\left(\varrho\left(t_{i}\right)\right)^{\left|u\right|_{i}-1}\omega\left(\varrho\left(t\right)\right).
\]
Ce résultat est vrai pour tout monôme $u$ de degré 1, car $\Delta_{i,t}\left(t_{j}\right)=t$
si $j=i$ et $0$ sinon, donc $\omega\left(\widehat{\varrho}\Delta_{i,t}\left(t_{j}\right)\right)=\left|t_{j}\right|_{i}\omega\left(\varrho\left(t\right)\right)$.
Si on suppose cette propriété vérifiée pour tout monôme de degré
$\leq n$, soit $u$ un monôme de degré $n+1$, on a $u=u_{1}u_{2}$
avec $\left|u_{1}\right|,\left|u_{2}\right|\leq n$, 
\begin{align*}
\omega\left(\widehat{\varrho}\Delta_{i,t}\left(u\right)\right) & =\omega\left(\widehat{\varrho}\Delta_{i,t}\left(u_{1}\right)\right)\omega\left(\widehat{\varrho}\left(u_{2}\right)\right)+\omega\left(\widehat{\varrho}\left(u_{1}\right)\right)\omega\left(\widehat{\varrho}\Delta_{i,t}\left(u_{2}\right)\right)\\
 & =\left|u_{1}\right|_{i}\Bigl(\prod_{r\neq i}\omega\left(\varrho\left(t_{r}\right)\right)^{\left|u_{1}\right|_{r}}\Bigr)\omega\left(\varrho\left(t_{i}\right)\right)^{\left|u_{1}\right|_{i}-1}\omega\left(\varrho\left(t\right)\right)\times\prod_{r\geq1,}\omega\left(\varrho\left(t_{r}\right)\right)^{\left|u_{2}\right|_{r}}+\\
 & \qquad\qquad\prod_{r\geq1}\omega\left(\varrho\left(t_{r}\right)\right)^{\left|u_{1}\right|_{r}}\times\left|u_{2}\right|_{i}\Bigl(\prod_{r\neq i}\omega\left(\varrho\left(t_{r}\right)\right)^{\left|u_{2}\right|_{r}}\Bigr)\omega\left(\varrho\left(t_{i}\right)\right)^{\left|u_{2}\right|_{i}-1}\omega\left(\varrho\left(t\right)\right)\\
 & =\Bigl(\prod_{r\neq i}\omega\left(\varrho\left(t_{r}\right)\right)^{\left|u_{1}\right|_{r}+\left|u_{2}\right|_{r}}\Bigr)\left(\left|u_{1}\right|_{i}+\left|u_{2}\right|_{i}\right)\omega\left(\varrho\left(t_{i}\right)\right)^{\left|u_{1}\right|_{i}+\left|u_{2}\right|_{i}-1}\omega\left(\varrho\left(t\right)\right)
\end{align*}
et le résultat découle de $\left|u_{1}\right|_{i}+\left|u_{2}\right|_{i}=\left|u\right|_{i}$.
\end{proof}
\medskip{}

Un élément $e$ d'une $K$-algèbre $A$ est un idempotent si
$e\neq0$ et $e^{2}=e$. Si l'algèbre $A$ est pondérée par $\omega$,
de $e^{2}=e$ on déduit que $\omega\left(e\right)\left(\omega\left(e\right)-1\right)=0$
donc $\omega\left(e\right)=0$ ou $1$. 

Soit $\left(A,\omega\right)$ une $K$-algèbre admettant un idempotent
$e$ tel que $\omega\left(e\right)=1$, en prenant dans (\ref{eq:rho_delta_u*v})
les applications $\varrho_{\left(e,y\right)}$ définies pour
tout $y\in\ker\omega$ par
\begin{align*}
\varrho_{\left(e,y\right)} & :\mathscr{\mathcal{T}}\cup\left\{ t\right\} \rightarrow A,\\
\varrho_{\left(e,y\right)}\left(t_{i}\right) & =e\;\left(i\geq1\right),\\
\varrho_{\left(e,y\right)}\left(t\right) & =y.
\end{align*}
on obtient $\widehat{\varrho}_{\left(e,y\right)}\left(u_{q}\star v_{q}\right)=\widehat{\varrho}_{\left(e,y\right)}\Delta_{i,t}\left(v_{q}\right)$
et avec ceci (\ref{eq:rho_L(f)}) devient: 
\begin{equation}
\widehat{\varrho}_{\left(e,y\right)}\mathscr{L}_{i,1}\left(f\right)=\sum_{p\geq1}\alpha_{p}\widehat{\varrho}_{\left(e,y\right)}\Delta_{i,t}\left(w_{p}\right)+\sum_{q\geq1}\beta_{q}\widehat{\varrho}_{\left(e,y\right)}\Delta_{i,t}\left(v_{q}\right).\label{eq:rho(e,y)_yds_kerw}
\end{equation}

On en déduit
\begin{cor}
Soient $\left(A,\omega\right)$ une $K$-algèbre admettant un
idempotent $e$ de poids 1 et $f\in K\left(\mathcal{T}\right)$.
Pour toute homogénéisation $f^{*}$ de $f$ on a:
\[
\widehat{\varrho}_{\left(e,y\right)}\mathscr{L}_{i,1}\left(f^{*}\right)=\widehat{\varrho}_{\left(e,y\right)}\Delta_{i,t}\left(f\right),\quad\left(\forall y\in\ker\omega\right).
\]
\end{cor}

\begin{proof}
Soit $i$ un entier, on écrit $f\in K\left(\mathcal{T}\right)$
sous la forme $f=\sum_{p\geq1}\alpha_{p}w_{p}+\sum_{q\geq1}\beta_{q}v_{q}$
avec $w_{p},v_{q}\in\mathfrak{M}\left(\mathcal{T}\right)$ tels
que $\left|w_{p}\right|_{i}=\left|f\right|_{i}$ pour tout $p\geq1$
et $\left|v_{q}\right|_{i}<\left|f\right|_{i}$ pour tout $q\geq1$.
Soit $f^{*}=\sum_{p\geq1}\alpha_{p}w_{p}+\sum_{q\geq1}\beta_{q}u_{q}\star v_{q}$
une homogénéisation de $f$, où pour tout $q\geq1$ tel que $\beta_{q}\neq0$
on a $u_{q}\in\mathfrak{M}\left(\mathscr{\mathcal{T}}\right)$
tel que $\left|u_{q}\right|_{i}=\left|f\right|_{i}-\left|v_{q}\right|_{i}$.
D'après (\ref{eq:rho(e,y)_yds_kerw}) on a
\[
\widehat{\varrho}_{\left(e,y\right)}\Delta_{i,t}\left(f\right)=\sum_{p\geq1}\alpha_{p}\widehat{\varrho}_{\left(e,y\right)}\Delta_{i,t}\left(w_{p}\right)+\sum_{q\geq1}\beta_{q}\widehat{\varrho}_{\left(e,y\right)}\Delta_{i,t}\left(v_{q}\right)=\widehat{\varrho}_{\left(e,y\right)}\mathscr{L}_{i,1}\left(f^{*}\right),
\]
quod erat demonstrandum.
\end{proof}
Dans la suite on note $K\left\langle t\right\rangle $ la sous-algèbre
de $K\left(\left\{ t\right\} \right)$ engendrée par l'ensemble
$\left\{ t^{n};n\geq1\right\} $ où pour tout entier $n\geq1$
on a: $t^{n+1}=tt^{n}=t^{n}t$ avec $t^{1}=t$. Pour chaque $i\geq1$
on définit l'application linéaire $\partial_{i}$ par
\begin{align}
\partial_{i} & :K\left(\mathcal{T}\right)\rightarrow K\left\langle t\right\rangle \nonumber \\
\partial_{i}\left(t_{j}\right) & =\begin{cases}
1 & \text{si }j=i\\
0 & \text{si }j\neq i,
\end{cases}\label{eq:Def_de_di}\\
\partial_{i}\left(uv\right) & =t\left(\partial_{i}\left(u\right)+\partial_{i}\left(v\right)\right),\quad\left(\forall u,v\in\mathfrak{M}\left(\mathscr{\mathcal{T}}\right)\right).\nonumber 
\end{align}

Pour simplifier les notations on écrira $\partial_{i}f$ au lieu
de $\partial_{i}\left(f\right)$ pour $f\in K\left(\mathcal{T}\right)$.
\begin{example}
\label{exa:Exemple1}Soit $w=\left(\left(\left(t_{1}t_{2}\right)t_{2}\right)t_{3}^{2}\right)\left(\left(t_{1}^{2}t_{3}\right)t_{1}\right)$
on a: 
\begin{align*}
\partial_{1}w & =t\left(\partial_{1}\left(\left(\left(t_{1}t_{2}\right)t_{2}\right)t_{3}^{2}\right)+\partial_{1}\left(\left(t_{1}^{2}t_{3}\right)t_{1}\right)\right)=t\left(t^{3}+t\left(\partial_{1}\left(t_{1}^{2}t_{3}\right)+1\right)\right)\\
 & =t\left(t^{3}+t^{2}\partial_{1}\left(t_{1}^{2}\right)+t\right)=3t^{4}+t^{2}.\\
\partial_{2}w & =t\left(t\partial_{2}\left(\left(t_{1}t_{2}\right)t_{2}\right)\right)=t\left(t^{2}\left(\partial_{2}\left(t_{1}t_{2}\right)+1\right)\right)=t^{4}+t^{3}.\\
\partial_{3}w & =t\left(t\left(\partial_{3}\left(\left(t_{1}t_{2}\right)t_{2}\right)+\partial_{3}\left(t_{3}^{2}\right)\right)+t\partial_{3}\left(t_{1}^{2}t_{3}\right)\right)=t^{2}\left(2t+t\right)=3t^{3}.
\end{align*}
\end{example}

Une façon plus commode de calculer les polynômes $\partial_{i}f$
utilise la représentation des éléments de $\mathfrak{M}\left(\mathcal{T}\right)$
par des arbres binaires enracinées à feuilles étiquetées. 

\medskip{}

Un arbre est un graphe $T=\left(T^{0},T^{1}\right)$ non orienté,
connexe, sans cycle, où $T^{0}\neq\textrm{Ø}$ (resp. $T^{1}$)
est l'ensemble des sommets (resp. des arêtes). 

Un arbre $T$ est dit enraciné si un sommet, noté $\rho_{_{T}}$
et appelé la racine, est distingué. 

Deux sommets $s_{1},s_{2}\in T^{0}$ sont incidents si $s_{1}$
et $s_{2}$ sont les sommets d'une même arête. La valence d'un
sommet $s$ est le nombre de sommets incidents à $s$. Un arbre
$T$ est binaire si la valence de $\rho_{_{T}}$ vaut $0$ ou
$2$ et si la valence de $s\in T^{0}$, $s\neq\rho_{_{T}}$ vaut
$1$ ou $3$. 

Les sommets univalents d'un arbre binaire enraciné $T$ sont
appelés feuilles, on note $L\left(T\right)$ l'ensemble des feuilles
de $T$. 

Un arbre binaire enraciné $T$ est dit $\mathcal{T}$-étiqueté
s'il existe une application $\Lambda:L\left(T\right)\rightarrow\mathscr{T}$,
on note $\left(T,\Lambda\right)$ un tel arbre. 

Deux arbres binaires enracinés $T_{1}$ et $T_{2}$ sont isomorphes
s'il existe un isomorphisme de graphes $\varphi:T_{1}\rightarrow T_{2}$
tel que $\varphi\left(\rho_{_{T_{1}}}\right)=\rho_{_{T_{2}}}$.
L'isomorphisme d'arbres enracinés $\varphi:T_{1}\rightarrow T_{2}$
conserve les feuilles: $\varphi\left(L\left(T_{1}\right)\right)=L\left(T_{2}\right)$. 

Deux arbres binaires enracinés $\mathscr{T}$-étiquetés $\left(T_{1},\Lambda_{1}\right)$
et $\left(T_{2},\Lambda_{2}\right)$ sont isomorphes s'il existe
un isomorphisme $\varphi:T_{1}\rightarrow T_{2}$ d'arbres enracinés
tel que $\Lambda_{2}\circ\varphi_{\left|L\left(T_{1}\right)\right.}=\Lambda_{1}$. 

On note $\mathscr{T}_{\mathscr{\mathcal{T}}}$ l'ensemble des
classes d'isomorphisme des arbres binaires enracinés $\mathscr{\mathcal{T}}$-étiquetés,
on munit $\mathscr{T}_{\mathcal{T}}$ de la loi de greffage:
soient $T_{1},T_{2}\in\mathscr{T}_{\mathcal{T}}$, on associe
à $\left(T_{1},T_{2}\right)$ l'arbre $T_{1}\cdot T_{2}$ tel
que le graphe de $T_{1}\cdot T_{2}$ privé de sa racine $\rho_{_{T_{1}\cdot T_{2}}}$
et des deux arêtes adjacentes à $\rho_{_{T_{1}\cdot T_{2}}}$
a deux composantes connexes $T_{1}$ et $T_{2}$. Muni de la
loi de greffage $\mathscr{T}_{\mathscr{\mathcal{T}}}$ est un
magma isomorphe au magma non commutatif $\text{Mag}\left(\mathscr{\mathcal{T}}\right)$,
par cet isomorphisme $\Psi:\text{Mag}\left(\mathscr{\mathcal{T}}\right)\rightarrow\mathscr{T}_{\mathcal{T}}$,
le degré de $w\in\text{Mag}\left(\mathscr{\mathcal{T}}\right)$
en $t_{i}\in\mathscr{\mathcal{T}}$ est égal au nombre de feuilles
étiquetées $t_{i}$ de l'arbre $\Psi\left(w\right)$. 

\medskip{}

Muni de ces notions sur les arbres binaires enracinés et étiquetés
on a le résultat suivant qui fournit un moyen pratique et rapide
pour calculer les polynômes $\partial_{i}w$.

Soit $\left(T,\Lambda\right)$ un arbre binaire enraciné $\mathcal{T}$-étiqueté,
la hauteur d'un sommet $s\in T^{0}$, notée $\hslash\left(s\right)$,
est le nombre minimum d'arêtes joignant $s$ à la racine $\rho_{_{T}}$
(cf. \cite{Ether-60}). 
\begin{prop}
\label{prop:Diw=000026arbre}Pour tout $w\in\mathfrak{M}\left(\mathscr{\mathcal{T}}\right)$
et tout $i\geq1$, on a: 
\begin{align}
\partial_{i}w\left(t\right) & =\quad\smashoperator{\sum_{\begin{subarray}{c}
s\in\Lambda_{w}^{-1}\left(t_{i}\right)\end{subarray}}}\;t^{\hslash\left(s\right)},\label{eq:diw=000026arbre}
\end{align}
où $\Lambda_{w}^{-1}\left(t_{i}\right)$ est l'ensemble des feuilles
étiquetées $t_{i}$ dans l'arbre $\Psi\left(w\right)$, autrement
dit, $\Lambda_{w}^{-1}\left(t_{i}\right)=\left\{ s\in L\left(\Psi\left(w\right)\right);\Lambda\left(s\right)=t_{i}\right\} $. 
\end{prop}

\begin{proof}
Montrons le résultat par récurrence sur le degré de $w$. La
propriété est vraie si $w$ est de degré 1, en effet si $w=t_{i}$
sur l'arbre $\Psi\left(w\right)$ la feuille étiquetée $t_{i}$
est à la hauteur $0$ car elle est confondue avec la racine donc
$\partial_{i}w=1=t^{0}$, si $w=t_{j}$ avec $j\neq i$ alors
on a $\Lambda^{-1}\left(t_{i}\right)=\textrm{Ø}$ et par convention
la somme est nulle. Supposons la propriété (\ref{eq:diw=000026arbre})
vraie pour tout monôme de degré $n$, soit $w$ de degré $n+1$,
on a $w=uv$ avec $u,v\in\mathfrak{M}\left(\mathscr{\mathcal{T}}\right)$
de degré au moins 1. D'après (\ref{eq:Def_de_di}) on a 
\[
\partial_{i}w=t\left(\partial_{i}u+\partial_{i}v\right)\;=\;\smashoperator{\sum_{\begin{subarray}{c}
s\in\Lambda_{u}^{-1}\left(t_{i}\right)\end{subarray}}}\;t^{\hslash\left(s\right)+1}+\smashoperator{\sum_{\begin{subarray}{c}
s\in\Lambda_{v}^{-1}\left(t_{i}\right)\end{subarray}}}\;t^{\hslash\left(s\right)+1}\quad=\quad\smashoperator{\sum_{\begin{subarray}{c}
s\in\Lambda_{u}^{-1}\left(t_{i}\right)\cup\Lambda_{v}^{-1}\left(t_{i}\right)\end{subarray}}}\;t^{\hslash\left(s\right)+1},
\]
où $\Lambda_{u}^{-1}\left(t_{i}\right)$ et $\Lambda_{v}^{-1}\left(t_{i}\right)$
dénotent respectivement l'ensemble des feuilles des arbres $\Psi\left(u\right)$
et $\Psi\left(v\right)$ étiquetés $t_{i}$. Or, l'arbre $\Psi\left(w\right)$
étant le résultat du greffage des arbres $\Psi\left(u\right)$
et $\Psi\left(v\right)$, il en résulte que l'ensemble des feuilles
de l'arbre $\Psi\left(w\right)$ étiquetées $t_{i}$ est la réunion
des ensembles $\Lambda_{u}^{-1}\left(t_{i}\right)$ et $\Lambda_{v}^{-1}\left(t_{i}\right)$,
et par définition du greffage les hauteurs des feuilles de $\Lambda_{u}^{-1}\left(t_{i}\right)$
et $\Lambda_{v}^{-1}\left(t_{i}\right)$ dans l'arbre $\Psi\left(w\right)$
sont augmentées d'une unité par rapport à leurs valeurs dans
les arbres $\Psi\left(u\right)$ et $\Psi\left(v\right)$, on
déduit de tout ceci que $\sum_{\begin{subarray}{c}
s\in\Lambda_{u}^{-1}\left(t_{i}\right)\cup\Lambda_{v}^{-1}\left(t_{i}\right)\end{subarray}}t^{\hslash\left(s\right)+1}=\sum_{\begin{subarray}{c}
s\in\Lambda_{w}^{-1}\left(t_{i}\right)\end{subarray}}t^{\hslash\left(s\right)}$.
\end{proof}
\begin{example}
Pour illustrer ce résultat, on reprend l'exemple \ref{exa:Exemple1}
avec le monôme $w=\left(\left(\left(t_{1}t_{2}\right)t_{2}\right)t_{3}^{2}\right)\left(\left(t_{1}^{2}t_{3}\right)t_{1}\right)$.
L'arbre binaire enraciné étiqueté associé à $w$ est donné ci-dessous
(on a figuré seulement les indices des étiquettes). 

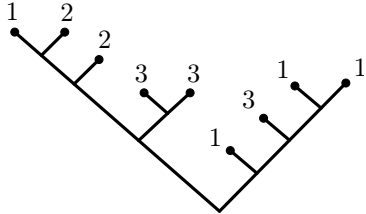
\hspace{35mm}\begin{tikzpicture}[line cap=round,line join=round,x=1.0cm,y=1.0cm] \clip(-1.011943716791299,-0.08808417134596927) rectangle (4.0,3.0); 
\draw [line width=1.2pt] (-0.64,2.42)-- (2.0820383479266718,0.03688264448086892); 
\draw [line width=1.2pt] (2.0820383479266718,0.03688264448086892)-- (3.761591286833347,1.7424751638357057); 
\draw [line width=1.2pt] (3.4333831495460143,1.4091785282958569)-- (3.084562195181044,1.7034157931634575); 
\draw [line width=1.2pt] (3.012395996234126,0.9816644423744829)-- (2.6679289080103956,1.2737627157687277); 
\draw [line width=1.2pt] (2.5797627958595797,0.5423237505212626)-- (2.225256040391582,0.8441096383739978); 
\draw [line width=1.2pt] (1.0075265630716985,0.977606967908625)-- (1.3919894660502856,1.338861666889141); 
\draw [line width=1.2pt] (1.3919894660502856,1.338861666889141)-- (1.0795145006722997,1.6122772615948784); 
\draw [line width=1.2pt] (1.3919894660502856,1.338861666889141)-- (1.691444641204189,1.6122772615948784); 
\draw [line width=1.2pt] (0.14906778245492602,1.7291790464738943)-- (0.48060415036449294,2.054950129213691); 
\draw [line width=1.2pt] (-0.2860214600455177,2.1100953277570893)-- (0.02491149252159656,2.419504255488007); 
\draw (-0.8834096276564463,2.925326140593348) node[anchor=north west] {1}; 
\draw (-0.15504978922561466,2.9110445751339205) node[anchor=north west] {2}; 
\draw (0.34480500185436785,2.554005438648219) node[anchor=north west] {2}; 
\draw (3.7438175811982495,2.18268473670309) node[anchor=north west] {1}; 
\draw (2.715544868119428,2.154121605784234) node[anchor=north west] {1}; 
\draw (0.8303782274749223,2.11127690940595) node[anchor=north west] {3}; 
\draw (1.5301749349868978,2.11127690940595) node[anchor=north west] {3}; 
\draw (1.8158062441754592,1.254382981840267) node[anchor=north west] {1}; 
\draw (2.2585347734177295,1.7542377729202485) node[anchor=north west] {3}; \begin{scriptsize} \draw [fill=black] (-0.64,2.42) circle (1.5pt); 
\draw [fill=black] (3.761591286833347,1.7424751638357057) circle (1.5pt); 
\draw [fill=black] (3.084562195181044,1.7034157931634575) circle (1.5pt); 
\draw [fill=black] (2.6679289080103956,1.2737627157687277) circle (1.5pt); 
\draw [fill=black] (2.225256040391582,0.8441096383739978) circle (1.5pt); 
\draw [fill=black] (1.0795145006722997,1.6122772615948784) circle (1.5pt); 
\draw [fill=black] (1.691444641204189,1.6122772615948784) circle (1.5pt); 
\draw [fill=black] (0.48060415036449294,2.054950129213691) circle (1.5pt); 
\draw [fill=black] (0.02491149252159656,2.419504255488007) circle (1.5pt); 
\end{scriptsize} 
\end{tikzpicture}

Cet arbre a 4 feuilles étiquetées $t_{1}$ dont trois de hauteur
4 et une de hauteur 2 donc d'après (\ref{eq:diw=000026arbre})
on a $\partial_{1}w=3t^{4}+t^{2}$. Il a 2 feuilles avec l'étiquette
$t_{2}$, l'une de hauteur 4, l'autre de hauteur 3 donc $\partial_{2}w=t^{4}+t^{3}$.
Enfin l'étiquette $t_{3}$ est portée par trois feuilles toutes
situées à la hauteur 3 par conséquent $\partial_{3}w=3t^{3}$.
\end{example}

\begin{cor}
\label{cor:Prop_de_Di}Pour tout $w\in\mathfrak{M}\left(\mathscr{\mathcal{T}}\right)$
et tout $i\geq1$, concernant $\partial_{i}w$ on a:

a) Les coefficients du polynôme $\partial_{i}w$ sont des entiers
naturels.

b) Le degré du polynôme $\partial_{i}w$ est égal à la hauteur
maximale des feuilles étiquetées $t_{i}$ dans l'arbre $\Psi\left(w\right)$,
autrement dit, 
\[
\deg\left(\partial_{i}w\right)=\max\left\{ \hslash\left(s\right);s\in L\left(\Psi\left(w\right)\right),\Lambda\left(s\right)=t_{i}\right\} .
\]

c) $\left|\partial_{i}w\right|\leq\left|w\right|-1.$

d) $\partial_{i}w\left(1\right)=\left|w\right|_{i}$.

e) Si $w=uv$ avec $u,v\in\mathfrak{M}\left(\mathscr{\mathcal{T}}\right)$
tels que $\left|u\right|_{i},\left|v\right|_{i}\geq1$, la valuation
de $\partial_{i}w$ est: 
\[
\text{val}\left(\partial_{i}w\right)=\begin{cases}
\min\left\{ \text{val}\left(\partial_{i}u\right),\text{val}\left(\partial_{i}v\right)\right\} +1 & \text{si }\left|u\right|_{i}\left|v\right|_{i}\neq\text{0},\\
\max\left\{ \text{val}\left(\partial_{i}u\right),\text{val}\left(\partial_{i}v\right)\right\} +1 & \text{si }\left|u\right|_{i}\left|v\right|_{i}=\text{0}.
\end{cases}
\]
\end{cor}

\begin{proof}
\emph{a}) et \emph{b}) sont des conséquences immédiates de (\ref{eq:diw=000026arbre}).

\emph{c}) Par récurrence sur le degré de $w$. Si $\left|w\right|=1$
le résultat est immédiat car on a $\partial_{i}w=0,1$. Si le
résultat est vrai pour tout monôme de degré $\leq n$, soit $w\in\mathfrak{M}\left(\mathscr{\mathcal{T}}\right)$
de degré $n+1$, il existe $u,v\in\mathfrak{M}\left(\mathscr{\mathcal{T}}\right)$
tels que $w=uv$, on a $\partial_{i}w=t\left(\partial_{i}u+\partial_{i}v\right)$,
compte tenu de \emph{a}) on a $\left|\partial_{i}w\right|=\max\left\{ \left|\partial_{i}u\right|,\left|\partial_{i}v\right|\right\} +1$,
on en déduit avec l'hypothèse de récurrence que $\left|\partial_{i}w\right|\leq\max\left\{ \left|u\right|,\left|v\right|\right\} $,
comme $u\neq w$ et $v\neq w$ on a $\left|u\right|<\left|w\right|$
et $\left|v\right|<\left|w\right|$ donc $\left|\partial_{i}w\right|<\left|w\right|$.

\emph{d}) D'après (\ref{eq:diw=000026arbre}) on a $\partial_{i}w\left(1\right)=\text{card}\left(\Lambda_{w}^{-1}\left(t_{i}\right)\right)$
et par l'isomorphisme de magmas $\Psi:\text{Mag}\left(\mathscr{\mathcal{T}}\right)\rightarrow\mathscr{T}_{\mathscr{\mathcal{T}}}$,
le nombre de feuilles de l'arbre $\Psi\left(w\right)$ étiquetées
$t_{i}$ est égal au degré de $w$ en $t_{i}$.

\emph{e}) C'est une conséquence immédiate de (\ref{eq:diw=000026arbre})
et de la loi de greffage des arbres binaires enracinés.
\end{proof}
Soit $e\in A$, on note $L_{e}$ l'endomorphisme d'une $K$-algèbre
$A$ défini par $L_{e}:x\mapsto ex$.
\begin{prop}
Soient $\left(A,\omega\right)$ une $K$-algèbre admettant un
idempotent $e\in H_{\omega}$ et $f\in K\left(\mathscr{\mathcal{T}}\right)$.
Pour tout entier $i\geq1$ on a 
\[
\left(\partial_{i}f\right)\left(L_{e}\right)\left(y\right)=\widehat{\varrho}_{\left(e,y\right)}\Delta_{i,t}\left(f\right)\quad\left(\forall y\in\ker\omega\right).
\]

De plus, si $f\in K\left(\mathscr{\mathcal{T}}\right)$ est une
identité vérifiée par $A$ on a: $\left(\partial_{i}f\right)\left(L_{e}\right)\left(y\right)=0$
pour tout $y\in\ker\omega$. 
\end{prop}

\begin{proof}
Par linéarité des applications $\partial_{i}f$ et $\widehat{\varrho}_{\left(e,y\right)}\Delta_{i,t}$
il suffit de montrer que l'on a $\widehat{\varrho}_{\left(e,y\right)}\Delta_{i,t}\left(w\right)=\left(\partial_{i}w\right)\left(L_{e}\right)\left(y\right)$
pour tout $w\in\mathfrak{M}\left(\mathscr{\mathcal{T}}\right)$.
Montrons cela par récurrence sur le degré de $w$. Si $w$ est
de degré 1 on a $\widehat{\varrho}_{\left(e,y\right)}\Delta_{i,t}\left(t_{j}\right)=0$,
$\left(\partial_{i}t_{j}\right)\left(L_{e}\right)\left(y\right)=0$
si $j\neq i$, et $\widehat{\varrho}_{\left(e,y\right)}\Delta_{i,t}\left(t_{i}\right)=y=\left(\partial_{i}t_{i}\right)\left(L_{e}\right)\left(y\right)$.
Supposons le résultat vrai pour tous les monômes de degré $\leq n$,
soit $w\in\mathfrak{M}\left(\mathscr{\mathcal{T}}\right)$ de
degré $n+1$, le monôme $\omega$ s'écrit $w=uv$ avec $u,v\in\mathfrak{M}\left(\mathscr{\mathcal{T}}\right)$
de degrés $\leq n$, et d'après (\ref{eq:Delta_prod_uv}) et
l'hypothèse de récurrence on a 
\begin{align*}
\widehat{\varrho}_{\left(e,y\right)}\Delta_{i,t}\left(w\right) & =\widehat{\varrho}_{\left(e,y\right)}\left(\Delta_{i,t}\left(u\right)v\right)+\widehat{\varrho}_{\left(e,y\right)}\left(u\Delta_{i,t}\left(v\right)\right)=L_{e}\partial_{i}\left(u\right)\left(L_{e}\right)y+L_{e}\partial_{i}\left(v\right)\left(L_{e}\right)y,
\end{align*}
car $\widehat{\varrho}_{\left(e,y\right)}\left(u\right)=\widehat{\varrho}_{\left(e,y\right)}\left(v\right)=e$.
Enfin on a: 
\[
L_{e}\partial_{i}\left(u\right)\left(L_{e}\right)y+L_{e}\partial_{i}\left(v\right)\left(L_{e}\right)y=L_{e}\left(\partial_{i}u+\partial_{i}v\right)\left(L_{e}\right)\left(y\right)=\partial_{i}\left(uv\right)\left(L_{e}\right)\left(y\right).
\]

Si $f$ est une identité vérifiée par $A$, d'après la proposition
\ref{prop:Lin=000026Identity} on a $\widehat{\varrho}_{\left(e,y\right)}\Delta_{i,t}\left(f\right)=0$
pour tout $y\in\ker\omega$.
\end{proof}
Il résulte de cette proposition que pour tout $f\in Id\left(A\right)$,
les polynômes $\partial_{i}f$ sont annulateurs de l'opérateur
$L_{e}$ quel que soit l'idempotent $e$ de poids 1 de $A$. 
\begin{defn}
Soit $f\in K\left(\mathscr{\mathcal{T}}\right)$, pour $i\geq1$,
le polynôme $\partial_{i}f$ est appelé le polynôme de Peirce
en $t_{i}$ de $f$.

Le polynôme $f$ est dit Peirce-évanescent si $f\neq0$ et si
tous ses polynômes de Peirce $\partial_{i}f$, ($i\geq1$) sont
nuls.

Le polynôme $f$ est une identité Peirce-évanescente (en abrégé,
une identité évanescente) si $f\left(\mathds{1}\right)=0$ et
si $f$ est Peirce-évanescent.

Une $K$-algèbre $\left(A,\omega\right)$ admettant un idempotent
$e\in H_{\omega}$ et vérifiant une identité $f\in K\left(\mathscr{\mathcal{T}}\right)$
est dite Peirce-évanescente pour $f$ si le polynôme $f$ est
une identité évanescente.
\end{defn}

\begin{example}
Soit $\left(A,\omega\right)$ une algèbre vérifiant l'identité
\[
f\left(x,y\right)=x^{2}\left(xy^{2}\right)-x\left(xy^{2}\right)-x^{2}y+xy.
\]
Par rapport à un idempotent de $A$ on trouve $\partial_{x}f\left(t\right)=3t^{2}-\left(t^{2}+t\right)-2t^{2}+t$
et $\partial_{y}f\left(t\right)=2t^{3}-2t^{3}-t+t$, donc l'algèbre
$A$ est évanescente.
\end{example}

On note $Ev\left(\mathscr{\mathcal{T}}\right)$ le sous-ensemble
de $K\left(\mathscr{\mathcal{T}}\right)$ dont les éléments sont
des polynômes évanescents, alors pour toute $K$-algèbre $\left(A,\omega\right)$
admettant un idempotent $e$ de poids 1, l'ensemble $Ev\left(A\right)=Ev\left(\mathscr{\mathcal{T}}\right)\cap Id\left(A\right)$
désigne l'ensemble des identités évanescentes relativement à
$e$ vérifiées par $A$. 
\begin{prop}
L'ensemble $Ev\left(\mathscr{\mathcal{T}}\right)$ est un idéal
de $K\left(\mathscr{\mathcal{T}}\right)$.
\end{prop}

\begin{proof}
Il est immédiat que $Ev\left(\mathscr{\mathcal{T}}\right)$ est
un sous-espace de $K\left(\mathscr{\mathcal{T}}\right)$. Montrons
que pour tout $f,g\in K\left(\mathscr{\mathcal{T}}\right)$ et
pour tout $i\geq1$ on a 
\begin{equation}
\partial_{i}\left(fg\right)=t\left(f\left(\mathds{1}\right)\partial_{i}g+g\left(\mathds{1}\right)\partial_{i}f\right).\label{eq:di(fg)}
\end{equation}

Soient $f=\sum_{p\geq1}\alpha_{p}u_{p}$ et $g=\sum_{q\geq1}\beta_{q}v_{q}$
où $\alpha_{p},\beta_{q}\in K$ et $u_{p},v_{q}\in\mathfrak{M}\left(\mathscr{\mathcal{T}}\right)$,
on a 
\[
\partial_{i}\left(fg\right)=\sum_{p,q\geq1}\alpha_{p}\beta_{q}\partial_{i}\left(u_{p}v_{q}\right)=t\biggl(\Bigl(\sum_{p\geq1}\alpha_{p}\Bigr)\partial_{i}g+\Bigl(\sum_{q\geq1}\beta_{q}\Bigr)\partial_{i}u\biggr),
\]
or on a $\sum_{p\geq1}\alpha_{p}=f\left(\mathds{1}\right)$ et
$\sum_{q\geq1}\beta_{q}=g\left(\mathds{1}\right)$.

En particulier, si on prend $f\in Ev\left(\mathscr{\mathcal{T}}\right)$
et $g\in K\left(\mathscr{\mathcal{T}}\right)$ on a $\partial_{i}f=0$
et d'après la proposition \ref{prop:H(A)} on a $f\left(\mathds{1}\right)=0$
d'où $\partial_{i}\left(fg\right)=0$.
\end{proof}
En revanche l'idéal $Ev\left(A\right)$ n'est pas un $T$-idéal
de $\left(K\left(\mathscr{\mathcal{T}}\right),\cdot,\star\right)$
ni un $T$-idéal stochastique (cf. remarque \ref{rem:T-ideal-stoch})
comme le montre l'exemple suivant.
\begin{example}
Partant de l'identité évanescente caractérisant les algèbres
de rétrocroisement $f\left(x\right)=x^{2}x^{2}-2x^{3}+x^{2}$,
on considère l'identité $g\left(x\right)=f\left(\frac{1}{2}\left(x^{2}+x\right)\right)$. 

On a $g\left(x\right)=\frac{1}{16}\Bigl[\left(x^{2}+x\right)^{2}\left(x^{2}+x\right)^{2}-4\left(x^{2}+x\right)^{3}+4\left(x^{2}+x\right)^{2}\Bigr]$. 

De $\partial_{x}\bigl(\left(x^{2}+x\right)^{2}\bigr)=4t\left(2t+1\right)$
on déduit:
\begin{align*}
\partial_{x}\bigl(\left(x^{2}+x\right)^{2}\left(x^{2}+x\right)^{2}\bigr) & =2t\partial_{x}\bigl(\left(x^{2}+x\right)^{2}\bigr)=8t^{2}\left(2t+1\right)\\
\partial_{x}\bigl(\left(x^{2}+x\right)^{3}\bigr) & =t\Bigl(\partial_{x}\bigl(\left(x^{2}+x\right)^{2}\bigr)+\partial_{x}\bigl(x^{2}+x\bigr)\Bigr)=t\left(2t+1\right)\left(4t+1\right)
\end{align*}
 finalement on a $\partial_{x}g\left(x\right)=\frac{1}{16}\left(2t-\left(4t+1\right)+4\right)4t\left(2t+1\right)=\frac{1}{4}t\left(2t+1\right)\left(3-2t\right)\neq0$.
\end{example}

\medskip{}

La relation (\ref{eq:di(fg)}) donne une méthode simple pour
construire des identités évanescentes.
\begin{prop}
Soit $\left(A,\omega\right)$ une algèbre admettant un idempotent
$e\in H_{\omega}$ et vérifiant une identité de la forme $fg$
où $f,g\in K\left(\mathscr{\mathcal{T}}\right)$. Si on a $f\left(\mathds{1}\right)=g\left(\mathds{1}\right)=0$
alors l'identité $fg$ est évanescente.
\end{prop}

Etant donnée $\left(A,\omega\right)$ une $K$-algèbre vérifiant
une identité $f$, le spectre de Peirce est l'ensemble des racines
des polynômes de Peirce $\partial_{i}f$, relativement à un idempotent
$e\in H_{\omega}$ ce sont les valeurs propres de l'opérateur
$L_{e}$ qui interviennent dans la décomposition de Peirce de
la $K$-algèbre $\left(A,\omega\right)$ vérifiant l'identité
$f$. Il est évident que si l'algèbre $\left(A,\omega\right)$
vérifie une identité évanescente, en l'abscence de polynômes
de Peirce, le spectre de l'opérateur $L_{e}$ est indéterminé.
Dans ce qui suit on précise cela en montrant que le spectre de
$L_{e}$ peut être n'importe quelle partie de $K$ contenant
1, pour cela on utilise les algèbres de mutation.\medskip{}

Une $K$-algèbre de mutation $\left(A,M,\omega\right)$ est définie
par la donnée d'un $K$-espace vectoriel $A$, d'une application
linéaire $M:A\rightarrow A$, d'une forme linéaire $\omega:A\rightarrow K$
telle que $\omega\neq0$, $\omega\circ M=\omega$ et du produit
$xy=\frac{1}{2}\left(\omega\left(y\right)M\left(x\right)+\omega\left(x\right)M\left(y\right)\right)$
où $x,y\in A$. Il résulte de la définition que $\omega\left(xy\right)=\omega\left(x\right)\omega\left(y\right)$
donc $\omega$ est une pondération.
\begin{example}
Les algèbres de mutation vérifient une multitude d'identités.
La construction de ces identités s'appuie sur la propriété que
pour une algèbre de mutation $\left(A,M,\omega\right)$ on a
$\left(\ker\omega\right)^{2}=0$, alors en prenant $x,y,x',y'$
dans $H_{\omega}$ tels que $x-y\neq0$ et $x'-y'\neq0$ on a
$\left(x-y\right)\left(x'-y'\right)=0$. Avec ce procédé on construit
ad libitum des identités vérifiées par toutes les algèbres de
mutation, par exemple $\left(t_{1}-t_{2}\right)^{2}$, $\left(t_{1}^{2}-t_{2}\right)^{2}$,
$\left(t_{1}^{2}-t_{1}\right)\left(t_{2}^{2}-t_{2}\right)$,
$\left(t_{1}+t_{2}-t_{3}-t_{4}\right)\left(t_{1}-t_{2}+t_{3}-t_{4}\right)$,
$t_{1}^{2}t_{2}^{2}-\left(t_{1}t_{2}\right)^{2}$ et cetera … 
\end{example}

Les algèbres de mutation vérifient toutes les identités évanescentes.
\begin{prop}
Soit $\left(A,M,\omega\right)$ une algèbre de mutation, quel
que soit $f\in Ev\left(\mathscr{\mathcal{T}}\right)$ l'algèbre
$A$ vérifie l'identité $f$.
\end{prop}

\begin{proof}
Soit $f\in Ev\left(\mathscr{\mathcal{T}}\right)$ une identité
évanescente, $f=\sum_{k\geq1}\alpha_{k}w_{k}$ où $\alpha_{k}\in K$
et $w_{k}\in\mathfrak{M}\left(\mathscr{\mathcal{T}}\right)$.
Soit $\left(A,M,\omega\right)$ une algèbre de mutation, pour
toute famille $\left(x_{n}\right)_{n\geq1}$ d'éléments de $H_{\omega}$
on définit l'application $\varrho:\mathscr{\mathcal{T}}\rightarrow H_{\omega}$
par $\varrho\left(t_{i}\right)=x_{i}$, ($i\geq1$) et les morphismes
d'algèbres $\varphi_{i}:K\left\langle t\right\rangle \rightarrow A$,
définis par: 
\[
\varphi_{i}\left(t^{n}\right)=\frac{1}{2^{n}}M^{n}\left(x_{i}\right).
\]
Montrons par récurrence sur le degré que pour tout $w\in\mathfrak{M}\left(\mathscr{\mathcal{T}}\right)$
on a:
\[
\widehat{\varrho}\left(w\right)=\sum_{i\geq1}\varphi_{i}\left(\partial_{i}w\right).
\]
Le résultat est immédiat si $w$ est de degré 1. Supposons le
résultat vrai pour les monômes de degré $\leq n$. Soit $w$
un monôme de degré $n+1$, il existe $u,v\in\mathfrak{M}\left(\mathscr{\mathcal{T}}\right)$
de degrés $\leq n$ tels que $w=uv$, alors $\widehat{\varrho}\left(w\right)=\widehat{\varrho}\left(uv\right)=\widehat{\varrho}\left(u\right)\widehat{\varrho}\left(v\right)$,
en utilisant la structure d'algèbre de mutation de $A$ on obtient
$\widehat{\varrho}\left(u\right)\widehat{\varrho}\left(v\right)=\frac{1}{2}M\left(\widehat{\varrho}\left(u\right)\right)+\frac{1}{2}M\left(\widehat{\varrho}\left(v\right)\right)$,
avec l'hypothèse de récurrence ceci devient $\frac{1}{2}M\left(\widehat{\varrho}\left(u\right)\right)+\frac{1}{2}M\left(\widehat{\varrho}\left(v\right)\right)=\frac{1}{2}M\left(\sum_{i\geq1}\varphi_{i}\left(\partial_{i}u\right)+\varphi_{i}\left(\partial_{i}v\right)\right)$,
or on a $\varphi_{i}\left(t\partial_{i}u\right)=\frac{1}{2}M\varphi_{i}\left(\partial_{i}u\right)$,
par conséquent $\widehat{\varrho}\left(u\right)\widehat{\varrho}\left(v\right)=\sum_{i\geq1}\varphi_{i}\left(t\left(\partial_{i}u+\partial_{i}v\right)\right)=\sum_{i\geq1}\varphi_{i}\left(\partial_{i}uv\right)$
d'où le résultat. On en déduit que
\[
\widehat{\varrho}\left(f\right)=\sum_{k\geq1}\alpha_{k}\widehat{\varrho}\left(w_{k}\right)=\sum_{k\geq1}\alpha_{k}\sum_{i\geq1}\varphi_{i}\left(\partial_{i}w_{k}\right)=\sum_{i\geq1}\varphi_{i}\Bigl(\sum_{k\geq1}\alpha_{k}\partial_{i}w_{k}\Bigr)=\sum_{i\ge1}\varphi_{i}\left(\partial_{i}f\right),
\]
et comme $f$ est évanescente on a $\partial_{i}f=0$ pour tout
$i\geq1$ par conséquent $\widehat{\varrho}\left(f\right)=0$,
et d'après la proposition \ref{prop:Bar-Id_pond1} on a montré
que l'algèbre $\left(A,M,\omega\right)$ vérifie l'identité $f$.
\end{proof}
\begin{prop}
Pour toute partie $P$ de $K$ contenant $\left\{ 1\right\} $,
il existe une algèbre de mutation $\left(A,M,\omega\right)$
admettant un idempotent $e$ dont le spectre de l'opérateur $L_{e}$
est $P$.
\end{prop}

\begin{proof}
Considérons le $K$-espace $A$ de base $\left(e_{n}\right)_{n\in\mathbb{N}}$,
muni de la structure d'algèbre de mutation par $M:A\rightarrow A$
telle que $M\left(e_{0}\right)=e_{0}$, $M\left(e_{i}\right)=2e_{i+1}$
pour tout $i\geq1$ et $\omega:A\rightarrow K$ telle que $\omega\left(e_{0}\right)=1$,
$\omega\left(e_{i}\right)=0$ pour tout $i\geq1$, alors on a
$e_{0}^{2}=e_{0}$ et $e_{0}e_{i}=e_{i+1}$ par conséquent l'élément
$e_{0}$ est un idempotent de $A$ et le spectre de $L_{e_{0}}$
est $P=\left\{ 1\right\} $.

Soient $I$ un ensemble non vide et $P=\left\{ 1\right\} \cup\left\{ \lambda_{i};i\in I\right\} $
une partie de $K$. On considère le $K$-espace vectoriel $A$
de base $\left\{ e\right\} \cup\left\{ e_{i};i\in I\right\} $
muni de la structure d'algèbre de mutation par les applications
$M:A\rightarrow A$ définie par $M\left(e\right)=e$, $M\left(e_{i}\right)=2\lambda_{i}e_{i}$
et $\omega:A\rightarrow K$ telle que $\omega\left(e\right)=1$,
$\omega\left(e_{i}\right)=0$. Cette algèbre $\left(A,M,\omega\right)$
admet $e$ pour élément idempotent et pour tout $i\in I$ on
a $ee_{i}=\frac{1}{2}M\left(e_{i}\right)=\lambda_{i}e_{i}$ par
conséquent pour cette algèbre, le spectre de $L_{e}$ est $P$. 
\end{proof}

\section{Identités évanescentes de type $\left[n\right]$, $\left[n,1\right]$,
$\left[n,2\right]$, $\left[n,1,1\right]$.}

\subsection{Méthodes d'obtention des générateurs des polynômes évanescents
homogènes et non homogènes.}

\textcompwordmark{}

On recherche des générateurs des identités évanescentes sous
la forme de polynômes non homogènes définis comme suit.
\begin{defn}
Un polynôme non homogène $f\in K\left(t_{1},\ldots,t_{n}\right)$
est appelé une train polynôme de degré $\left(d_{1},\ldots,d_{n}\right)$
si $f=g-\sum_{i=1}^{r}h_{i}$, avec $g,h_{1},\ldots,h_{r}\in K\left(t_{1},\ldots,t_{n}\right)$
vérifiant les conditions suivantes:

a) $f\left(\mathds{1}\right)=0$,

b) le polynôme $g$ est homogène de type $\left[d_{1},\ldots,d_{n}\right]$,

c) pour tout $1\leq i\leq r$, le polynôme $h_{i}$ est homogène
de type $\left[\delta_{i},d_{2},\ldots,d_{n}\right]$, 

d) on a $0\leq\delta_{1}<\ldots<\delta_{r}<d_{1}$.
\end{defn}

\begin{rem}
Pour $n=1$, si les polynômes $g_{i}$ et $h_{j}$ sont pris
dans l'ensemble $\left\{ x^{k};k\geq1\right\} $, on retrouve
la définition des train polynômes aux puissances principales
introduits par Etherington \cite{Ether-39b} et pour $g_{i}$
et $h_{j}$ dans $\left\{ x^{\left[k\right]};k\geq1\right\} $
où $x^{\left[n+1\right]}=x^{\left[n\right]}x^{\left[n\right]}$,
$x^{\left[1\right]}=x$ on obtient les train polynômes aux puissances
plénières étudiées dans \cite{Gut-00}. 
\end{rem}

Dans les cas étudiés dans la suite on utilise la méthode suivante
pour obtenir les générateurs des polynômes évanescents sous la
forme de train polynômes.\medskip{}

Pour un $n$-uplet $\left(d_{1},\ldots,d_{n}\right)$ donné et
pour $w\in\mathfrak{M}\left(t_{1},\ldots,t_{n}\right)$ de type
$\left[d_{1},\ldots,d_{n}\right]$, on cherche un polynôme $P_{w}\in K\left(t_{1},\ldots,t_{n}\right)$
tel que $w-P_{w}$ soit un train polynôme de degré $\left(d_{1},\ldots,d_{n}\right)$
vérifiant $\partial_{i}\left(w-P_{w}\right)=0$ pour tout $1\leq i\leq n$.
Pour cela on choisit un ensemble $\mathscr{F}=\left\{ w_{1,k},\ldots,w_{m,k};k\geq0\right\} $
où 
\begin{itemize}
\item pour chaque $1\leq j\leq m$ et tout $k\geq0$ on a $w_{j,k}\in\mathfrak{M}\left(t_{1},\ldots,t_{n}\right)$
et $w_{j,k}$ est de type $\left[k,d_{2},\ldots,d_{n}\right]$
ou $\left[k+1,d_{2},\ldots,d_{n}\right]$, 
\item pour chaque $1\leq i\leq n$ il existe $1\leq j\leq m$ tel que
la suite d'entiers $\left(\bigl|\partial_{i}\left(w_{j,k}\right)\bigr|\right)_{k\geq0}$
est strictement croissante et l'ensemble des entiers $\left\{ \bigl|\partial_{i}\left(w_{j,k}\right)\bigr|;k\geq0\right\} =\mathbb{N}\text{ ou }\mathbb{N}^{*}$. 
\end{itemize}
Alors pour $w\in\mathfrak{M}\left(t_{1},\ldots,t_{n}\right)$,
$w\notin\mathscr{F}$, on pose $P_{w}=\sum_{j=1}^{m}\left(\sum_{k=0}^{\delta_{j}}\alpha_{j,k}w_{j,k}\right)$
où $\delta_{j}=\bigl|\partial_{i}w\bigr|$, pour chaque $1\leq i\leq n$
on a $\partial_{i}\left(w-P_{w}\right)\in K\left\langle t\right\rangle $,
par conséquent la recherche du polynôme $P_{w}$ vérifiant $\partial_{i}\left(w-P_{w}\right)=0$
($1\leq i\leq n$) est équivalente à la résolution d'un système
d'équations linéaires d'inconnues $\left(\alpha_{j,k}\right)_{\begin{subarray}{c}
1\leq j\leq m\\
0\leq k\leq\delta_{j}
\end{subarray}}$.

\medskip{}

Pour ce qui concerne les générateurs des polynômes évanescents
homogènes de type $\left[d_{1},\ldots,d_{n}\right]$. On note
$N$ le cardinal de $\mathfrak{M}\left(t_{1},\ldots,t_{n}\right)_{\left[d_{1},\ldots,d_{n}\right]}$
et $\left(w_{k}\right)_{1\leq k\leq N}$ les éléments de cet
ensemble. Soit $f=\sum_{k=1}^{N}\alpha_{k}w_{k}$, on cherche
$\left(\alpha_{k}\right)_{1\leq k\leq N}$ tels que $\partial_{i}f=0$
pour tout $1\leq i\leq n$ et $\sum_{k=1}^{N}\alpha_{k}=0$,
on a $\partial_{i}f\in K\left\langle t\right\rangle $ et d'après
le corollaire \ref{cor:Prop_de_Di}, $\left|\partial_{i}f\right|\leq\sum_{j=1}^{n}d_{j}-1$,
par conséquent les conditions $\sum_{k=1}^{N}\alpha_{k}\left(\partial_{i}w_{k}\right)=0$
et $\sum_{k=1}^{N}\alpha_{k}=0$ se traduisent par au plus $k\left(\sum_{j=1}^{n}d_{j}-1\right)+1$
équations linéaires d'inconnues $\left(\alpha_{k}\right)_{1\leq k\leq N}$.

\subsection{Identités évanescentes train de degré $\left(n\right)$ et homogènes
de type $\left[n\right]$.}

\textcompwordmark{}\medskip{}

Dans cette section pour simplifier les notations on écrira $\frak{M}\left(x\right)$
au lieu de $\frak{M}\left(\left\{ x\right\} \right)$ et $K\left(x\right)$
au lieu de $K\left(\left\{ x\right\} \right)$.\medskip{}

Pour tout $n\geq1$ on note $\frak{M}\left(x\right)_{\left[n\right]}$
le sous-ensemble de $\frak{M}\left(x\right)$ formé par les monômes
de type $\left[n\right]$. Les nombres $W_{\left[n\right]}=\mbox{card }\frak{M}\left(x\right)_{\left[n\right]}$
sont les nombres de Wedderburn-Etherington, ils vérifient les
relations de récurrence suivantes dépendant de la parité de $n$.
En partant de $W_{\left[1\right]}=1$, on a :
\begin{eqnarray*}
\begin{aligned}W_{\left[2p\right]} & =\sum_{i=1}^{p-1}W_{\left[i\right]}W_{\left[2p-i\right]}+\dbinom{W_{\left[p\right]}+1}{2},\end{aligned}
 &  & \begin{aligned}W_{\left[2p+1\right]} & ={\displaystyle \sum_{i=1}^{p}}W_{\left[i\right]}W_{\left[2p+1-i\right]},\quad\left(p\geq1\right).\end{aligned}
\end{eqnarray*}

Les premières valeurs de $W_{\left[n\right]}$ sont:
\begin{center}
\begin{tabular}{c|c|c|c|c|c|c|c|c|c|c|c}
$n$ & 0 & 1 & 2 & 3 & 4 & 5 & 6 & 7 & 8 & 9 & 10\tabularnewline
\hline 
$W_{\left[n\right]}$ & 0 & 1 & 1 & 1 & 2 & 3 & 6 & 11 & 23 & 46 & 98\tabularnewline
\end{tabular}
\par\end{center}

\medskip{}

\subsubsection{Train identités évanescentes de degré $\left(n\right)$.}

\textcompwordmark{}

\medskip{}

Dans ce qui suit on note $\mathbb{Q}\left\langle x\right\rangle $
le $\mathbb{Q}$-espace vectoriel engendré par l'ensemble $\left\{ x^{n};n\geq1\right\} $.
\begin{prop}
\label{prop:Evanes_type_=00005Bn=00005D}Il n'existe pas de train
identité évanescente de degré $\left(2\right)$ et $\left(3\right)$.

Pour tout $n\geq4$ et tout $w\in\frak{M}\left(x\right)_{\left[n\right]}$
vérifiant $w\neq x^{n}$, il existe un unique polynôme $P_{w}\in\mathbb{Q}\left\langle x\right\rangle $
de degré $<n$ tel que le polynôme $w-P_{w}$ soit une train
identité évanescente.
\end{prop}

\begin{proof}
Soit $f\left(x\right)=\alpha x^{2}+\beta x$, on a $\partial_{x}f\left(t\right)=2\alpha t+\beta$
donc $\partial_{x}f=0$ si $\alpha=\beta=0$. Soit $f\left(x\right)=\alpha x^{3}+\beta x^{2}+\gamma x$,
on a $\partial_{x}f\left(t\right)=2\alpha t^{2}+\left(\alpha+2\beta\right)t+\gamma$
et par suite $\partial_{x}f=0$ seulement si $f=0$. 

Soit $w\in\frak{M}\left(x\right)_{\left[n\right]}$ tel que $w\neq x^{n}$,
d'après le résultat c) du corollaire \ref{cor:Prop_de_Di}on
a $\left|\partial_{x}w\right|\leq n-1$ et pour tout $k\geq3$
on a 
\begin{equation}
\partial_{x}\left(x^{k}\right)=2t^{k-1}+\sum_{i=1}^{k-2}t^{i}.\label{eq:Dx(x^k)}
\end{equation}
Soit $p=\left|\partial_{x}w\right|$, on a $\partial_{x}w\left(t\right)=\sum_{k=0}^{p}\alpha_{k}t^{k}$
et on cherche $P_{w}\left(x\right)=\sum_{k=1}^{p+1}\beta_{k}x^{k}$
tel que $\partial_{x}\left(w-P_{w}\right)=0$ et $P_{w}\left(1\right)=w\left(1\right)=1$.
Un simple calcul donne $\partial_{x}P_{w}\left(t\right)=2\beta_{p+1}t^{p}+\sum_{k=1}^{p-1}\left(2\beta_{k+1}+\sum_{i=k+2}^{p+1}\beta_{i}\right)t^{k}+2\beta_{1}$
et on a $\partial_{x}w=\partial_{x}P_{w}$ si et seulement si
$\left(\beta_{i}\right)_{1\leq i\leq p+1}$ est solution du système
linéaire: 
\[
2\beta_{p+1}=\alpha_{p},\;2\beta_{k+1}+\sum_{i=k+2}^{p+1}\beta_{i}=\alpha_{k},(1\leq k\leq p-1),\;2\beta_{1}=\alpha_{0},\;\sum_{k=1}^{p+1}\beta_{k}=1,
\]
qui est équivalent au système linéaire triangulaire: $2\beta_{1}=\alpha_{0}$,
$\beta_{k}-\sum_{i=1}^{k-1}\beta_{i}=\alpha_{k}-1$ ($2\leq k\leq p+1$)
comme d'après le corollaire \ref{cor:Prop_de_Di} on a $\alpha_{k}\in\mathbb{N}$
pour tout $0\leq k\leq p$, la solution $\left(\beta_{i}\right)_{1\leq i\leq p+1}$
de ce système vérifie $\beta_{i}\in\mathbb{Q}$ pour tout $1\leq i\leq p+1$.
\end{proof}
On en déduit immédiatement le corollaire qui suit.
\begin{cor}
Pour tout $n\geq4$, l'espace vectoriel des train identités évanescentes
de degré $\left(n\right)$ est de dimension $W_{\left[n\right]}-1$.
\end{cor}

La démonstration de la proposition \ref{prop:Evanes_type_=00005Bn=00005D}
donne une méthode basée sur la résolution de systèmes linéaires
triangulaires pour obtenir des polynômes évanescents, malheureusement
elle est difficile à appliquer pour les grandes valeurs de $n$,
heureusement le résultat suivant donne un algorithme plus facile
à mettre en oeuvre. 
\begin{thm}
\label{thm:Evanesc_type=00005Bn=00005D_id=0000E9al}Pour tout
entier $p,q\geq1$ on pose:
\[
E_{p,q}\left(x\right)=x^{p}x^{q}-x^{p+1}-x^{q+1}+x^{2}.
\]
 Soient $\mathscr{E}$ l'idéal engendré par la famille $\left\{ E_{p,q};p,q\geq1\right\} $
et $\pi:K\left(x\right)\rightarrow{}^{K\left(x\right)}\!/_{\mathscr{E}}$
la surjection canonique. Alors pour tout $n\geq4$ et $w\in\frak{M}\left(x\right)_{\left[n\right]}$,
$w\neq x^{n}$ on a $\pi\left(w\right)=P_{w}$ et pour tout $f\in K\left(x\right)$
de degré $\geq4$, le polynôme $f-\pi\left(f\right)$ est une
train identité évanescente.
\end{thm}

\begin{proof}
En utilisant la relation (\ref{eq:Dx(x^k)}) on montre par un
simple calcul que les polynômes $E_{p,q}$ sont évanescents.
Montrons par récurrence sur le degré $n\geq4$ que $w-\pi\left(w\right)$
est évanescent et que $\pi\left(w\right)\in\mathbb{Z}\left\langle x\right\rangle $
pour tout $w\in\frak{M}\left(x\right)_{\left[n\right]}$, $w\neq x^{n}$.
On a $\pi\left(x^{2}x^{2}\right)=2x^{3}-x^{2}$ et on sait que
le polynôme $x^{2}x^{2}-\left(2x^{3}-x^{2}\right)$ est évanescent.
Si le résultat est vrai pour tout $u\in\frak{M}\left(x\right)_{\left[k\right]}$,
$u\neq x^{k}$ où $4\leq k\leq n$, soit $w\in\frak{M}\left(x\right)_{\left[n+1\right]}$,
$w\neq x^{n+1}$, il existe $u\in\frak{M}\left(x\right)_{\left[p\right]}$
et $v\in\frak{M}\left(x\right)_{\left[q\right]}$ tels que $w=uv$
où $1\leq p\leq q$ et $p+q=n+1$. Si $u=x^{p}$ ou $v=x^{q}$
on a $\pi\left(u\right)=u$ ou $\pi\left(v\right)=v$ alors on
a $\partial_{x}\left(w-\pi\left(w\right)\right)=\partial_{x}\left(uv-\pi\left(u\right)\pi\left(v\right)\right)=t\left(\partial_{x}\left(u-\pi\left(u\right)\right)+\partial_{x}\left(v-\pi\left(v\right)\right)\right)=\text{0}$.
De plus si $\pi\left(u\right),\pi\left(v\right)\in\mathbb{Z}\left\langle x\right\rangle $
alors $\pi\left(w\right)=\pi\left(u\right)\pi\left(v\right)\in\mathbb{Z}\left\langle x\right\rangle $.

Pour tout $w\in\frak{M}\left(x\right)_{\left[n\right]}$ on a
$\pi\left(w\right)\in\mathbb{Z}\left\langle x\right\rangle $
et $w-\pi\left(w\right)$ est évanescent, alors d'après la proposition
\ref{eq:Dx(x^k)}, par unicité du polynôme $P_{w}$ on a $\pi\left(w\right)=P_{w}$.
Enfin pour tout $f\in K\left(x\right)$, $f=\sum_{k\geq1}\alpha_{k}w_{k}$
on a $f-\pi\left(f\right)=\sum_{k\geq1}\alpha_{k}\left(w_{k}-\pi\left(w_{k}\right)\right)$
donc $f-\pi\left(f\right)\in\mathscr{E}$ autrement dit, le polynôme
$f-\pi\left(f\right)$ est évanescent. 
\end{proof}
\begin{rem}
Ce théorème permet de préciser une propriété énoncée à la proposition
\ref{prop:Evanes_type_=00005Bn=00005D}: pour tout $w\in\frak{M}\left(x\right)_{\left[n\right]}$
tel que $w\neq x^{n}$, le polynôme $P_{w}$ vérifie $P_{w}\in\mathbb{Z}\left\langle x\right\rangle $.
\end{rem}

Le théorème \ref{thm:Evanesc_type=00005Bn=00005D_id=0000E9al}
donne un moyen très pratique et très rapide pour obtenir des
polynômes évanescents.
\begin{example}
Soit $w=\left(\left(x^{3}x^{3}\right)x^{2}\right)\left(\left(x^{2}x^{4}\right)x^{3}\right)$,
on a:
\begin{align*}
\pi\left(w\right) & =\left(x^{2}\left(2x^{4}-x^{2}\right)\right)\left(x^{3}\left(x^{3}+x^{5}-x^{2}\right)\right)=\left(2x^{5}-x^{2}\right)\left(x^{6}+2x^{4}-x^{3}-x^{2}\right)\\
 & =x^{7}+2x^{6}+2x^{5}-x^{4}-2x^{3}-x^{2}
\end{align*}
on obtient ainsi la train identité évanescente de degré $\left(17\right)$:
\[
\left(\left(x^{3}x^{3}\right)x^{2}\right)\left(\left(x^{2}x^{4}\right)x^{3}\right)-x^{7}-2x^{6}-2x^{5}+x^{4}+2x^{3}+x.
\]
\end{example}

En utilisant cette méthode on obtient les train identités évanescentes:%
\begin{longtable}{lll}
– de degré (4):$\medskip$ &  & \tabularnewline
$x^{2}x^{2}-2x^{3}+x^{2}.$ &  & \tabularnewline
 &  & \tabularnewline
– de degré (5):$\medskip$ &  & \tabularnewline
$\left(x^{2}x^{2}\right)x-2x^{4}+x^{3};$ &  & $\;x^{3}x^{2}-x^{4}-x^{3}+x^{2}.$\tabularnewline
 &  & \tabularnewline
– de degré (6):$\medskip$ &  & \tabularnewline
$\left(x^{2}x^{2}\right)x^{2}-2x^{4}+x^{2};$ &  & $\left(x^{3}x^{2}\right)x-x^{5}-x^{4}+x^{3};$$\medskip$\tabularnewline
$\;x^{3}x^{3}-2x^{4}+x^{2};$ &  & $\;x^{4}x^{2}-x^{5}-x^{3}+x^{2}.$$\medskip$\tabularnewline
$\left(\left(x^{2}x^{2}\right)x\right)x-2x^{5}+x^{4};$ &  & \tabularnewline
 &  & \tabularnewline
– de degré (7):$\medskip$ &  & \tabularnewline
$\left(x^{2}x^{2}\right)x^{3}-3x^{4}+x^{3}+x^{2};$ &  & $\;x^{5}x^{2}-x^{6}-x^{3}+x^{2};$$\medskip$\tabularnewline
$\left(x^{3}x^{3}\right)x-2x^{5}+x^{3};$ &  & $\left(x^{4}x^{2}\right)x-x^{6}-x^{4}+x^{3};$$\medskip$\tabularnewline
$\;x^{4}x^{3}-x^{5}-x^{4}+x^{2};$ &  & $\left(\left(x^{3}x^{2}\right)x\right)x-x^{6}-x^{5}+x^{4};$$\medskip$\tabularnewline
$\left(x^{3}x^{2}\right)x^{2}-x^{5}-x^{4}+x^{2};$ &  & $\left(\left(\left(x^{2}x^{2}\right)x\right)x\right)x-2x^{6}+x^{5}.$$\medskip$\tabularnewline
$\left(\left(x^{2}x^{2}\right)x\right)x^{2}-2x^{5}+x^{4}-x^{3}+x^{2};$ &  & \tabularnewline
 &  & \tabularnewline
– de degré (8):$\medskip$ &  & \tabularnewline
$\;x^{4}x^{4}-2x^{5}+x^{2};$ &  & $\left(\left(\left(x^{2}x^{2}\right)x^{2}\right)x\right)x-2x^{6}+x^{4};$$\medskip$\tabularnewline
$\left(x^{3}x^{3}\right)x^{2}-2x^{5}+x^{2};$ &  & $\left(x^{4}x^{3}\right)x-x^{6}-x^{5}+x^{3};$$\medskip$\tabularnewline
$\left(\left(x^{2}x^{2}\right)x^{2}\right)x^{2}-2x^{5}+x^{2};$ &  & $\left(\left(x^{3}x^{2}\right)x\right)x^{2}-x^{6}-x^{5}+x^{4}-x^{3}+x^{2};$$\medskip$\tabularnewline
$\left(\left(x^{2}x^{2}\right)x\right)x^{3}-2x^{5}+x^{2};$ &  & $\left(\left(\left(x^{2}x^{2}\right)x\right)x\right)x^{2}-2x^{6}+x^{5}-x^{3}+x^{2};$$\medskip$\tabularnewline
$\left(x^{3}x^{2}\right)x^{3}-x^{5}-2x^{4}+x^{3}+x^{2};$ &  & $\left(\left(\left(x^{2}x^{2}\right)x\right)x^{2}\right)x-2x^{6}+x^{5}-x^{4}+x^{3};$$\medskip$\tabularnewline
$\left(\left(x^{2}x^{2}\right)x^{3}\right)x-3x^{5}+x^{4}+x^{3};$ &  & $\left(x^{5}x^{2}\right)x-x^{7}-x^{4}+x^{3};$$\medskip$\tabularnewline
$\left(x^{2}x^{2}\right)x^{4}-x^{5}-2x^{4}+x^{3}+x^{2};$ &  & $\left(\left(x^{3}x^{2}\right)x^{2}\right)x-x^{6}-x^{5}+x^{3};$$\medskip$\tabularnewline
$\left(x^{2}x^{2}\right)\left(x^{2}x^{2}\right)-4x^{4}+2x^{3}+x^{2};$ &  & $\left(\left(\left(x^{2}x^{2}\right)x\right)x^{2}\right)x-2x^{6}+x^{5}-x^{4}+x^{3};$$\medskip$\tabularnewline
$\;x^{5}x^{3}-x^{6}-x^{4}+x^{2};$ &  & $\;x^{6}x^{2}-x^{7}-x^{3}+x^{2};$$\medskip$\tabularnewline
$\left(\left(x^{3}x^{3}\right)x\right)x-2x^{6}+x^{4};$ &  & $\left(\left(x^{4}x^{2}\right)x\right)x-x^{7}-x^{5}+x^{4};$$\medskip$\tabularnewline
$\left(x^{4}x^{2}\right)x^{2}-x^{6}-x^{4}+x^{2};$ &  & $\left(\left(\left(x^{3}x^{2}\right)x\right)x\right)x-x^{7}-x^{6}+x^{5}.$\tabularnewline
\end{longtable}

\subsubsection{Identités homogènes évanescentes de type $\left[n\right]$.}
\begin{prop}
Il n'existe pas d'identité homogène évanescente de type $\left[n\right]$
pour $n\leq5$. Pour $n\geq6$, l'espace des identités homogènes
évanescentes de degré $n$ est engendré par au moins $W_{\left[n\right]}-n+2$
polynômes homogènes évanescents.
\end{prop}

\begin{proof}
Le résultat est immédiat pour les types $\left[2\right]$ et
$\left[3\right]$ où il n'y a pas de polynômes évanescents, pour
le type $\left[4\right]$ il n'y a qu'un unique polynôme évanescent
qui n'est pas homogène. Soit $n\geq5$, pour simplifier les notations
on pose $N=W_{\left[n\right]}$, on note $w_{1},\ldots,w_{N}$
les éléments de $\frak{M}\left(x\right)_{\left[n\right]}$. Dire
qu'il existe un polynôme homogène évanescent de type $\left[n\right]$
est équivalent à dire qu'il existe $\alpha_{1},\ldots,\alpha_{N}$
dans $K$ non tous nuls tels que $\sum_{k=1}^{N}\alpha_{k}w_{k}=0$,
$\sum_{k=1}^{N}\alpha_{k}=0$ et $\sum_{k=1}^{N}\alpha_{k}\partial_{x}\left(w_{k}\right)=0$.
Or du résultat c) du corollaire \ref{cor:Prop_de_Di} on a $\left|\partial_{x}w\right|\leq n-1$
pour tout $w\in\frak{M}_{\left[n\right]}\left(x\right)$, donc
pour tout $1\leq k\leq N$ on a $\partial_{x}\left(w_{k}\right)=\sum_{i=1}^{n-1}\lambda_{k,i}t^{i}$
avec $\lambda_{k,i}=0$ pour $i>\left|\partial_{x}w_{k}\right|$
et $\sum_{k=1}^{N}\alpha_{k}\partial_{x}\left(w_{k}\right)=\sum_{i=1}^{n-1}\left(\sum_{k=1}^{N}\lambda_{k,i}\alpha_{k}\right)t^{i}$
soit à résoudre le système linéaire $S:\sum_{k=1}^{N}\lambda_{k,i}\alpha_{k}=0,(1\leq i\leq n-1)$,
à $n-1$ équations d'inconnues $\alpha_{1},\ldots,\alpha_{N}$.
En tenant compte de $\sum_{k=1}^{N}\alpha_{k}=0$, le système
$\left(S\right)$ est de rang $\leq n-2$ et ses solutions forment
un espace vectoriel de dimension $\geq N-\left(n-2\right)$.
Si $n=5$, on a $\partial_{x}\left(x^{5}\right)=2t^{4}+t^{3}+t^{2}+t$,
$\partial_{x}\left(\left(x^{2}x^{2}\right)x\right)=4t^{3}+t$,
$\partial_{x}\left(x^{3}x^{2}\right)=2t^{3}+3t^{2}$, le système
$\left(S\right)$ est de rang 3 et il a $\left(0,0,0\right)$
pour unique solution.
\end{proof}
En utilisant la méthode utilisée dans la démonstration on obtient
les générateurs des identités homogènes évanescentes: 
\begin{align*}
 & \text{-- de type \ensuremath{\left[6\right]}}\\
 & \;x^{3}x^{3}+\left(\left(x^{2}x^{2}\right)x\right)x-x^{4}x^{2}-\left(x^{3}x^{2}\right)x;\\
 & \left(x^{2}x^{2}\right)x^{2}+\left(\left(x^{2}x^{2}\right)x\right)x-x^{4}x^{2}-\left(x^{3}x^{2}\right)x.\\
\\
 & \text{-- de type \ensuremath{\left[7\right]}}\\
 & \;x^{4}x^{3}-\left(x^{3}x^{2}\right)x^{2};\\
 & \left(x^{2}x^{2}\right)x^{3}+\left(\left(x^{2}x^{2}\right)x\right)x^{2}-2x^{4}x^{3};\\
 & \;x^{4}x^{3}+\left(x^{4}x^{2}\right)x-\left(\left(x^{3}x^{2}\right)x\right)x-\left(x^{2}x^{2}\right)x^{3};\\
 & \;x^{5}x^{2}+\left(x^{2}x^{2}\right)x^{3}+\left(\left(x^{3}x^{2}\right)x\right)x-\left(\left(\left(x^{2}x^{2}\right)x\right)x\right)x-2x^{4}x^{3};\\
 & \;x^{4}x^{3}+\left(x^{3}x^{3}\right)t+\left(\left(\left(x^{2}x^{2}\right)x\right)x\right)x-2\left(\left(x^{3}x^{2}\right)x\right)x-\left(x^{2}x^{2}\right)x^{3};\\
 & \;x^{4}x^{3}+\left(\left(x^{2}x^{2}\right)x^{2}\right)t+\left(\left(\left(x^{2}x^{2}\right)x\right)x\right)x-2\left(\left(x^{3}x^{2}\right)x\right)x-\left(x^{2}x^{2}\right)x^{3}.\\
\\
\text{} & \text{-- de type \ensuremath{\left[8\right]}}\\
 & \;x^{5}x^{3}-\left(x^{4}x^{2}\right)x^{2};\\
 & \left(x^{2}x^{2}\right)x^{4}-\left(x^{3}x^{2}\right)x^{3};\\
 & \left(\left(x^{3}x^{2}\right)x^{2}\right)x-\left(x^{4}x^{3}\right)x;\\
 & \left(x^{4}x^{2}\right)x^{2}+\left(\left(x^{4}x^{2}\right)x\right)x-\left(x^{4}x^{3}\right)x-x^{6}x^{2};\\
 & \left(x^{3}x^{2}\right)x^{3}+\left(\left(\left(x^{2}x^{2}\right)x\right)x\right)x^{2}-2\left(x^{4}x^{2}\right)x^{2};\\
 & \left(\left(x^{2}x^{2}\right)x^{3}\right)x+\left(\left(\left(x^{2}x^{2}\right)x\right)x^{2}\right)x-2\left(x^{4}x^{3}\right)x;\\
 & \;x^{4}x^{4}+\left(x^{4}x^{3}\right)x-\left(\left(x^{2}x^{2}\right)x^{3}\right)x-\left(x^{4}x^{2}\right)x^{2};\\
 & \left(x^{3}x^{2}\right)x^{3}+\left(\left(x^{3}x^{3}\right)x\right)x-\left(x^{4}x^{3}\right)x-\left(x^{4}x^{2}\right)x^{2};\\
 & \left(x^{3}x^{3}\right)x^{2}+\left(x^{4}x^{3}\right)x-\left(\left(x^{2}x^{2}\right)x^{3}\right)x-\left(x^{4}x^{2}\right)x^{2};\\
 & \left(x^{4}x^{3}\right)x+\left(\left(x^{2}x^{2}\right)x\right)x^{3}-\left(\left(x^{2}x^{2}\right)x^{3}\right)x-\left(x^{4}x^{2}\right)x^{2};\\
 & \left(x^{3}x^{2}\right)x^{3}+\left(\left(\left(x^{2}x^{2}\right)x^{2}\right)x\right)x-\left(x^{4}x^{3}\right)x-\left(x^{4}x^{2}\right)x^{2};\\
 & \left(\left(x^{2}x^{2}\right)x^{2}\right)x^{2}+\left(x^{4}x^{3}\right)x-\left(\left(x^{2}x^{2}\right)x^{3}\right)x-\left(x^{4}x^{2}\right)x^{2};\\
 & \left(x^{4}x^{2}\right)x^{2}+\left(\left(x^{2}x^{2}\right)x^{3}\right)x+\left(\left(\left(x^{3}x^{2}\right)x\right)x\right)x-2\left(x^{4}x^{3}\right)x-x^{6}x^{2};\\
 & \left(x^{3}x^{2}\right)x^{3}+\left(x^{4}x^{3}\right)x+\left(\left(x^{3}x^{2}\right)x\right)x^{2}-\left(\left(x^{2}x^{2}\right)x^{3}\right)x-2\left(x^{4}x^{2}\right)x^{2};\\
 & \left(x^{2}x^{2}\right)\left(x^{2}x^{2}\right)+\left(\left(x^{2}x^{2}\right)x^{3}\right)x+\left(x^{4}x^{2}\right)x^{2}-\left(x^{4}x^{3}\right)x-2\left(x^{3}x^{2}\right)x^{3};\\
 & \left(x^{5}x^{2}\right)x+\left(\left(x^{2}x^{2}\right)x^{3}\right)x+2\left(x^{4}x^{2}\right)x^{2}-2\left(x^{4}x^{3}\right)x-x^{6}x^{2}-\left(x^{3}x^{2}\right)x^{3};\\
 & \left(\left(x^{2}x^{2}\right)x^{3}\right)x+\left(\left(\left(\left(x^{2}x^{2}\right)x\right)x\right)x\right)x+3\left(x^{4}x^{2}\right)x^{2}-2\left(x^{4}x^{3}\right)x-2x^{6}x^{2}-\left(x^{3}x^{2}\right)x^{3}.
\end{align*}

\medskip{}

\subsection{Identités évanescentes train de degré $\left(n,1\right)$ et
homogènes de type $\left[n,1\right]$.}

\textcompwordmark{}\medskip{}

Soit $W_{\left[n,1\right]}$ le cardinal de l'ensemble $\frak{M}\left(x,y\right)_{\left[n,1\right]}$
des monômes de type $\left[n,1\right]$, du fait que l'on peut
écrire tout $w\in\frak{M}\left(x,y\right)_{\left[n,1\right]}$
sous la forme $w=w_{1}w_{2}$ avec $w_{1}\in\frak{M}\left(x,y\right)_{\left[n-i\right]}$
et $w_{2}\in\frak{M}\left(x,y\right)_{\left[i,1\right]}$ où
$0\leq i\leq n-1$ on déduit immédiatement que
\[
W_{\left[n,1\right]}=\sum_{i=0}^{n-1}W_{\left[n-i\right]}W_{\left[i,1\right]},\quad W_{\left[0,1\right]}=1.
\]

Et les premières valeurs de $W_{\left[n,1\right]}$ sont:
\begin{center}
\begin{tabular}{c|c|c|c|c|c|c|c|c|c|c|c}
$n$ & 0 & 1 & 2 & 3 & 4 & 5 & 6 & 7 & 8 & 9 & 10\tabularnewline
\hline 
$W_{\left[n,1\right]}$ & 1 & 1 & 2 & 4 & 9 & 20 & 46 & 106 & 248 & 582 & 1376\tabularnewline
\end{tabular}\medskip{}
\par\end{center}

\subsubsection{Train identités évanescentes de degré $\left(n,1\right)$.}

\textcompwordmark{}

Pour tout entier $r\geq1$, on définit $x^{\left\{ r\right\} }y=x\left(x^{\left\{ r-1\right\} }y\right)$
où $x^{\left\{ 0\right\} }y=y$, on pose $\mathscr{F}=\left\{ x^{n+1}y,x^{\left\{ n\right\} }y;n\geq0\right\} $
et $\mathbb{Q}\left\langle \mathscr{F}\right\rangle $ dénote
le $\mathbb{Q}$-espace vectoriel engendré par l'ensemble $\mathscr{F}$. 
\begin{lem}
\label{lem:Dx_ds_M=00005Bn,1=00005D}Pour tout entier $p\geq2$
et $r\geq0$ on a:
\begin{align}
\partial_{x}\left(x^{p}y\right) & \;=\;2t^{p}+\sum_{i=2}^{p-1}t^{i}, & \partial_{y}\left(x^{p}y\right) & \;=\;t;\label{eq:Dx(x^p*y)}\\
\partial_{x}\bigl(x^{\left\{ r\right\} }y\bigr) & \;=\;\sum_{i=1}^{r}t^{i}, & \partial_{y}\bigl(x^{\left\{ r\right\} }y\bigr) & \;=\;t^{r}.\label{eq:Dx(x^=00007Bp=00007D*y)}
\end{align}
\end{lem}

\begin{proof}
On a $\partial_{x}\left(x^{p}y\right)=t\left(\partial_{x}\left(x^{p}\right)+\partial_{x}\left(y\right)\right)=t\partial_{x}\left(x^{p}\right)$
d'où le résultat en utilisant la relation (\ref{eq:Dx(x^k)}),
on a aussi $\partial_{y}\left(x^{p}y\right)=t\left(\partial_{y}\left(x^{p}\right)+\partial_{y}\left(y\right)\right)=t$.
Pour $r\geq1$ on a $\partial_{x}\bigl(x^{\left\{ r\right\} }y\bigr)=t\left(1+\partial_{x}\bigl(x^{\left\{ r-1\right\} }y\bigr)\right)$
et $\partial_{y}\bigl(x^{\left\{ r\right\} }y\bigr)=t\partial_{y}\bigl(x^{\left\{ r-1\right\} }y\bigr)$,
on en déduit les résultats par récursivité.
\end{proof}
\begin{prop}
Il n'existe pas de train identité évanescente de degré $\left(2,1\right)$.

Pour tout $n\geq3$ et tout $w\in\frak{M}\left(x,y\right)_{\left[n,1\right]}$
vérifiant $w\neq x^{n}y$ et $w\neq x^{\left\{ n\right\} }y$,
il existe un unique polynôme $P_{w}\in\mathbb{Q}\left\langle \mathscr{F}\right\rangle $
avec $\left|P_{w}\right|_{x}<n$ tel que le polynôme $w-P_{w}$
soit une train identité évanescente.
\end{prop}

\begin{proof}
On a $\partial_{x}\left(x^{2}y\right)=2t^{2}$, $\partial_{y}\left(x^{2}y\right)=t$,
$\partial_{x}\left(x\left(xy\right)\right)=t^{2}+t$ et $\partial_{y}\left(x\left(xy\right)\right)=t^{2}$,
or $\partial_{x}\left(xy\right)=\partial_{y}\left(xy\right)=t$,
$\partial_{x}\left(y\right)=0$ et $\partial_{y}\left(y\right)=1$,
avec ceci on montre sans difficulté qu'on ne peut pas trouver
$\left(\alpha,\beta,\gamma,\delta\right)\neq\left(0,0,0,0\right)$
tel que le polynôme $f=\alpha x^{2}y+\beta x\left(xy\right)+\gamma xy+\delta y$
vérifie $\partial_{x}f=\partial_{y}f=0$.

Etant donné $w\in\frak{M}\left(x,y\right)_{\left[n,1\right]}$
tel que $w\neq x^{n}y$ et $w\neq x^{\left\{ n\right\} }y$ avec
$n\geq3$. Soient $p=\left|\partial_{x}w\right|$ et $q=\left|\partial_{y}w\right|$,
du résultat c) du corollaire \ref{cor:Prop_de_Di} il vient $p,q\leq n$.
Du résultat \emph{d}) du corollaire \ref{cor:Prop_de_Di} on
déduit que $\partial_{x}w\left(0\right)=\partial_{y}w\left(0\right)=0$
et donc les valuations de $\partial_{x}w$ et $\partial_{y}w$
sont supérieures à 1, posons $\partial_{x}w=\sum_{k=1}^{n}\alpha_{k}t^{k}$
et $\partial_{y}w=\sum_{k=1}^{n}\beta_{k}t^{k}$ avec $\alpha_{k}=0$
si $k>p$ et $\beta_{k}=0$ dés que $k>q$. On cherche $P_{w}=\sum_{k=1}^{n}\lambda_{k}x^{k}y+\sum_{k=1}^{n}\mu_{k}x^{\left\{ k\right\} }y$
vérifiant $\partial_{x}P_{w}=\partial_{x}w$, $\partial_{y}P_{w}=\partial_{y}w$
et $P_{w}\left(1,1\right)=1$. En utilisant le lemme \ref{lem:Dx_ds_M=00005Bn,1=00005D}
on obtient 
\begin{align*}
\partial_{x}P_{w} & =\left(2\lambda_{n}+\mu_{n}\right)t^{n}+\sum_{i=1}^{n-1}\Bigl(2\lambda_{i}+\sum_{k=i+1}^{n}\lambda_{k}+\sum_{k=i}^{n}\mu_{k}\Bigr)t^{i}\\
\partial_{y}P_{w} & =\sum_{i=2}^{n}\mu_{i}t^{i}+\Bigl(\sum_{k=1}^{n}\lambda_{k}+\mu_{1}\Bigr)t.
\end{align*}

De l'équation $\partial_{y}P_{w}=\partial_{y}w$ il résulte $\mu_{i}=\beta_{i}$
pour $2\leq i\leq n$ et $\sum_{k=1}^{n}\lambda_{k}+\mu_{1}=\beta_{1}$.
De l'équation $\partial_{x}P_{w}=\partial_{x}w$ on déduit $2\lambda_{n}+\mu_{n}=\alpha_{n}$
et $2\lambda_{i}+\sum_{k=i+1}^{n}\lambda_{k}+\sum_{k=i}^{n}\mu_{k}=\alpha_{i}$
pour $1\leq i\leq n-1$, on a donc $\lambda_{n}=\frac{1}{2}\left(\alpha_{n}-\beta_{n}\right)$
et pour tout $2\leq i\leq n-1$ on trouve $\lambda_{i}=\frac{1}{2}\left(\alpha_{i}-\beta_{i}\right)+\sum_{k=i+1}^{n}\frac{1}{2^{k-i+1}}\left(\alpha_{k}+\beta_{k}\right)$,
enfin en écrivant $\alpha_{1}=2\lambda_{1}+\sum_{k=2}^{n}\lambda_{k}+\sum_{k=1}^{n}\mu_{k}$
sous la forme $\alpha_{1}=\lambda_{1}+\left(\sum_{k=1}^{n}\lambda_{k}+\mu_{1}\right)+\sum_{k=2}^{n}\beta_{k}$,
on obtient $\lambda_{1}=\alpha_{1}-\sum_{k=1}^{n}\beta_{k}=\alpha_{1}-1$,
tout ceci permet de déterminer $\mu_{1}=\beta_{1}-\sum_{k=1}^{n}\lambda_{k}$.
Et on peut vérifier que $P_{w}\left(1,1\right)=\sum_{i=1}^{n}\lambda_{i}+\sum_{i=1}^{n}\mu_{i}=\sum_{i=1}^{n}\lambda_{i}+\sum_{i=2}^{n}\beta_{i}+\left(\beta_{1}-\sum_{k=1}^{n}\lambda_{k}\right)=\sum_{i=1}^{n}\beta_{i}=\partial_{y}w\left(1\right)=1$
d'après le résultat (d) du corollaire \ref{cor:Prop_de_Di}.
On a montré que le système d'équations $\partial_{x}P_{w}=\partial_{x}w$,
$\partial_{y}P_{w}=\partial_{y}w$, $P_{w}\left(1,1\right)=1$
admet une unique solution $P_{w}$, de plus d'après le corollaire
\ref{cor:Prop_de_Di} on a $\alpha_{k},\beta_{k}\in\mathbb{N}$
et ce qui précède permet d'affirmer que $\lambda_{k},\mu_{k}\in\mathbb{Q}$
pour tout $1\leq k\leq n$ et donc $P_{w}\in\mathbb{Q}\left\langle \mathscr{F}\right\rangle $.
\end{proof}
\begin{cor}
Pour tout $n\geq3$, l'espace vectoriel des train polynômes évanescents
de degré $\left(n,1\right)$ est de dimension $W_{\left[n,1\right]}-2$.
\end{cor}

\begin{proof}
D'après la proposition précédente l'espace des polynômes évanescents
de type $\left[n,1\right]$ est engendré par les polynômes $w-P_{w}$
pour tout $w\in\frak{M}\left(x,y\right)_{\left[n,1\right]}$
tels que $w\notin\mathscr{F}$.
\end{proof}
\begin{thm}
\label{thm:Evanesc_type=00005Bn,1=00005D_id=0000E9al}Pour tout
entier $p,q\geq2$ et $r\geq0$ on pose:
\begin{eqnarray*}
E_{p,q}\left(x\right) & = & x^{p}x^{q}-x^{p+1}-x^{q+1}+x^{2},\\
F_{p,q}\left(x,y\right) & = & x^{p}\left(x^{q}y\right)-x\left(xy\right)-x^{p}y-x^{q+1}y+x^{2}y+xy,\\
F_{p,\left\{ r\right\} }\left(x,y\right) & = & x^{p}\bigl(x^{\left\{ r\right\} }y\bigr)-x^{p}y-x^{\left\{ r+1\right\} }y+xy,\\
F_{\left\{ r\right\} ,p}\left(x,y\right) & = & x^{\left\{ r\right\} }\left(x^{p}y\right)-x^{\left\{ r+1\right\} }y-x^{p+r}y+x^{r+1}y.
\end{eqnarray*}
 Soient $\mathscr{I}$ l'idéal engendré par la famille de polynômes
$\left(E_{p,q},F_{p,q},F_{p,\left\{ r\right\} },F_{\left\{ r\right\} ,p}\right)_{\substack{p,q\geq2\\
r\geq0
}
}$ et $\pi:K\left(x\right)\rightarrow{}^{K\left(x\right)}\!/_{\mathscr{\mathscr{I}}}$
la surjection canonique. Alors pour tout $n\geq3$ et tout monôme
$w\in\frak{M}\left(x,y\right)_{\left[n,1\right]}$ tel que $w\notin\mathscr{F}$
on a $\pi\left(w\right)=P_{w}$ et pour tout $f\in\bigoplus_{n\geq3}K\left(x,y\right)_{\left[n,1\right]}$,
le polynôme $f-\pi\left(f\right)$ est une train identité évanescente.
\end{thm}

\begin{proof}
On a vu pour le théorème \ref{thm:Evanesc_type=00005Bn=00005D_id=0000E9al}
que les polynômes $E_{p,q}$ sont évanescents, montrons qu'il
en est de même pour les polynômes $F_{p,q}$, $F_{p,\left\{ r\right\} }$
et $F_{\left\{ r\right\} ,q}$. 

Pour $p,q\geq2$, on a $\partial_{x}\left(x^{p}\left(x^{q}y\right)\right)=t\left(\partial_{x}\left(x^{p}\right)+\partial_{x}\left(x^{q}y\right)\right)=t\partial_{x}\left(x^{p}\right)+t^{2}\partial_{x}\left(x^{q}\right)$
et $\partial_{y}\left(x^{p}\left(x^{q}y\right)\right)=t\partial_{y}\left(x^{q}y\right)$,
avec la relation (\ref{eq:Dx(x^k)}) et le lemme \ref{lem:Dx_ds_M=00005Bn,1=00005D}
on obtient:
\begin{align}
\partial_{x}\left(x^{p}\left(x^{q}y\right)\right) & =2t^{p}+2t^{q+1}+\sum_{i=2}^{p-1}t^{i}+\sum_{i=3}^{q}t^{i}, & \partial_{y}\left(x^{p}\left(x^{q}y\right)\right) & =t^{2}.\label{eq:Dx(x^p*(x^q*y))}
\end{align}

Pour $p\geq2$ et $r\geq0$ on a $\partial_{x}\left(x^{p}\left(x^{\left\{ r\right\} }y\right)\right)=t\left(\partial_{x}\left(x^{p}\right)+\partial_{x}\left(x^{\left\{ r\right\} }y\right)\right)$
et $\partial_{y}\left(x^{p}\left(x^{\left\{ r\right\} }y\right)\right)=t\partial_{y}\left(x^{\left\{ r\right\} }y\right)$,
avec la relation (\ref{eq:Dx(x^k)}) et le lemme \ref{lem:Dx_ds_M=00005Bn,1=00005D}
on en déduit
\begin{align}
\partial_{x}\left(x^{p}\left(x^{\left\{ r\right\} }y\right)\right) & =2t^{p}+\sum_{i=2}^{p-1}t^{i}+\sum_{i=2}^{r+1}t^{i}, & \partial_{y}\left(x^{p}\left(x^{\left\{ r\right\} }y\right)\right) & =t^{r+1}.\label{eq:Dx(x^p*(x^=00007Br=00007D*y))}
\end{align}

Pour tout $r\geq1$ on a $\partial_{x}\left(x^{\left\{ r\right\} }\left(x^{p}y\right)\right)=t\left(1+\partial_{x}\left(x^{\left\{ r-1\right\} }\left(x^{p}y\right)\right)\right)$
par conséquent $\partial_{x}\left(x^{\left\{ r\right\} }\left(x^{p}y\right)\right)=\sum_{i=1}^{r}t^{i}+t^{r}\partial_{x}\left(x^{p}y\right)$
et $\partial_{y}\left(x^{\left\{ r\right\} }\left(x^{p}y\right)\right)=t\partial_{y}\left(x^{\left\{ r-1\right\} }\left(x^{p}y\right)\right)$
il en résulte $\partial_{y}\left(x^{\left\{ r\right\} }\left(x^{p}y\right)\right)=t^{r}\partial_{y}\left(x^{p}y\right)$,
et avec le lemme \ref{lem:Dx_ds_M=00005Bn,1=00005D} on a
\begin{align}
\partial_{x}\left(x^{\left\{ r\right\} }\left(x^{p}y\right)\right) & =2t^{r+p}+\sum_{i=r+1}^{r+p-1}t^{i}+\sum_{i=1}^{r}t^{i}, & \partial_{y}\left(x^{\left\{ r\right\} }\left(x^{p}y\right)\right) & =t^{r+1}.\label{eq:Dx(x^=00007Br=00007D*(x^p*y))}
\end{align}

Un simple calcul utilisant les relations obtenues ci-dessus et
celles du lemme \ref{lem:Dx_ds_M=00005Bn,1=00005D} montre que
les polynômes $F_{p,q}$, $F_{p,\left\{ r\right\} }$ et $F_{\left\{ r\right\} ,q}$
sont évanescents.

Montrons que pour tout $w\in\frak{M}\left(x,y\right)_{\left[n,1\right]}$
tel que $w\notin\mathscr{F}$ , le polynôme $w-\pi\left(w\right)$
est évanescent. C'est immédiat pour tout $w\in\frak{M}\left(x,y\right)_{\left[3,1\right]}$
tel que $w\neq x^{3}y$ et $w\neq x^{\left\{ 3\right\} }y$ d'après
la relation (\ref{eq:Dx(x^p*(x^q*y))}). Supposons la propriété
vraie pour tous les monômes de $\frak{M}\left(x,y\right)_{\left[p,1\right]}\setminus\left\{ x^{p}y,x^{\left\{ p\right\} }y\right\} $
avec $3\leq p\leq n$, soit $w\in\frak{M}\left(x,y\right)_{\left[n+1,1\right]}$
vérifiant $w\neq x^{n+1}y$ et $w\neq x^{\left\{ n+1\right\} }y$.
On a deux cas:

– il existe $u\in\frak{M}\left(x,y\right)_{\left[p\right]}$
et $v\in\frak{M}\left(x,y\right)_{\left[q,1\right]}$ tels que
$w=uv$ avec $u\neq x$, $1\leq p,q$ et $p+q=n+1$, alors on
a $\partial_{x}\left(w-\pi\left(w\right)\right)=\partial_{x}\left(uv-\pi\left(u\right)\pi\left(v\right)\right)=t\left(\partial_{x}\left(u-\pi\left(u\right)\right)+\partial_{x}\left(v-\pi\left(v\right)\right)\right)=0$
et de même $\partial_{y}\left(w-\pi\left(w\right)\right)=0$;

– il existe $u\in\frak{M}\left(x,y\right)_{\left[p\right]}$
et $v\in\frak{M}\left(x,y\right)_{\left[q\right]}$ tels que
$w=\left(uv\right)y$ avec $uv\neq x^{n+1}$, $1\leq p,q$ et
$p+q=n+1$, alors $\partial_{x}\left(w-\pi\left(w\right)\right)=\partial_{x}\left(\left(\left(uv\right)-\pi\left(u\right)\pi\left(v\right)\right)y\right)$,
il en résulte $\partial_{x}\left(w-\pi\left(w\right)\right)=t^{2}\left(\partial_{x}\left(u-\pi\left(u\right)\right)+\partial_{x}\left(v-\pi\left(v\right)\right)\right)=0$,
de manière analogue on a $\partial_{y}\left(w-\pi\left(w\right)\right)=0$.

On montre aisément par récurrence que pour tout $w\in\frak{M}\left(x,y\right)_{\left[n,1\right]}$
tel que $w\neq x^{n}y$ et $w\neq x^{\left\{ n\right\} }y$ on
a $\left|\pi\left(w\right)\right|_{x}<n$ et $\pi\left(w\right)\in\mathbb{Z}\left\langle \mathscr{F}\right\rangle $,
alors par unicité du polynôme $P_{w}$ tel que $w-P_{w}$ est
évanescent on a $\pi\left(w\right)=P_{w}$.
\end{proof}
Ce théorème donne un algorithme efficace pour construire les
train identités évanescentes de degré $\left(n,1\right)$, illustrons-le
par un exemple.
\begin{example}
Soit $w=x^{5}\left(x\left(x\left(x\left(x^{4}\left(\left(x^{2}x^{3}\right)y\right)\right)\right)\right)\right)$,
dans l'algèbre $^{K\left(x\right)}\!/_{\mathscr{F}}$ on trouve
modulo $E_{p,q}$: $\left(x^{2}x^{3}\right)y=x^{4}y+x^{3}y-x^{2}y$,
ensuite modulo $F_{p,q}$ on a: $x^{4}\left(\left(x^{2}x^{3}\right)y\right)=x\left(xy\right)+x^{5}y+2x^{4}y-x^{3}y-x^{2}y-xy$.
Modulo $F_{\left\{ r\right\} ,p}$ on obtient $x^{\left\{ 3\right\} }\left(x^{4}\left(\left(x^{2}x^{3}\right)y\right)\right)=x^{\left\{ 5\right\} }y+x^{8}y+2x^{7}y-x^{6}y-x^{5}y-x^{4}y$,
enfin modulo $F_{p,\left\{ r\right\} }$ et $F_{p,q}$ on obtient
finalement $\pi\left(w\right)=x^{\left\{ 6\right\} }y+x^{9}y+2x^{8}y-x^{7}y-x^{6}y-xy$,
on peut donc affirmer que le polynôme de type $\left[17,1\right]$:
\[
x^{5}\left(x\left(x\left(x\left(x^{4}\left(\left(x^{2}x^{3}\right)y\right)\right)\right)\right)\right)-x\left(x\left(x\left(x\left(x\left(xy\right)\right)\right)\right)\right)-x^{9}y-2x^{8}y+x^{7}y+x^{6}y+xy
\]
est une identité évanescente.
\end{example}

En appliquant cet algorithme on obtient aisément les générateurs
des train identités évanescentes

\begin{longtable}{lll}
– de degré $\left(3,1\right)$:$\medskip$ &  & \tabularnewline
{\small{}$\;x^{2}\left(xy\right)-x\left(xy\right)-x^{2}y+xy;$} &  & {\small{}$\;x\left(x^{2}y\right)-x\left(xy\right)-x^{3}y+x^{2}y.$}\tabularnewline
 &  & \tabularnewline
– de degré $\left(4,1\right)$:$\medskip$ &  & \tabularnewline
{\small{}$\left(x^{2}x^{2}\right)y-2x^{3}y+x^{2}y;$} &  & {\small{}$\;x^{2}\left(x\left(xy\right)\right)-x\left(x\left(xy\right)\right)-x^{2}y+xy;$}$\medskip$\tabularnewline
{\small{}$\;x^{2}\left(x^{2}y\right)-x\left(xy\right)-x^{3}y+xy;$} &  & {\small{}$\;x\left(x^{2}\left(xy\right)\right)-x\left(x\left(xy\right)\right)-x^{3}y+x^{2}y;$}$\medskip$\tabularnewline
{\small{}$\;x^{3}\left(xy\right)-x\left(xy\right)-x^{3}y+xy;$} &  & {\small{}$\;x\left(x\left(x^{2}y\right)\right)-x\left(x\left(xy\right)\right)-x^{4}y+x^{3}y.$}$\medskip$\tabularnewline
{\small{}$\;x\left(x^{3}y\right)-x\left(xy\right)-x^{4}y+x^{2}y;$} &  & \tabularnewline
 &  & \tabularnewline
– de degré $\left(5,1\right)$:$\medskip$ &  & \tabularnewline
{\small{}$\;x^{2}\left(x^{2}\left(xy\right)\right)-x\left(x\left(xy\right)\right)-x^{3}y+xy;$} &  & {\small{}$\;x^{3}\left(x^{2}y\right)-x\left(xy\right)-2x^{3}y+x^{2}y+xy;$}$\medskip$\tabularnewline
{\small{}$\;x\left(x\left(x^{3}y\right)\right)-x\left(x\left(xy\right)\right)-x^{5}y+x^{3}y;$} &  & {\small{}$\;x\left(x^{2}\left(x^{2}y\right)\right)-x\left(x\left(xy\right)\right)-x^{4}y+x^{2}y;$}$\medskip$\tabularnewline
{\small{}$\;x^{3}\left(x\left(xy\right)\right)-x\left(x\left(xy\right)\right)-x^{3}y+xy;$} &  & {\small{}$\left(x^{2}x^{2}\right)\left(xy\right)-x\left(xy\right)-2x^{3}y+x^{2}y+xy;$}$\medskip$\tabularnewline
{\small{}$\;x\left(x^{3}\left(xy\right)\right)-x\left(x\left(xy\right)\right)-x^{4}y+x^{2}y;$} &  & {\small{}$\;x\left(\left(x^{2}x^{2}\right)y\right)-x\left(xy\right)-2x^{4}y+x^{3}y+x^{2}y;$}$\medskip$\tabularnewline
{\small{}$\;x^{2}\left(x^{3}y\right)-x\left(xy\right)-x^{4}y+xy;$} &  & {\small{}$\;x^{2}\left(x\left(x\left(xy\right)\right)\right)-x\left(x\left(x\left(xy\right)\right)\right)-x^{2}y+xy;$}$\medskip$\tabularnewline
{\small{}$\left(x^{3}x^{2}\right)y-x^{4}y-x^{3}y+x^{2}y;$} &  & {\small{}$\;x\left(x^{2}\left(x\left(xy\right)\right)\right)-x\left(x\left(x\left(xy\right)\right)\right)-x^{3}y+x^{2}y;$}$\medskip$\tabularnewline
{\small{}$\;x^{4}\left(xy\right)-x\left(xy\right)-x^{4}y+xy;$} &  & {\small{}$\;x\left(x\left(x^{2}\left(xy\right)\right)\right)-x\left(x\left(x\left(xy\right)\right)\right)-x^{4}y+x^{3}y;$}$\medskip$\tabularnewline
{\small{}$\;x\left(x^{4}y\right)-x\left(xy\right)-x^{5}y+x^{2}y;$} &  & {\small{}$\;x\left(x\left(x\left(x^{2}y\right)\right)\right)-x\left(x\left(x\left(xy\right)\right)\right)-x^{5}y+x^{4}y;$}$\medskip$\tabularnewline
{\small{}$\left(\left(x^{2}x^{2}\right)x\right)y-2x^{4}y+x^{3}y;$} &  & {\small{}$\;x^{2}\left(x\left(x^{2}y\right)\right)-x\left(x\left(xy\right)\right)-x^{4}y+x^{3}y-x^{2}y+xy.$}\tabularnewline
\end{longtable}

\subsubsection{Identités homogènes évanescentes de type $\left[n,1\right]$.}
\begin{prop}
Il n'existe pas de d'identité homogène évanescente de type $\left[n,1\right]$
avec $n\leq3$. Pour $n\geq4$, l'espace des identités homogènes
évanescentes de type $\left[n,1\right]$ est engendré par au
moins $W_{\left[n,1\right]}-2\left(n-1\right)$ identités homogènes
évanescentes.
\end{prop}

\begin{proof}
Soient $\frak{M}\left(x,y\right)_{\left[n,1\right]}=\left\{ w_{k};1\leq k\leq N\right\} $
où $N=W_{\left[n,1\right]}$ et $f=\sum_{i=1}^{N}\alpha_{i}w_{i}$,
on cherche $\left(\alpha_{i}\right)_{1\leq i\leq N}\in K^{N}$
tel que $\partial_{x}f=\partial_{y}f=0$. avec $\partial_{x}f=\sum_{i=1}^{N}\alpha_{i}\partial_{x}w_{i}$
et $\partial_{y}f=\sum_{i=1}^{N}\alpha_{i}\partial_{y}w_{i}$.
Pour tout $1\leq i\leq N$ on a $\partial_{x}w_{i},\partial_{y}w_{i}\in K\left[t\right]$
avec $\left|\partial_{x}w_{i}\right|\leq n$ et $\left|\partial_{y}w_{i}\right|\leq n$
par conséquent les équations $\partial_{x}f=0$ et $\partial_{y}f=0$
se traduisent chacun par deux systèmes linéaires de $n$ équations
à $N$ inconnues.

Ainsi le cas $n=2$, avec $f\left(x,y\right)=\alpha x^{2}y+\beta x\left(xy\right)$
de $\partial_{y}\left(f\right)=\alpha t+\beta t^{2}$ on déduit
$\alpha=\beta=0$, il n'y donc pas de polynôme homogène évanescent
de type $\left[2,2\right]$. 

Si $n=3$, en prenant $f\left(x,y\right)=\alpha x^{3}y+\beta x^{2}\left(xy\right)+\gamma x\left(x^{2}y\right)+\delta x\left(x\left(xy\right)\right)$,
on a $\partial_{y}\left(f\right)=\alpha t+\beta t^{2}+\gamma t^{2}+\delta t^{3}$
donc de l'équation $\partial_{y}f=0$ on déduit $\alpha=\delta=0$
et $\beta$$+\gamma=0$, alors $\partial_{y}\left(\beta x^{2}\left(xy\right)+\gamma x\left(x^{2}y\right)\right)=\beta\left(3t^{2}\right)+\gamma\left(2t^{3}+t\right)=0$
on déduit $\beta=\gamma=0$, il n'existe pas de polynôme homogène
évanescent de type $\left[3,2\right]$.

Supposons $n\geq4$, de $f\left(1,1\right)=0$ on déduit que
$\sum_{i=1}^{N}\alpha_{i}=0$, par conséquent le système d'équations
$\partial_{x}f=0$ est de rang $\leq n-1$, il en est de même
du système $\partial_{y}f=0$, donc le système d'équations $\partial_{x}f=\partial_{y}f=0$
est de rang $\leq2\left(n-1\right)$ par conséquent l'espace
des solutions est de dimension $\geq W_{\left[n,1\right]}-2\left(n-1\right)$.
\end{proof}
La méthode utilisée dans la démonstration permet de donner les
générateurs des identités homogènes évanescentes:
\begin{align*}
 & \text{-- de type }\left[4,1\right]\\
 & \left(x^{2}x^{2}\right)y-x^{3}\left(xy\right);\\
 & \;x^{4}y+x\left(x^{2}\left(xy\right)\right)-\left(x^{2}x^{2}\right)y-x\left(x\left(x^{2}y\right)\right);\\
 & \;x\left(x^{3}y\right)+x^{2}\left(x\left(xy\right)\right)-x^{3}\left(xy\right)-x\left(x\left(x^{2}y\right)\right).\\
\\
 & \text{-- de type \ensuremath{\left[5,1\right]}}\\
 & \;x\left(x^{2}\left(x^{2}y\right)\right)-x\left(x^{3}\left(xy\right)\right);\\
 & \left(x^{3}x^{2}\right)y+x\left(x\left(x^{3}y\right)\right)-x^{5}y-x\left(x^{3}\left(xy\right)\right);\\
 & \left(\left(x^{2}x^{2}\right)x\right)y+x\left(x^{4}y\right)-x^{5}y-x\left(\left(x^{2}x^{2}\right)y\right);\\
 & \;x^{4}\left(xy\right)+x\left(x\left(x^{3}y\right)\right)-x^{2}\left(x\left(x^{2}y\right)\right)-x\left(x^{4}y\right);\\
 & \;x^{2}\left(x^{3}y\right)+x\left(x\left(x^{3}y\right)\right)-x^{2}\left(x\left(x^{2}y\right)\right)-x\left(x^{4}y\right);\\
 & \;x^{2}\left(x\left(x^{2}y\right)\right)+x\left(x^{2}\left(x\left(xy\right)\right)\right)-x\left(x^{3}\left(xy\right)\right)-x^{2}\left(x\left(x\left(xy\right)\right)\right);\\
 & \;x^{2}\left(x\left(x^{2}y\right)\right)+x\left(x\left(x\left(x^{2}y\right)\right)\right)-x^{2}\left(x\left(x\left(xy\right)\right)\right)-x\left(x\left(x^{3}y\right)\right);\\
 & \;x^{3}\left(x^{2}y\right)+x\left(\left(x^{2}x^{2}\right)y\right)+2x\left(x\left(x^{3}y\right)\right)-x^{2}\left(x\left(x^{2}y\right)\right)-2x\left(x^{4}y\right)-x\left(x^{3}\left(xy\right)\right);\\
 & \;x^{3}\left(x\left(xy\right)\right)+x\left(\left(x^{2}x^{2}\right)y\right)+x\left(x\left(x^{3}y\right)\right)-x^{2}\left(x\left(x^{2}y\right)\right)-x\left(x^{4}y\right)-x\left(x^{3}\left(xy\right)\right);\\
 & \;x\left(\left(x^{2}x^{2}\right)y\right)+x\left(x\left(x^{3}y\right)\right)+x^{2}\left(x^{2}(xy\right)-x\left(x^{4}y\right)-x\left(x^{3}\left(xy\right)\right)-x^{2}\left(x\left(x^{2}y\right)\right);\\
 & \left(x^{2}x^{2}\right)\left(xy\right)+x\left(\left(x^{2}x^{2}\right)y\right)+2x\left(x\left(x^{3}y\right)\right)-x^{2}\left(x\left(x^{2}y\right)\right)-2x\left(x^{4}y\right)-x\left(x^{3}\left(xy\right)\right);\\
 & \;x^{2}\left(x\left(x^{2}y\right)\right)+x\left(x^{4}y\right)+x\left(x\left(x^{2}\left(xy\right)\right)\right)-x^{2}\left(x\left(x\left(xy\right)\right)\right)-x\left(\left(x^{2}x^{2}\right)y\right)-x\left(x\left(x^{3}y\right)\right).
\end{align*}

\medskip{}

\subsection{Identités évanescentes train de degré $\left(n,2\right)$ et
homogènes de type $\left[n,2\right]$.}

\textcompwordmark{}\medskip{}

Soit $W_{\left[n,2\right]}$ le cardinal de l'ensemble $\frak{M}\left(x,y\right)_{\left[n,2\right]}$
des monômes de type $\left[n,2\right]$. Pour $w\in\frak{M}\left(x,y\right)_{\left[n,2\right]}$
il y a deux façons de décomposer $w$ comme produit de deux monômes
$w=w_{1}w_{2}$. Soit en prenant $w_{1}\in\frak{M}\left(x,y\right)_{\left[n-i\right]}$
et $w_{2}\in\frak{M}\left(x,y\right)_{\left[i,2\right]}$ pour
$0\leq i\leq n-1$ et dans ce cas on a $\sum_{i=0}^{n-1}W_{\left[n-i\right]}W_{\left[i,2\right]}$
écritures possibles. Soit avec $w_{1}\in\frak{M}\left(x,y\right)_{\left[i,1\right]}$
et $w_{2}\in\frak{M}\left(x,y\right)_{\left[j,1\right]}$ pour
$0\leq i\leq j\leq n$ tels que $i+j=n$, on en déduit que $2i\leq n$
et selon la parité de $n$ on a deux cas. Si $n$ est impair,
$n=2p+1$ pour tout $0\leq i\leq p$ les mots $w_{1}$ et $w_{2}$
sont de degrés en $x$ distincts il y donc $\sum_{i=0}^{p}W_{\left[2p+1-i,1\right]}W_{\left[i,1\right]}$
décompositions possibles de $w$ en produit de deux monômes.
Si $n$ est pair, $n=2p$ pour tout $0\leq i<p$ les monômes
$w_{1}$ et $w_{2}$ ne sont pas de même degré en $x$ on a donc
$\sum_{i=0}^{p-1}W_{\left[2p-i,1\right]}W_{\left[i,1\right]}$
décompositions de cette forme; pour $i=p$ on a $w_{1}$ et $w_{2}$
dans $\frak{M}\left(x,y\right)_{\left[p,1\right]}$ d'où $\tbinom{W_{\left[p,1\right]}+1}{2}$
décompositions de $w$.

En résumé on a obtenu:
\begin{align*}
W_{\left[0,2\right]} & =1,\\
W_{\left[n,2\right]} & =\sum_{i=0}^{n-1}W_{\left[2p+1-i\right]}W_{\left[i,2\right]}+\sum_{i=0}^{\left\lfloor \nicefrac{n}{2}\right\rfloor }W_{\left[n-i,1\right]}W_{\left[i,1\right]}+\begin{cases}
\qquad0 & \text{si }n\text{ est impair}\\
\dbinom{W_{\left[\nicefrac{n}{2},1\right]}}{2} & \text{si }n\text{ est pair}
\end{cases},\left(n\geq1\right).
\end{align*}

Les premières valeurs de $W_{\left[n,2\right]}$ sont
\begin{center}
\begin{tabular}{c|c|c|c|c|c|c|c|c|c|c|c}
$n$ & 0 & 1 & 2 & 3 & 4 & 5 & 6 & 7 & 8 & 9 & 10\tabularnewline
\hline 
$W_{\left[n,2\right]}$ & 1 & 2 & 6 & 15 & 41 & 106 & 280 & 726 & 1891 & 4886 & 12622\tabularnewline
\end{tabular}
\par\end{center}

\begin{center}
\medskip{}
\par\end{center}

\subsubsection{Train identités évanescentes de degré $\left(n,2\right)$.}

\textcompwordmark{}\medskip{}

Pour tout $f\in K\left(x,y\right)$ et tout entier $r\geq1$,
on définit $x^{\left\{ r\right\} }f=x\left(x^{\left\{ r-1\right\} }f\right)$
où $x^{\left\{ 0\right\} }f=f$. 
\begin{lem}
\label{lem:Dx_ds_M=00005Bn,2=00005D}Pour tout entier $p\geq2$
et $r\geq0$ on a:
\begin{align}
\partial_{x}\bigl(\bigl(x^{\left\{ r\right\} }y\bigr)y\bigr) & \;=\;\sum_{i=2}^{r+1}t^{i}, & \partial_{y}\bigl(\bigl(x^{\left\{ r\right\} }y\bigr)y\bigr) & \;=\;t^{r+1}+t;\label{eq:Dx(x^=00007Br=00007D*y*y)}\\
\partial_{x}\bigl(x^{\left\{ r\right\} }y^{2}\bigr) & \;=\;\sum_{i=1}^{r}t^{i}, & \partial_{y}\bigl(x^{\left\{ r\right\} }y^{2}\bigr) & \;=\;2t^{r+1}.\label{eq:Dx(x^=00007Br=00007D*y2)}
\end{align}
\end{lem}

\begin{proof}
Partant de $\partial_{x}\bigl(\bigl(x^{\left\{ r\right\} }y\bigr)y\bigr)=t\partial_{x}\bigl(x^{\left\{ r\right\} }y\bigr)=t\left(1+\partial_{x}\bigl(x^{\left\{ r-1\right\} }y\bigr)\right)$
et $\partial_{y}\left(\bigl(x^{\left\{ r\right\} }y\bigr)y\right)=t\left(\partial_{y}\bigl(x^{\left\{ r\right\} }y\bigr)+1\right)=t^{2}\partial_{y}\bigl(x^{\left\{ r-1\right\} }y\bigr)+t$,
et de $\partial_{x}\bigl(x^{\left\{ r\right\} }y^{2}\bigr)=t\left(1+\partial_{x}\bigl(x^{\left\{ r-1\right\} }y^{2}\bigr)\right)$
et $\partial_{y}\bigl(x^{\left\{ r\right\} }y^{2}\bigr)=t\partial_{y}\bigl(x^{\left\{ r-1\right\} }y^{2}\bigr)$
on obtient récursivement les résultats énoncés.
\end{proof}
On note $\mathscr{F}=\left\{ \left(x^{\left\{ r\right\} }y\right)y,x^{\left\{ r\right\} }y^{2};r\geq0\right\} $
et $\mathbb{Q}\left\langle \mathscr{F}\right\rangle $ le $\mathbb{Q}$-espace
vectoriel engendré par l'ensemble $\mathscr{F}$. 
\begin{prop}
Il n'existe pas de train identité évanescente de type $\left(1,2\right)$.

Pour tout $n\geq2$ et tout $w\in\frak{M}\left(x,y\right)_{\left[n,2\right]}$
où $w\neq\left(x^{\left\{ n\right\} }y\right)y$ et $w\neq x^{\left\{ n\right\} }y^{2}$,
il existe un unique polynôme $P_{w}\in\mathbb{Q}\left\langle \mathscr{F}\right\rangle $
avec $\left|P_{w}\right|_{x}\leq n$ tel que le polynôme $w-P_{w}$
soit une train identité évanescente.
\end{prop}

\begin{proof}
On a vu qu'il n'existe pas de train identité évanescente de degré
$\left(2,1\right)$ donc par permutation des variables $x$ et
$y$ il n'en existe pas de degré $\left(1,2\right)$.

Soit $n\geq2$. Pour tout $w\in\frak{M}\left(x,y\right)_{\left[n,2\right]}$
on a $\left|\partial_{x}w\right|\leq n+1$ avec $\left|\partial_{x}\left(x^{\left\{ n-k\right\} }\left(\left(x^{\left\{ k\right\} }y\right)y\right)\right)\right|=n+1$
pour tout $1\leq k<n$. 

Soit $w\in\frak{M}\left(x,y\right)_{\left[n,2\right]}$ vérifiant
les conditions de la proposition, on cherche $P_{w}=\sum_{r=1}^{n}\alpha_{r}\left(x^{\left\{ r\right\} }y\right)y+\sum_{s=0}^{n}\beta_{s}x^{\left\{ s\right\} }y^{2}$
tel que le polynôme $f=w-P_{w}$ vérifie $\partial_{x}f=\partial_{y}f=0$
et $f\left(1,1\right)=0$. Soit $\partial_{x}w=\sum_{i=1}^{n+1}\lambda_{i}t^{i}$
et $\partial_{y}w=\sum_{i=1}^{n+1}\mu_{i}t^{i}$, on a:
\begin{align*}
\partial_{x}f & =\left(\lambda_{n+1}-\alpha_{n}\right)t^{n+1}+\sum_{i=2}^{n}\biggl(\lambda_{i}-\Bigl(\sum_{r=i-1}^{n}\alpha_{r}+\sum_{s=i}^{n}\beta_{s}\Bigr)\biggr)t^{i}+\Bigl(\lambda_{1}-\sum_{s=1}^{n}\beta_{s}\Bigr)t,\\
\partial_{y}f & =\sum_{i=2}^{n+1}\left(\mu_{i}-\Bigl(\alpha_{i-1}+2\beta_{i-1}\Bigr)\right)t^{i}+\biggl(\mu_{1}-\Bigl(\sum_{r=1}^{n}\alpha_{r}+2\beta_{0}\Bigr)\biggr)t.
\end{align*}
 Et la solution du système linéaire $\partial_{x}f=\partial_{y}f=0$
et $\sum_{r=1}^{n}\alpha_{r}+\sum_{s=0}^{n}\beta_{s}=1$ est:
\begin{align*}
\alpha_{i} & =\lambda_{i+1}-\sum_{k=i+2}^{n+1}\frac{1}{2^{k-i-1}}\left(\mu_{k}+\lambda_{k}\right), & \left(1\leq i\leq n\right)\\
\beta_{i} & =\frac{1}{2}\left(\mu_{i+1}-\lambda_{i+1}\right)+\sum_{k=i+2}^{n+1}\frac{1}{2^{k-i}}\left(\mu_{k}+\lambda_{k}\right), & \left(1\leq i\leq n\right)\\
\beta_{0} & =1-\sum_{k=2}^{n+1}\frac{1}{2^{k-1}}\left(\mu_{k}+\lambda_{k}\right).
\end{align*}
Or d'après le résultat \emph{a}) du corollaire \ref{cor:Prop_de_Di}
on a $\lambda_{i},\mu_{i}\in\mathbb{N}$ pour tout $1\leq i\leq n$,
par conséquent on a $\alpha_{i},\beta_{i}\in\mathbb{Q}$, ce
qui achève la démonstration. 
\end{proof}
On en déduit aussitôt que
\begin{cor}
Pour $n\geq2$, l'espace vectoriel des train polynômes évanescents
de degré $\left(n,2\right)$ est de dimension $W_{\left[n,2\right]}-2$.
\end{cor}

Le résultat qui suit donne une procédure pour construire rapidement
des identités évanescentes à partir d'éléments pris dans $\frak{M}\left(x,y\right)_{\left[n,2\right]}$.
\begin{thm}
\label{thm:Evanesc_type=00005Bn,2=00005D_id=0000E9al}Pour tout
entier $p,q\geq2$ et $r\geq0$ on pose:
\begin{eqnarray*}
E_{p,q}^{\left[n\right]}\left(x\right) & = & x^{p}x^{q}-x^{p+1}-x^{q+1}+x^{2};\\
E_{p,q}^{\left[n,1\right]}\left(x,y\right) & = & x^{p}\left(x^{q}y\right)-x\left(xy\right)-x^{p}y-x^{q+1}y+x^{2}y+xy;\\
E_{p,\left\{ r\right\} }^{\left[n,1\right]}\left(x,y\right) & = & x^{p}\bigl(x^{\left\{ r\right\} }y\bigr)-x^{p}y-x^{\left\{ r+1\right\} }y+xy;\\
E_{\left\{ r\right\} ,p}^{\left[n,1\right]}\left(x,y\right) & = & x^{\left\{ r\right\} }\left(x^{p}y\right)-x^{\left\{ r+1\right\} }y-x^{p+r}y+x^{r+1}y;\\
E_{p,\left\{ r\right\} }^{\left[n,2\right]}\left(x,y\right) & = & x^{p}\bigl(\bigl(x^{\left\{ r\right\} }y\bigr)y\bigr)-2\bigl(x^{\left\{ p-1\right\} }y\bigr)y-\bigl(x^{\left\{ r+1\right\} }y\bigr)y+x^{\left\{ p-1\right\} }y^{2}\\
 &  & +\left(xy\right)y)-xy^{2}+y^{2};\\
E_{\left\{ r\right\} ,\left\{ s\right\} }^{\left[n,2\right]}\left(x,y\right) & = & x^{\left\{ r\right\} }\bigl(\bigl(x^{\left\{ s\right\} }y\bigr)y\bigr)-\bigl(x^{\left\{ r+s\right\} }y\bigr)y+\bigl(x^{\left\{ r\right\} }y\bigr)y-x^{\left\{ r\right\} }y^{2};\\
F_{p,\left\{ r\right\} }^{\left[n,2\right]}\left(x,y\right) & = & x^{p}\bigl(x^{\left\{ r\right\} }y^{2}\bigr)-2\bigl(x^{\left\{ p-1\right\} }y\bigr)y+x^{\left\{ p-1\right\} }y^{2}-x^{\left\{ r+1\right\} }y^{2}+y^{2};\\
G_{p,q}^{\left[n,2\right]}\left(x,y\right) & = & \left(x^{p}y\right)\left(x^{q}y\right)-2\bigl(x^{\left\{ p\right\} }y\bigr)y-2\bigl(x^{\left\{ q\right\} }y\bigr)y+x^{\left\{ p\right\} }y^{2}+x^{\left\{ q\right\} }y^{2}\\
 &  & +2\left(xy\right)y-2xy^{2}+y^{2};\\
G_{p,\left\{ r\right\} }^{\left[n,2\right]}\left(x,y\right) & = & \left(x^{p}y\right)\bigl(x^{\left\{ r\right\} }y\bigr)-2\bigl(x^{\left\{ p\right\} }y\bigr)y-\bigl(x^{\left\{ r\right\} }y\bigr)y+x^{\left\{ p\right\} }y^{2}\\
 &  & +\left(xy\right)y-xy^{2}+y^{2};\\
G_{\left\{ r\right\} ,\left\{ s\right\} }^{\left[n,2\right]}\left(x,y\right) & = & \bigl(x^{\left\{ r\right\} }y\bigr)\bigl(x^{\left\{ s\right\} }y\bigr)-\bigl(x^{\left\{ r\right\} }y\bigr)y-\bigl(x^{\left\{ s\right\} }y\bigr)y+y^{2}.
\end{eqnarray*}

Soient $\mathscr{G}$ l'idéal engendré par la famille de polynômes
\[
\left(E_{p,q}^{\left[n\right]},E_{p,q}^{\left[n,1\right]},E_{p,\left\{ r\right\} }^{\left[n,1\right]},E_{\left\{ r\right\} ,p}^{\left[n,1\right]},E_{p,\left\{ r\right\} }^{\left[n,2\right]},E_{p,\left\{ r\right\} }^{\left[n,2\right]},E_{\left\{ r\right\} ,\left\{ s\right\} }^{\left[n,2\right]},F_{p,\left\{ r\right\} }^{\left[n,2\right]},G_{p,q}^{\left[n,2\right]},G_{p,\left\{ r\right\} }^{\left[n,2\right]},G_{\left\{ r\right\} ,\left\{ s\right\} }^{\left[n,2\right]}\right)_{\substack{p,q\geq2\\
r,s\geq0
}
}
\]
et $\pi:K\left(x\right)\rightarrow{}^{K\left(x\right)}\!/_{\mathscr{G}}$
la surjection canonique. 

Alors pour tout $n\geq2$ et tout monôme $w\in\frak{M}\left(x,y\right)_{\left[n,2\right]}$
tel que $w\neq\left(x^{\left\{ n\right\} }y\right)y$ et $w\neq x^{\left\{ n\right\} }y^{2}$
on a $\pi\left(w\right)=P_{w}$ et pour tout $f\in\bigoplus_{n\geq2}K\left(x,y\right)_{\left[n,2\right]}$,
le polynôme $f-\pi\left(f\right)$ est une train identité évanescente.
\end{thm}

\begin{proof}
On a montré aux théorèmes \ref{thm:Evanesc_type=00005Bn=00005D_id=0000E9al}
et \ref{thm:Evanesc_type=00005Bn,1=00005D_id=0000E9al} que les
polynômes $E_{p,q}^{\left[n\right]}$, $E_{p,q}^{\left[n,1\right]}$,
$E_{p,\left\{ r\right\} }^{\left[n,1\right]}$, $E_{\left\{ r\right\} ,p}^{\left[n,1\right]}$
et $E_{p,\left\{ r\right\} }^{\left[n,2\right]}$ sont évanescents,
montrons que c'est aussi le cas pour les autres polynômes de
l'énoncé.

Pour $p\geq2$ et $r\geq0$ on a $\partial_{x}\left(x^{p}\bigl(\bigl(x^{\left\{ r\right\} }y\bigr)y\bigr)\right)=t\left(\partial_{x}\left(x^{p}\right)+\partial_{x}\bigl(x^{\left\{ r\right\} }y\bigr)\right)$
et $\partial_{y}\left(x^{p}\bigl(\bigl(x^{\left\{ r\right\} }y\bigr)y\bigr)\right)=t\partial_{y}\left(\bigl(x^{\left\{ r\right\} }y\bigr)y\right)$
et avec les relations (\ref{eq:Dx(x^p*y)}), (\ref{eq:Dx(x^=00007Bp=00007D*y)})
et (\ref{eq:Dx(x^=00007Br=00007D*y*y)}) on obtient:
\begin{align*}
\partial_{x}\left(x^{p}\bigl(\bigl(x^{\left\{ r\right\} }y\bigr)y\bigr)\right) & =2t^{p}+\sum_{i=2}^{p-1}t^{i}+\sum_{i=3}^{r+2}t^{i}, & \partial_{y}\left(x^{p}\bigl(\bigl(x^{\left\{ r\right\} }y\bigr)y\bigr)\right) & =\sum_{i=2}^{r+2}t^{i}.
\end{align*}

Pour tout $r,s\geq0$ on a $\partial_{x}\left(x^{\left\{ r\right\} }\bigl(\bigl(x^{\left\{ s\right\} }y\bigr)y\bigr)\right)=t\left(1+\partial_{x}\left(x^{\left\{ r-1\right\} }\bigl(\bigl(x^{\left\{ s\right\} }y\bigr)y\bigr)\right)\right)$
on en déduit que $\partial_{x}\left(x^{\left\{ r\right\} }\bigl(\bigl(x^{\left\{ s\right\} }y\bigr)y\bigr)\right)=\sum_{i=1}^{r}t^{i}+t^{r}\partial_{x}\left(\bigl(x^{\left\{ s\right\} }y\bigr)y\right)$
et $\partial_{y}\left(x^{\left\{ r\right\} }\bigl(\bigl(x^{\left\{ s\right\} }y\bigr)y\bigr)\right)=t\partial_{y}\left(x^{\left\{ r-1\right\} }\bigl(\bigl(x^{\left\{ s\right\} }y\bigr)y\bigr)\right)$
d'où $\partial_{y}\left(x^{\left\{ r\right\} }\bigl(\bigl(x^{\left\{ s\right\} }y\bigr)y\bigr)\right)=t^{r}\partial_{y}\left(\bigl(x^{\left\{ s\right\} }y\bigr)y\right)$,
en se servant des relations (\ref{eq:Dx(x^=00007Br=00007D*y*y)})
on a:
\begin{align*}
\partial_{x}\left(x^{\left\{ r\right\} }\bigl(\bigl(x^{\left\{ s\right\} }y\bigr)y\bigr)\right) & =\sum_{i=1}^{r+s+1}t^{i}-t^{r+1}, & \partial_{y}\left(x^{\left\{ r\right\} }\bigl(\bigl(x^{\left\{ s\right\} }y\bigr)y\bigr)\right) & =t^{r+s+1}+t^{r+1}.
\end{align*}

Pour $p\geq2$ et $r\geq0$ on a $\partial_{x}\left(x^{p}\bigl(x^{\left\{ r\right\} }y^{2}\bigr)\right)=t\left(\partial_{x}\left(x^{p}\right)+\partial_{x}\left(x^{\left\{ r\right\} }y^{2}\right)\right)$
et $\partial_{y}\left(x^{p}\bigl(x^{\left\{ r\right\} }y^{2}\bigr)\right)=t\partial_{y}\left(x^{\left\{ r\right\} }y^{2}\right)$
on en déduit avec les relations (\ref{eq:Dx(x^=00007Br=00007D*y2)}):
\begin{align*}
\partial_{x}\left(x^{p}\bigl(\bigl(x^{\left\{ r\right\} }y\bigr)y\bigr)\right) & =\sum_{i=1}^{r+s+1}t^{i}-t^{r+1}, & \partial_{y}\left(x^{p}\bigl(\bigl(x^{\left\{ r\right\} }y\bigr)y\bigr)\right) & =t^{r+s+1}+t^{r+1}.
\end{align*}

Pour tout entier $p,q\geq2$ on a $\partial_{x}\left(\left(x^{p}y\right)\left(x^{q}y\right)\right)=t\left(\partial_{x}\left(x^{p}y\right)+\partial_{x}\left(x^{q}y\right)\right)=t^{2}\left(\partial_{x}\left(x^{p}\right)+\partial_{x}\left(x^{q}\right)\right)$
et $\partial_{y}\left(\left(x^{p}y\right)\left(x^{q}y\right)\right)=t\left(\partial_{y}\left(x^{p}y\right)+\partial_{y}\left(x^{q}y\right)\right)=2t^{2}\partial_{y}\left(y\right)$
on en déduit avec la relation (\ref{eq:Dx(x^k)}) que 
\begin{align*}
\partial_{x}\left(\left(x^{p}y\right)\left(x^{q}y\right)\right) & =2t^{p+1}+2t^{q+1}+\sum_{i=3}^{p}t^{i}+\sum_{i=3}^{q}t^{i}, & \partial_{y}\left(\left(x^{p}y\right)\left(x^{q}y\right)\right) & =2t^{2}.
\end{align*}

Soient $p\geq2$ et $r\geq0$, on a $\partial_{x}\left(\left(x^{p}y\right)\left(x^{\left\{ r\right\} }y\right)\right)=t\left(\partial_{x}\left(x^{p}y\right)+\partial_{x}\left(x^{\left\{ r\right\} }y\right)\right)$
et $\partial_{y}\left(\left(x^{p}y\right)\left(x^{\left\{ r\right\} }y\right)\right)=t\left(\partial_{y}\left(x^{p}y\right)+\partial_{y}\left(x^{\left\{ r\right\} }y\right)\right)$,
en utilisant les relations (\ref{eq:Dx(x^p*y)}) et (\ref{eq:Dx(x^=00007Bp=00007D*y)})
on obtient:
\begin{align*}
\partial_{x}\left(\left(x^{p}y\right)\left(x^{\left\{ r\right\} }y\right)\right) & =2t^{p+1}+\sum_{i=3}^{p}t^{i}+\sum_{i=2}^{r+1}t^{i}, & \partial_{y}\left(\left(x^{p}y\right)\left(x^{\left\{ r\right\} }y\right)\right) & =t^{r+1}+t^{2}.
\end{align*}

Si $r,s\geq0$, de $\partial_{x}\left(\bigl(x^{\left\{ r\right\} }y\bigr)\bigl(x^{\left\{ s\right\} }y\bigr)\right)=t\left(\partial_{x}\left(x^{\left\{ r\right\} }y\right)+\partial_{x}\left(x^{\left\{ r\right\} }y\right)\right)$
et $\partial_{y}\left(\bigl(x^{\left\{ r\right\} }y\bigr)\bigl(x^{\left\{ s\right\} }y\bigr)\right)=t\left(\partial_{y}\left(x^{\left\{ r\right\} }y\right)+\partial_{y}\left(x^{\left\{ r\right\} }y\right)\right)$
et de la relation (\ref{eq:Dx(x^p*(x^=00007Br=00007D*y))}) on
déduit
\begin{align*}
\partial_{x}\left(\left(x^{\left\{ r\right\} }y\right)\left(x^{\left\{ s\right\} }y\right)\right) & =\sum_{i=2}^{r+1}t^{i}+\sum_{i=2}^{s+1}t^{i}, & \partial_{y}\left(\left(x^{\left\{ r\right\} }y\right)\left(x^{\left\{ s\right\} }y\right)\right) & =t^{r+1}+t^{s+1}.
\end{align*}

En utilisant ces résultats et les relations du lemme \ref{lem:Dx_ds_M=00005Bn,2=00005D}
on établit que les polynômes $E_{p,\left\{ r\right\} }^{\left[n,2\right]}$,
$E_{\left\{ r\right\} ,\left\{ s\right\} }^{\left[n,2\right]}$,
$F_{p,\left\{ r\right\} }^{\left[n,2\right]}$, $G_{p,q}^{\left[n,2\right]}$,
$G_{p,\left\{ r\right\} }^{\left[n,2\right]}$ et $G_{\left\{ r\right\} ,\left\{ s\right\} }^{\left[n,2\right]}$
son évanescents.

Soit $w\in\frak{M}\left(x,y\right)_{\left[n,2\right]}$ tel que
$w\neq\left(x^{\left\{ n\right\} }y\right)y$ et $w\neq x^{\left\{ n\right\} }y^{2}$,
montrons que le polynôme $w-\pi\left(w\right)$ est évanescent.
Par récurrence sur le degré $n$ en $x$ de $w$, ce résultat
est vrai pour $n=2$ comme on le voit sur les générateurs des
train identités évanescentes de degré $\left(2,2\right)$ donnés
ci-dessous. Supposons le résultat vrai pour tout $v\in\frak{M}\left(x,y\right)_{\left[p,2\right]}$
et tout $2\leq p<n$. Il existe $u,v\in\mathfrak{M}\left(x,y\right)$
tels que $w=uv$ avec $u\in\frak{M}\left(x,y\right)_{\left[n-k,1\right]}$,
$v\in\frak{M}\left(x,y\right)_{\left[k,1\right]}$ ou $u\in\frak{M}\left(x,y\right)_{\left[n-k\right]}$
et $v\in\frak{M}\left(x,y\right)_{\left[k,2\right]}$ avec $1<k<n$.
On a $\partial_{x}\left(w-\pi\left(w\right)\right)=\partial_{x}\left(uv-\pi\left(u\right)\pi\left(v\right)\right)=t\left(\partial_{x}\left(u-\pi\left(u\right)\right)+\partial_{x}\left(v-\pi\left(v\right)\right)\right)$,
de même on a $\partial_{y}\left(w-\pi\left(w\right)\right)=t\partial_{y}\left(u-\pi\left(u\right)\right)+t\partial_{y}\left(v-\pi\left(v\right)\right)$. 

Par conséquent si $u\in\frak{M}\left(x,y\right)_{\left[n-k,1\right]}$,
$v\in\frak{M}\left(x,y\right)_{\left[k,1\right]}$, il résulte
du théorème \ref{thm:Evanesc_type=00005Bn,1=00005D_id=0000E9al}
que le polynôme $w-\pi\left(w\right)$ est évanescent. 

Dans le cas $u\in\frak{M}\left(x,y\right)_{\left[n-k\right]}$
et $v\in\frak{M}\left(x,y\right)_{\left[k,2\right]}$, avec le
théorème \ref{thm:Evanesc_type=00005Bn=00005D_id=0000E9al} on
a que $u-\pi\left(u\right)$ est évanescent et par hypothèse
de récurrence il en est de même du polynôme $v-\pi\left(v\right)$. 

Il est clair que pour tout $w\in\frak{M}\left(x,y\right)_{\left[n,2\right]}$
on a $\pi\left(w\right)\in\mathbb{Z}\left\langle \mathscr{F}\right\rangle $
alors par unicité du polynôme $P_{w}$ on a $\pi\left(w\right)=P_{w}$.
\end{proof}
En appliquant ce théorème on obtient les générateurs des train
identités évanescentes:
\begin{align*}
 & \text{-- de degré }\left(2,2\right)\\
 & \;x^{2}y^{2}-2\left(xy\right)y+y^{2};\\
 & \left(xy\right)^{2}-2\left(xy\right)y+y^{2};\\
 & \;x\left(\left(xy\right)y\right)-\left(x\left(xy\right)\right)y+\left(xy\right)y-xy^{2};\\
 & \left(x^{2}y\right)y-2\left(x\left(xy\right)\right)y+\left(xy\right)y+x\left(xy^{2}\right)-xy^{2}.\\
\\
 & \text{-- de degré }\left(3,2\right)\\
 & \;x^{2}\left(\left(xy\right)y\right)-\left(x\left(xy\right)\right)y-\left(xy\right)y+y^{2};\\
 & \;x\left(xy\right)^{2}-2\left(x\left(xy\right)\right)y+2\left(xy\right)y-xy^{2};\\
 & \;x\left(x^{2}y^{2}\right)-2\left(x\left(xy\right)\right)y+2\left(xy\right)y-xy^{2};\\
 & \left(x\left(xy\right)\right)\left(xy\right)-\left(x\left(xy\right)\right)y-\left(xy\right)y+y^{2};\\
 & \;x^{3}y^{2}-2\left(x\left(xy\right)\right)y+x\left(xy^{2}\right)-xy^{2}+y^{2};\\
 & \;x^{2}\left(xy^{2}\right)-2\left(xy\right)y-x\left(xy^{2}\right)+xy^{2}+y^{2};\\
 & \;x\left(\left(x\left(xy\right)\right)y\right)-x\left(x\left(xy\right)\right)y+\left(xy\right)y-xy^{2};\\
 & \left(xy\right)\left(x^{2}y\right)-2\left(x\left(xy\right)\right)y+x\left(xy^{2}\right)-xy^{2}+y^{2};\\
 & \left(x^{2}\left(xy\right)\right)y-3\left(x\left(xy\right)\right)y+2\left(xy\right)y+x\left(xy^{2}\right)-xy^{2};\\
 & \;x\left(x\left(\left(xy\right)y\right)\right)-\left(x\left(x\left(xy\right)\right)\right)y+\left(x\left(xy\right)\right)y-x\left(xy^{2}\right);\\
 & \left(x^{3}y\right)y-2\left(x\left(x\left(xy\right)\right)\right)y+\left(xy\right)y+x\left(x\left(xy^{2}\right)\right)-xy^{2};\\
 & \left(x\left(x^{2}y\right)\right)y-2\left(x\left(x\left(xy\right)\right)\right)y+\left(x\left(xy\right)\right)y+x\left(x\left(xy^{2}\right)\right)-x\left(xy^{2}\right);\\
 & \;x\left(\left(x^{2}y\right)y\right)-2x\left(x\left(xy\right)\right)y+\left(x\left(xy\right)\right)y+\left(xy\right)y+x\left(x\left(xy^{2}\right)\right)-x\left(xy^{2}\right)-xy^{2}.\\
\\
 & \text{-- de degré }\left(4,2\right)\\
 & \;x^{2}\left(xy\right)^{2}-2\left(x\left(xy\right)\right)y+y^{2};\\
 & \;x^{3}\left(xy^{2}\right)-2\left(x\left(xy\right)\right)y+y^{2};\\
 & \left(x\left(xy\right)\right)^{2}-2\left(x\left(xy\right)\right)y+y^{2};\\
 & \;x^{2}\left(x^{2}y^{2}\right)-2\left(x\left(xy\right)\right)y+y^{2};\\
 & \;x^{2}(\left(x\left(xy\right)y\right)-\left(x\left(x\left(xy\right)\right)\right)y-\left(xy\right)y+y^{2};\\
 & \;x^{4}y^{2}-2\left(x\left(x\left(xy\right)\right)\right)y+x\left(x\left(xy^{2}\right)\right)-xy^{2}+y^{2};\\
 & \;x^{2}\left(x\left(xy^{2}\right)\right)-2\left(xy\right)y-x\left(x\left(xy^{2}\right)\right)+xy^{2}+y^{2};\\
 & \left(x^{3}y\right)\left(xy\right)-2\left(x\left(x\left(xy\right)\right)\right)y+x\left(x\left(xy^{2}\right)\right)-xy^{2}+y^{2};\\
 & \;x^{3}\left(\left(xy\right)y\right)-3\left(x\left(xy\right)\right)y+\left(xy\right)y+x\left(xy^{2}\right)-xy^{2}+y^{2};\\
 & \left(x^{2}y\right)^{2}-4\left(x\left(xy\right)\right)y+2\left(xy\right)y+2x\left(xy^{2}\right)-2xy^{2}+y^{2};\\
 & \left(x^{2}x^{2}\right)y^{2}-4\left(x\left(xy\right)\right)y+2x\left(xy\right)+2x\left(xy^{2}\right)-2xy^{2}+y^{2};\\
 & \left(x^{2}y\right)\left(x\left(xy\right)\right)-3\left(x\left(xy\right)\right)y+\left(xy\right)y+x\left(xy^{2}\right)-xy^{2}+y^{2};\\
 & \left(x^{2}\left(xy\right)\right)\left(xy\right)-3\left(x\left(xy\right)\right)y+\left(xy\right)y+x\left(xy^{2}\right)-xy^{2}+y^{2};\\
 & \left(x^{4}y\right)y-2\left(x\left(x\left(x\left(xy\right)\right)\right)\right)y+x\left(x\left(x\left(xy^{2}\right)\right)\right)+\left(xy\right)y-xy^{2};\\
 & \left(x\left(x\left(xy\right)\right)\right)\left(xy\right)-\left(x\left(x\left(xy\right)\right)\right)y-\left(xy\right)y+y^{2};\\
 & \;x\left(y\left(x\left(x\left(xy\right)\right)\right)\right)-\left(x\left(x\left(x\left(xy\right)\right)\right)\right)y+\left(xy\right)y-xy^{2};\\
 & \;x\left(x\left(xy\right)^{2}\right)-2\left(x\left(x\left(xy\right)\right)\right)y+2\left(x\left(xy\right)\right)y-x\left(xy^{2}\right);\\
 & \;x\left(x\left(x^{2}y^{2}\right)\right)-2\left(x\left(x\left(xy\right)\right)\right)y+2\left(x\left(xy\right)\right)y-x\left(xy^{2}\right);\\
 & \;x\left(x\left(y\left(x\left(xy\right)\right)\right)\right)-\left(x\left(x\left(x\left(xy\right)\right)\right)\right)y+\left(x\left(xy\right)\right)y-x\left(xy^{2}\right);\\
 & \;x\left(x^{2}\left(\left(xy\right)y\right)\right)-\left(x\left(x\left(xy\right)\right)\right)y-\left(x\left(xy\right)\right)y+2\left(xy\right)y-xy^{2};\\
 & \;x\left(\left(xy\right)\left(x\left(xy\right)\right)\right)-\left(x\left(x\left(xy\right)\right)\right)y-\left(x\left(xy\right)\right)y+2\left(xy\right)y-xy^{2};\\
 & \;x\left(x\left(x\left(\left(xy\right)y\right)\right)\right)-\left(x\left(x\left(x\left(xy\right)\right)\right)\right)y+\left(x\left(x\left(xy\right)\right)\right)y-x\left(x\left(xy^{2}\right)\right);\\
 & \;x\left(x^{3}y^{2}\right)-2\left(x\left(x\left(xy\right)\right)\right)y+2\left(xy\right)y+x\left(x\left(xy^{2}\right)\right)-x\left(xy^{2}\right)-xy^{2};\\
 & \;x\left(x^{2}\left(xy^{2}\right)\right)-2\left(x\left(xy\right)\right)y+2\left(xy\right)y-x\left(x\left(xy^{2}\right)\right)+x\left(xy^{2}\right)-xy^{2};\\
 & \left(x\left(x^{2}\left(xy\right)\right)\right)y-3\left(x\left(x\left(xy\right)\right)\right)y+2\left(x\left(xy\right)\right)y+x\left(x\left(xy^{2}\right)\right)-x\left(xy^{2}\right);\\
 & \left(x^{3}\left(xy\right)\right)y-2\left(x\left(x\left(xy\right)\right)\right)y-\left(x\left(xy\right)\right)y+2\left(xy\right)y+x\left(x\left(xy^{2}\right)\right)-xy^{2};\\
 & \left(x^{2}\left(x\left(xy\right)\right)\right)y-\left(x\left(x\left(xy\right)\right)\right)y-2\left(x\left(xy\right)\right)y+2\left(xy\right)y+x\left(xy^{2}\right)-xy^{2};\\
 & \left(\left(x^{2}x^{2}\right)y\right)y-2\left(x\left(x\left(xy\right)\right)\right)y-\left(x\left(xy\right)\right)y+2\left(xy\right)y+x\left(x\left(xy^{2}\right)\right)-xy^{2};\\
 & \left(x\left(x^{3}y\right)\right)y-2\left(x\left(x\left(x\left(xy\right)\right)\right)\right)y+x\left(x\left(x\left(xy^{2}\right)\right)\right)+\left(x\left(xy\right)\right)y-x\left(xy^{2}\right);\\
 & \;x\left(\left(xy\right)\left(x^{2}y\right)\right)-2\left(x\left(x\left(xy\right)\right)\right)y+2\left(xy\right)y+x\left(x\left(xy^{2}\right)\right)-x\left(xy^{2}\right)-xy^{2};\\
 & \left(\left(x^{2}x^{2}\right)y\right)y-2\left(x\left(x\left(xy\right)\right)\right)y-\left(x\left(xy\right)\right)y+2\left(xy\right)y+x\left(x\left(xy^{2}\right)\right)-xy^{2};\\
 & \;x^{2}\left(x\left(\left(xy\right)y\right)\right)-\left(x\left(x\left(xy\right)\right)\right)y+\left(x\left(xy\right)\right)y-2\left(xy\right)y-x\left(xy^{2}\right)+xy^{2}+y^{2};\\
 & \;x^{2}\left(\left(x^{2}y\right)y\right)-2\left(x\left(x\left(xy\right)\right)\right)y+\left(x\left(xy\right)\right)y-\left(xy\right)y+x\left(x\left(xy^{2}\right)\right)-x\left(xy^{2}\right)+y^{2};\\
 & \left(xy\right)\left(x\left(x^{2}y\right)\right)-2\left(x\left(x\left(xy\right)\right)\right)y+\left(x\left(xy\right)\right)y-\left(xy\right)y+x\left(x\left(xy^{2}\right)\right)-x\left(xy^{2}\right)+y^{2};\\
 & \left(x\left(x\left(x^{2}y\right)\right)\right)y-2\left(x\left(x\left(x\left(xy\right)\right)\right)\right)y+\left(x\left(x\left(xy\right)\right)\right)y+x\left(x\left(x\left(xy^{2}\right)\right)\right)-x\left(x\left(xy^{2}\right)\right);\\
 & \;x\left(\left(x^{2}\left(xy\right)\right)y\right)-3\left(x\left(x\left(xy\right)\right)\right)y+2\left(x\left(xy\right)\right)y+\left(xy\right)y+x\left(x\left(xy^{2}\right)\right)-x\left(xy^{2}\right)-xy^{2};\\
 & \;x\left(\left(x^{3}y\right)y\right)-2\left(x\left(x\left(x\left(xy\right)\right)\right)\right)y+\left(x\left(xy\right)\right)y+\left(xy\right)y+x\left(x\left(x\left(xy^{2}\right)\right)\right)-x\left(xy^{2}\right)-xy^{2};\\
 & \;x\left(y\left(x\left(x^{2}y\right)\right)\right)-2\left(x\left(x\left(x\left(xy\right)\right)\right)\right)y+\left(x\left(x\left(xy\right)\right)\right)y+\left(xy\right)y+x\left(x\left(x\left(xy^{2}\right)\right)\right)\\
 & \qquad-x\left(x\left(xy^{2}\right)\right)-xy^{2};\\
 & \;x\left(x\left(\left(x^{2}y\right)y\right)\right)-2\left(x\left(x\left(x\left(xy\right)\right)\right)\right)y+\left(x\left(x\left(xy\right)\right)\right)y+\left(x\left(xy\right)\right)y+x\left(x\left(x\left(xy^{2}\right)\right)\right)\\
 & \qquad-x\left(x\left(xy^{2}\right)\right)-x\left(xy^{2}\right).
\end{align*}

\medskip{}

\subsubsection{Identités homogènes évanescentes de type $\left[n,2\right]$.}
\begin{prop}
Pour $n\geq2$, l'espace des identités homogènes évanescentes
de type $\left[n,2\right]$ est engendré par au moins $W_{\left[n,2\right]}-2n$
identités homogènes évanescentes.
\end{prop}

\begin{proof}
Soit $n\geq2$, on note $N=W_{\left[n,2\right]}$. Soient $\frak{M}\left(x,y\right)_{\left[n,2\right]}=\left\{ w_{k};1\leq k\leq N\right\} $
et $f=\sum_{i=1}^{N}\alpha_{i}w_{i}$, on cherche $\left(\alpha_{i}\right)_{1\leq i\leq N}\in K^{N}$
tel que $f\left(1,1\right)=1$ et $\partial_{x}f=\partial_{y}f=0$
où $\partial_{x}f=\sum_{i=1}^{N}\alpha_{i}\partial_{x}w_{i}$
et $\partial_{y}f=\sum_{i=1}^{N}\alpha_{i}\partial_{y}w_{i}$.
Pour tout $1\leq i\leq N$ on a $\partial_{x}w_{i},\partial_{y}w_{i}\in K\left[t\right]$
avec $\left|\partial_{x}w_{i}\right|\leq n+1$ et $\left|\partial_{y}w_{i}\right|\leq n+1$,
de plus il existe $1\leq i,j\leq N$ tels que $\left|\partial_{x}w_{i}\right|=n+1$
et $\left|\partial_{y}w_{j}\right|=n+1$, par conséquent les
solutions $\left(\alpha_{i}\right)_{1\leq i\leq N}$ des équations
$\partial_{x}f=0$ et $\partial_{y}f=0$ sont solutions de deux
systèmes linéaires de $n+1$ équations à $N$ inconnues. De $f\left(1,1\right)=1$
on déduit que $\sum_{i=1}^{N}\alpha_{i}=0$, par conséquent le
système d'équations $\partial_{x}f=0$ est de rang $\leq n$,
il en est de même du système $\partial_{y}f=0$, donc le système
d'équations $\partial_{x}f=\partial_{y}f=0$ est de rang $\leq2n$
par conséquent l'espace des solutions est de dimension $\geq W_{\left[n,2\right]}-2n$.
\end{proof}
En utilisant la méthode utilisée dans la démonstration on obtient
les générateurs des polynômes homogènes évanescents:
\begin{align*}
 & \text{-- de type }\left[2,2\right]\\
 & \;x^{2}y^{2}-\left(xy\right)^{2};\\
 & \left(x^{2}y\right)y+x\left(xy^{2}\right)-x\left(\left(xy\right)y\right)-\left(x\left(xy\right)\right)y;\\
\\
 & \text{-- de type }\left[3,2\right]\\
 & \;x^{3}y^{2}-\left(xy\right)\left(x^{2}y\right);\\
 & \;x\left(x^{2}y^{2}\right)-x\left(xy\right)^{2};\\
 & \;x^{2}\left(\left(xy\right)y\right)-\left(xy\right)\left(x\left(xy\right)\right);\\
 & \;x^{3}y^{2}+x^{2}\left(xy^{2}\right)-2\left(xy\right)\left(x\left(xy\right)\right);\\
 & \;x^{3}y^{2}-\left(x^{3}y\right)y+\left(x\left(x^{2}y\right)\right)y-\left(xy\right)\left(x\left(xy\right)\right);\\
 & \left(x^{3}y\right)y-\left(x^{2}\left(xy\right)\right)y+x\left(x^{2}y^{2}\right)-x\left(\left(x^{2}y\right)y\right);\\
 & \left(x^{3}y\right)y+x\left(x\left(xy^{2}\right)\right)-x\left(\left(x\left(xy\right)\right)y\right)-\left(x\left(x\left(xy\right)\right)\right)y;\\
 & \left(x^{2}\left(xy\right)\right)y-\left(x\left(x\left(xy\right)\right)\right)y-x\left(x^{2}y^{2}\right)+x\left(x\left(\left(xy\right)y\right)\right);\\
 & \;x^{3}y^{2}+x\left(x^{2}y^{2}\right)-\left(x^{2}\left(xy\right)\right)y-\left(x\left(xy\right)\right)\left(xy\right)-x\left(y\left(x\left(xy\right)\right)\right)+\left(x\left(x\left(xy\right)\right)\right)y;\\
\\
 & \text{-- de type }\left[4,2\right]\\
 & \left(x\left(xy\right)\right)^{2}-x^{2}\left(x^{2}y^{2}\right);\\
 & \;x^{3}\left(xy^{2}\right)-x^{2}\left(x^{2}y^{2}\right);\\
 & \;x^{2}\left(xy\right)^{2}-x^{2}\left(x^{2}y^{2}\right);\\
 & \;x\left(x^{3}y^{2}\right)-x\left(\left(xy\right)\left(x^{2}y\right)\right);\\
 & \left(x^{3}\left(xy\right)\right)y-\left(x^{2}\left(x^{2}y\right)\right)y;\\
 & \;x\left(x\left(x^{2}y^{2}\right)\right)-x\bigl(x\left(xy\right)^{2}\bigr);\\
 & \left(x^{2}y\right)\left(x\left(xy\right)\right)-\left(xy\right)\left(x^{2}\left(xy\right)\right);\\
 & \;x^{2}\left(\left(x\left(xy\right)\right)y\right)-\left(xy\right)\left(x\left(x\left(xy\right)\right)\right);\\
 & \;x\left(x^{2}\left(\left(xy\right)y\right)\right)-x\left(\left(xy\right)\left(x\left(xy\right)\right)\right);\\
 & \;x\left(x^{3}y^{2}\right)-2x\left(\left(xy\right)\left(x\left(xy\right)\right)\right)+x\left(x^{2}\left(xy^{2}\right)\right);\\
 & \;x^{2}\left(x^{2}y^{2}\right)-x\left(x^{3}y^{2}\right)-x^{2}\left(x\left(xy^{2}\right)\right)+x\left(x\left(x^{2}y^{2}\right)\right);\\
 & \;x\left(x^{3}y^{2}\right)-x^{2}\left(x^{2}y^{2}\right)-x\left(\left(x^{2}\left(xy\right)\right)y\right)+x^{2}\left(\left(x\left(xy\right)\right)y\right);\\
 & \left(x^{4}y\right)y+x\left(\left(xy\right)\left(x\left(xy\right)\right)\right)-\left(x^{2}\left(x\left(xy\right)\right)\right)y-x\left(\left(x^{3}y\right)y\right);\\
 & \left(x\left(x^{3}y\right)\right)y+x\left(\left(xy\right)\left(x\left(xy\right)\right)\right)-x\left(x^{3}y^{2}\right)-\left(x\left(x\left(x^{2}y\right)\right)\right)y;\\
 & \;x\left(x^{3}y^{2}\right)+x\left(\left(x\left(x^{2}y\right)\right)y\right)-x\left(\left(x^{3}y\right)y\right)-x\left(\left(xy\right)\left(x\left(xy\right)\right)\right);\\
 & \;x\left(x^{3}y^{2}\right)+x^{2}\left(x\left(xy^{2}\right)\right)-x^{2}\left(x\left(\left(xy\right)y\right)\right)-x\left(\left(xy\right)\left(x\left(xy\right)\right)\right);\\
 & \;x^{2}\left(x^{2}y^{2}\right)-x^{2}\left(\left(x^{2}y\right)y\right)-x\left(\left(xy\right)\left(x\left(xy\right)\right)\right)+x\left(\left(x^{2}\left(xy\right)\right)y\right);\\
 & \left(\left(x^{2}x^{2}\right)y\right)y-\left(x^{3}\left(xy\right)\right)y-x\left(y\left(x^{2}\left(xy\right)\right)\right)+x\left(\left(xy\right)\left(x\left(xy\right)\right)\right);\\
 & \;x^{2}\left(x^{2}y^{2}\right)-\left(xy\right)\left(x\left(x^{2}y\right)\right)-x\left(\left(xy\right)\left(x\left(xy\right)\right)\right)+x\left(y\left(x^{2}\left(xy\right)\right)\right);\\
 & \left(x\left(x^{3}y\right)\right)y-x\left(\left(x^{3}y\right)y\right)+x\left(\left(x\left(x\left(xy\right)\right)\right)y\right)-\left(x\left(x\left(x\left(xy\right)\right)\right)\right)y;\\
 & \;x^{4}y^{2}+2x\left(x^{3}y^{2}\right)-2x^{2}\left(x^{2}y^{2}\right)-2x\left(\left(x^{2}\left(xy\right)\right)y\right)+x^{2}\left(x\left(xy^{2}\right)\right);\\
 & \;x\left(\left(x^{3}y\right)y\right)+x\left(x\left(x\left(xy^{2}\right)\right)\right)-x\left(x\left(y\left(x\left(xy\right)\right)\right)\right)-x\left(y\left(x\left(x\left(xy\right)\right)\right)\right);\\
 & \left(x^{3}y\right)\left(xy\right)+2x\left(x^{3}y^{2}\right)+x^{2}\left(x\left(xy^{2}\right)\right)-2x^{2}\left(x^{2}y^{2}\right)-2x\left(\left(x^{2}\left(xy\right)\right)y\right);\\
 & \;x\left(x^{3}y^{2}\right)+\left(x^{2}\left(x\left(xy\right)\right)\right)y+x\left(\left(x^{2}\left(xy\right)\right)y\right)-\left(\left(x^{2}x^{2}\right)y\right)y-2x\left(\left(xy\right)\left(x\left(xy\right)\right)\right);\\
 & \;x\left(x^{3}y^{2}\right)+x^{2}\left(x\left(xy^{2}\right)\right)+x\left(\left(x^{3}y\right)y\right)-x^{2}\left(x^{2}y^{2}\right)-x\left(x\left(\left(x^{2}y\right)y\right)\right)-x\left(\left(x^{2}\left(xy\right)\right)y\right);\\
 & \;x^{2}\left(x^{2}y^{2}\right)+\left(x\left(x^{2}\left(xy\right)\right)\right)y+x\left(\left(x^{2}\left(xy\right)\right)y\right)-x\left(x^{3}y^{2}\right)-x^{2}\left(x\left(xy^{2}\right)\right)-\left(\left(x^{2}x^{2}\right)y\right)y;\\
 & \;2x\left(x^{3}y^{2}\right)+x^{3}\left(\left(xy\right)y\right)+x^{2}\left(x\left(xy^{2}\right)\right)-2x^{2}\left(x^{2}y^{2}\right)-x\left(\left(x^{2}\left(xy\right)\right)y\right)-x\left(\left(xy\right)\left(x(xy\right)\right);\\
 & \left(x^{2}y\right)^{2}+4x\left(x^{3}y^{2}\right)+2x^{2}\left(x\left(xy^{2}\right)\right)-3x^{2}\left(x^{2}y^{2}\right)-2x\left(\left(x^{2}\left(xy\right)\right)y\right)-2x\left(\left(xy\right)\left(x\left(xy\right)\right)\right);\\
 & \;x^{2}\left(x^{2}y^{2}\right)+x\left(\left(x^{2}\left(xy\right)\right)y\right)+x\left(x\left(x\left(\left(xy\right)y\right)\right)\right)-x\left(x^{3}y^{2}\right)-x^{2}\left(x\left(xy^{2}\right)\right)-x\left(\left(x\left(x\left(xy\right)\right)\right)y\right);\\
 & \left(x^{2}x^{2}\right)y^{2}+4x\left(x^{3}y^{2}\right)+2x^{2}\left(x\left(xy^{2}\right)\right)-3x^{2}\left(x^{2}y^{2}\right)-2x\left(\left(x^{2}\left(xy\right)\right)y\right)-2x\left(\left(xy\right)\left(x\left(xy\right)\right)\right);\\
 & \;2x\left(x^{3}y^{2}\right)+x^{2}\left(x\left(xy^{2}\right)\right)+\left(x^{2}y\right)\left(x\left(xy\right)\right)-2x^{2}\left(x^{2}y^{2}\right)-x\left(\left(x^{2}\left(xy\right)\right)y\right)-x\left(\left(xy\right)\left(x\left(xy\right)\right)\right);\\
 & \;x^{2}\left(x^{2}y^{2}\right)+\left(x\left(x^{3}y\right)\right)y+2x\left(\left(x^{2}\left(xy\right)\right)y\right)-x\left(x^{3}y^{2}\right)-x^{2}\left(x\left(xy^{2}\right)\right)-\left(\left(x^{2}x^{2}\right)y\right)y\\
 & \qquad-x\left(\left(x^{3}y\right)y\right);\\
 & \;x^{2}\left(x^{2}y^{2}\right)+x\left(\left(x^{2}\left(xy\right)\right)y\right)+x\left(\left(xy\right)\left(x\left(xy\right)\right)\right)+x\left(x\left(y\left(x\left(xy\right)\right)\right)\right)-2x\left(x^{3}y^{2}\right)\\
 & \qquad-x^{2}\left(x\left(xy^{2}\right)\right)-x\left(y\left(x\left(x\left(xy\right)\right)\right)\right);
\end{align*}

\subsection{Identités évanescentes train de degré $\left(n,1,1\right)$
et homogènes de type $\left[n,1,1\right]$.}

\textcompwordmark{}

On peut écrire tout $w\in\frak{M}\left(x,y,z\right)_{\left[n,1,1\right]}$
sous la forme $w=w_{1}w_{2}$ avec $\left(w_{1},w_{2}\right)\in\frak{M}\left(x\right)_{\left[n,-i\right]}\times\frak{M}\left(x,y,z\right)_{\left[i,1,1\right]}$
pour $0\leq i\leq n-1$, ou $\left(w_{1},w_{2}\right)\in\frak{M}\left(x,y\right)_{\left[n-i,1\right]}\times\frak{M}\left(x,y,z\right)_{\left[i,0,1\right]}$
avec $0\leq i\leq n$, comme $\sharp\frak{M}\left(x,y,z\right)_{\left[i,0,1\right]}=\sharp\frak{M}\left(x,z\right)_{\left[i,1\right]}$
on a donc
\begin{align*}
W_{\left[n,1,1\right]} & =\sum_{i=0}^{n-1}W_{\left[n-i\right]}W_{\left[i,1,1\right]}+\sum_{i=0}^{n}W_{\left[n-i,1\right]}W_{\left[i,1\right]},\quad W_{\left[0,1,1\right]}=1.
\end{align*}

Les premières valeurs de $W_{\left[n,1,1\right]}$ sont
\begin{center}
\begin{tabular}{c|c|c|c|c|c|c|c|c|c|c|c}
$n$ & 0 & 1 & 2 & 3 & 4 & 5 & 6 & 7 & 8 & 9 & 10\tabularnewline
\hline 
$W_{\left[n,1,1\right]}$ & 1 & 3 & 9 & 25 & 69 & 186 & 497 & 1314 & 3453 & 9019 & 23454\tabularnewline
\end{tabular}
\par\end{center}

\begin{center}
\medskip{}
\par\end{center}

\subsubsection{Train identités évanescentes de degré $\left(n,1,1\right)$.}

\textcompwordmark{}\medskip{}

Pour tout $f\in K\left(x,y,z\right)$ et tout entier $r\geq1$,
on définit $x^{\left\{ r\right\} }f=x\left(x^{\left\{ r-1\right\} }f\right)$
où $x^{\left\{ 0\right\} }f=f$. 
\begin{lem}
\label{lem:Dx_ds_M=00005Bn,1,1=00005D}Pour tout entier $r\geq0$
on a:\medskip{}

\hspace{3cm}%
\begin{tabular}{c|c|c|cc}
\diagbox{$w$\ }{} & $\partial_{x}\left(w\right)$ & $\partial_{y}\left(w\right)$ & $\partial_{z}\left(w\right)$ & \tabularnewline
\cline{1-4} 
$x^{\left\{ r\right\} }\left(\left(xy\right)z\right)$ & $\overset{r}{\underset{i=1}{\sum}}t^{i}+t^{r+2},$ & $t^{r+2},$ & $t^{r+1};$ & %
\begin{tabular}{c}
\tabularnewline
\tabularnewline
\tabularnewline
\end{tabular}\tabularnewline
\cline{1-4} 
$x^{\left\{ r\right\} }\left(\left(xz\right)y\right)$ & $\overset{r}{\underset{i=1}{\sum}}t^{i}+t^{r+2},$ & $t^{r+1},$ & $t^{r+2};$ & %
\begin{tabular}{c}
\tabularnewline
\tabularnewline
\tabularnewline
\end{tabular}\tabularnewline
\cline{1-4} 
$x^{\left\{ r\right\} }\left(yz\right)$ & $\overset{r}{\underset{i=1}{\sum}}t^{i},$ & $t^{r+1},$ & $t^{r+1}.$ & %
\begin{tabular}{c}
\tabularnewline
\tabularnewline
\tabularnewline
\end{tabular}\tabularnewline
\cline{1-4} 
\end{tabular}

\medskip{}
\end{lem}

\begin{proof}
En effet, on a $\partial_{x}\bigl(x^{\left\{ r\right\} }\left(\left(xy\right)z\right)\bigr)=t\left(1+\partial_{x}\bigl(x^{\left\{ r-1\right\} }\left(\left(xy\right)z\right)\bigr)\right)$,
on en déduit récursivement que $\partial_{x}\bigl(x^{\left\{ r\right\} }\left(\left(xy\right)z\right)\bigr)=\sum_{i=1}^{r}t^{i}+t^{r}\partial_{x}\left(\left(xy\right)z\right)$
or $\partial_{x}\left(\left(xy\right)z\right)=t^{2}$. Ensuite,
$\partial_{y}\bigl(x^{\left\{ r\right\} }\left(\left(xy\right)z\right)\bigr)=t\partial_{y}\bigl(x^{\left\{ r-1\right\} }\left(\left(xy\right)z\right)\bigr)$
et $\partial_{z}\bigl(x^{\left\{ r\right\} }\left(\left(xy\right)z\right)\bigr)=t\partial_{z}\bigl(x^{\left\{ r-1\right\} }\left(\left(xy\right)z\right)\bigr)$
d'où l'on déduit que $\partial_{y}\bigl(x^{\left\{ r\right\} }\left(\left(xy\right)z\right)\bigr)=t^{r}\partial_{y}\bigl(\left(xy\right)z\bigr)$
et $\partial_{z}\bigl(x^{\left\{ r\right\} }\left(\left(xy\right)z\right)\bigr)=t^{r}\partial_{z}\bigl(\left(xy\right)z\bigr)$
avec $\partial_{y}\bigl(\left(xy\right)z\bigr)=t^{2}$ et $\partial_{z}\bigl(\left(xy\right)z\bigr)=t$.
On en déduit les résultats concernant les monômes $x^{\left\{ r\right\} }\left(\left(xz\right)y\right)$
par échange des rôles de $y$ et $z$. De $\partial_{x}\left(x^{\left\{ r\right\} }\left(yz\right)\right)=t\left(1+\partial_{x}\left(x^{\left\{ r-1\right\} }\left(yz\right)\right)\right)$
on déduit $\partial_{x}\left(x^{\left\{ r\right\} }\left(yz\right)\right)=\sum_{i=1}^{r}t^{i}$
et de $\partial_{z}\left(x^{\left\{ r\right\} }\left(yz\right)\right)=t\partial_{y}\left(x^{\left\{ r-1\right\} }\left(yz\right)\right)$
il résulte $\partial_{y}\left(x^{\left\{ r\right\} }\left(yz\right)\right)=t^{r}\partial_{y}\left(yz\right)=t^{r+1}$,
on en déduit en échangeant $y$ et $z$ que $\partial_{z}\left(x^{\left\{ r\right\} }\left(yz\right)\right)=t^{r+1}$.
\end{proof}
On considère l'ensemble $\mathscr{F}=\left\{ x^{\left\{ r\right\} }\left(\left(xy\right)z\right),x^{\left\{ r\right\} }\left(\left(xz\right)y\right),x^{\left\{ r\right\} }\left(yz\right);r\geq0\right\} $
et on note $\mathbb{Q}\left\langle \mathscr{F}\right\rangle $
le $\mathbb{Q}$-espace vectoriel engendré par l'ensemble $\mathscr{F}$. 
\begin{prop}
Il n'existe pas de train identité évanescente de degré $\left(1,1,1\right)$.

Pour tout $n\geq2$ et tout $w\in\frak{M}\left(x,y,z\right)_{\left[n,1,1\right]}$
tel que $w\notin\mathscr{F}$, il existe un unique polynôme $P_{w}\in\mathbb{Q}\left\langle \mathscr{F}\right\rangle $
avec $\left|P_{w}\right|_{x}\leq n$ tel que le polynôme $w-P_{w}$
soit une train identité évanescente.
\end{prop}

\begin{proof}
Soit $f=\lambda_{1}x\left(yz\right)+\lambda_{2}\left(xy\right)z+\lambda_{3}\left(xz\right)y+\mu_{1}xy+\mu_{2}xz+\mu_{3}yz$,
on a $\partial_{x}f=\left(\lambda_{2}+\lambda_{3}\right)t^{2}+\left(\lambda_{1}+\mu_{1}+\mu_{2}\right)t$,
$\partial_{y}f=\left(\lambda_{1}+\lambda_{2}\right)t^{2}+\left(\lambda_{3}+\mu_{1}+\mu_{3}\right)t$
et $\partial_{z}f=\left(\lambda_{1}+\lambda_{3}\right)t^{2}+\left(\lambda_{2}+\mu_{2}+\mu_{3}\right)t$,
on en déduit sans difficulté que $\partial_{x}f=\partial_{y}f=\partial_{z}f=0$
si et seulement si on a $f=0$, il n'existe donc pas de train
polynôme évanescent de degré $\left(1,1,1\right)$.

On prend $n\geq2$, soit $w\in\frak{M}\left(x,y,z\right)_{\left[n,1,1\right]}$
tel que $w\notin\mathscr{F}$, d'après le résultat \emph{c})
du corollaire \ref{cor:Prop_de_Di} les degrés en $x$ de $\partial_{x}w$,
$\partial_{y}w$ et $\partial_{z}w$ sont $\leq n+1$, soient
$\partial_{x}w=\sum_{k=1}^{n+1}\alpha_{k}t^{k}$, $\partial_{y}w=\sum_{k=1}^{n+1}\beta_{k}t^{k}$
et $\partial_{z}w=\sum_{k=1}^{n+1}\gamma_{k}t^{k}$. On pose
\[
P_{w}=\sum_{k=0}^{n-1}\lambda_{k}x^{\left\{ k\right\} }\left(\left(xy\right)z\right)+\sum_{k=0}^{n-1}\mu_{k}x^{\left\{ k\right\} }\left(\left(xz\right)y\right)+\sum_{k=0}^{n}\nu_{k}x^{\left\{ k\right\} }\left(yz\right)
\]
en appliquant les relations du lemme \ref{lem:Dx_ds_M=00005Bn,1,1=00005D}
on obtient 
\begin{align*}
\partial_{x}P_{w}= & \;\left(\lambda_{n-1}+\mu_{n-1}\right)t^{n+1}+\left(\lambda_{n-2}+\mu_{n-2}+\nu_{n}\right)t^{n}\\
 & +\sum_{i=2}^{n-1}\left(\sum_{k=i}^{n-1}\left(\lambda_{k}+\mu_{k}+\nu_{k}\right)+\left(\lambda_{i-2}+\mu_{i-2}\right)+\nu_{n}\right)t^{i}+\left(\sum_{k=1}^{n}\left(\lambda_{k}+\mu_{k}+\nu_{k}\right)+\nu_{n}\right)t,\\
\partial_{y}P_{w}= & \;\left(\lambda_{n-1}+\nu_{n}\right)t^{n+1}+\sum_{k=2}^{n}\left(\lambda_{k-2}+\mu_{k-1}+\nu_{k-1}\right)t^{k}+\left(\mu_{0}+\nu_{0}\right)t,\\
\partial_{z}P_{w}= & \;\left(\mu_{n-1}+\nu_{n}\right)t^{n+1}+\sum_{k=2}^{n}\left(\lambda_{k-1}+\mu_{k-2}+\nu_{k-1}\right)t^{k}+\left(\lambda_{0}+\nu_{0}\right)t.
\end{align*}

On a donc $\partial_{x}w=\partial_{x}P_{w}$ si et seulement
si 
\[
\begin{cases}
\lambda_{n-1}+\mu_{n-1} & =\alpha_{n+1}\\
\lambda_{n-2}+\mu_{n-2}+\nu_{n} & =\alpha_{n}\\
\left(\lambda_{k-2}+\mu_{k-2}\right)+\sum_{i=k}^{n-1}\left(\lambda_{i}+\mu_{i}+\nu_{i}\right)+\nu_{n} & =\alpha_{k},\quad\left(2\leq k\leq n-1\right)\\
\sum_{i=1}^{n-1}\left(\lambda_{i}+\mu_{i}+\nu_{i}\right)+\nu_{n} & =\alpha_{1}
\end{cases}
\]
et on a $\partial_{y}w=\partial_{y}P_{w}$, $\partial_{z}w=\partial_{z}P_{w}$
si et seulement si on a respectivement
\begin{align*}
\begin{cases}
\lambda_{n-1}+\nu_{n} & =\beta_{n+1}\\
\lambda_{k-2}+\mu_{k-1}+\nu_{k-1} & =\beta_{k},\quad\left(2\leq k\leq n\right),\\
\mu_{0}+\nu_{0} & =\beta_{1},
\end{cases}
\end{align*}
et 
\[
\begin{cases}
\mu_{n-1}+\nu_{n} & =\gamma_{n+1}\\
\lambda_{k-1}+\mu_{k-2}+\nu_{k-1} & =\gamma_{k},\quad\left(2\leq k\leq n\right)\\
\lambda_{0}+\nu_{0} & =\gamma_{1}.
\end{cases}
\]
On peut noter que dans ces deux systèmes on a $\sum_{k=0}^{n-1}\left(\lambda_{k}+\mu_{k}\right)+\sum_{k=0}^{n}\nu_{k}=\sum_{k=1}^{n+1}\beta_{k}=\partial_{y}P_{w}\left(1\right)=1$
d'après le résultat \emph{d}) du corollaire \ref{cor:Prop_de_Di},
par conséquent $P_{w}\left(1,1,1\right)=1.$

La solution de ces systèmes d'équations sont:
\begin{align*}
\lambda_{k} & =\sum_{i=k+2}^{n+1}\frac{1}{2^{i-k-1}}\left(\alpha_{i}+\beta_{i}-\left(2^{i-k-1}-1\right)\gamma_{i}\right),\quad\left(0\leq k\leq n-1\right)\\
\mu_{k} & =\sum_{i=k+2}^{n+1}\frac{1}{2^{i-k-1}}\left(\alpha_{i}-\left(2^{i-k-1}-1\right)\beta_{i}+\gamma_{i}\right),\quad\left(0\leq k\leq n-1\right)\\
\nu_{k} & =\frac{1}{2}\left(-\alpha_{k+1}+\beta_{k+1}+\gamma_{k+1}\right)\\
 & +\sum_{i=k+2}^{n+1}\frac{1}{2^{i-k}}\left(-3\alpha_{i}+\left(2^{i-k}-3\right)\beta_{i}+\left(2^{i-k}-3\right)\gamma_{i}\right),\quad\left(1\leq k\leq n-1\right)\\
\nu_{0} & =1-\sum_{i=2}^{n+1}\frac{1}{2^{i-1}}\left(\alpha_{i}+\beta_{i}+\gamma_{i}\right),\\
\nu_{n} & =\frac{1}{2}\left(-\alpha_{n+1}+\beta_{n+1}+\gamma_{n+1}\right).
\end{align*}
\end{proof}
\begin{cor}
Pour $n\geq2$, l'espace vectoriel des train identités évanescentes
de degré $\left(n,1,1\right)$ est de dimension $W_{\left[n,1,1\right]}-3$.
\end{cor}

\begin{thm}
Pour tout entier $p,q\geq2$, $r\geq0$ et $t,t_{1},t_{2}\in\left\{ y,z\right\} $,
$t_{1}\neq t_{2}$, on pose:
\begin{align*}
E_{p,q}^{\left[n\right]}\left(x\right) & =x^{p}x^{q}-x^{p+1}-x^{q+1}+x^{2};\\
E_{p,q}^{\left[n,1\right]}\left(x,t\right) & =x^{p}\left(x^{q}t\right)-x\left(xt\right)-x^{p}t-x^{q+1}t+x^{2}t+xt;\\
E_{p,\left\{ r\right\} }^{\left[n,1\right]}\left(x,t\right) & =x^{p}\bigl(x^{\left\{ r\right\} }t\bigr)-x^{p}t-x^{\left\{ r+1\right\} }t+xt;\\
E_{\left\{ r\right\} ,p}^{\left[n,1\right]}\left(x,t\right) & =x^{\left\{ r\right\} }\left(x^{p}t\right)-x^{\left\{ r+1\right\} }t-x^{p+r}t+x^{r+1}t;\\
E_{p,q}^{\left[n,1,1\right]}\left(x,y,z\right) & =\left(x^{p}y\right)\left(x^{q}z\right)+x^{\left\{ p\right\} }\left(yz\right)+x^{\left\{ q\right\} }\left(yz\right)-\sum_{i=0}^{p-1}x^{\left\{ i\right\} }\Bigl(\left(xy\right)z+\left(xz\right)y-2yz\Bigr)\\
 & \quad-\sum_{i=1}^{q-1}x^{\left\{ i\right\} }\Bigl(\left(xy\right)z+\left(xz\right)y-2yz\Bigr)-4x\left(yz\right)-yz;\\
E_{\left\{ r\right\} ,\left\{ s\right\} }^{\left[n,1,1\right]}\left(x,y,z\right) & =\bigl(x^{\left\{ r\right\} }y\bigr)\bigl(x^{\left\{ s\right\} }z\bigr)-\sum_{i=0}^{r-1}x^{\left\{ i\right\} }\Bigl(\left(xy\right)z-yz\Bigr)-\sum_{i=0}^{s-1}x^{\left\{ i\right\} }\Bigl(\left(xy\right)z-yz\Bigr)-yz;\\
E_{p}^{\left[n,1,1\right]}\left(x,t_{1},t_{2}\right) & =\left(x^{p}t_{1}\right)t_{2}-\sum_{i=1}^{p-1}x^{\left\{ i\right\} }\Bigl(\left(xt_{1}\right)t_{2}+\left(xt_{2}\right)t_{1}-2t_{1}t_{2}\Bigr)+x^{\left\{ p\right\} }\left(t_{1}t_{2}\right)\\
 & \quad-\left(xt_{1}\right)t_{2}-x\left(t_{1}t_{2}\right);\\
E_{p,\left\{ r\right\} }^{\left[n,1,1\right]}\left(x,t_{1},t_{2}\right) & =\left(x^{p}t_{1}\right)\left(x^{\left\{ r+1\right\} }t_{2}\right)-\sum_{i=0}^{^{p-1}}x^{\left\{ i\right\} }\Bigl(\left(xt_{1}\right)t_{2}+\left(xt_{2}\right)t_{1}-2t_{1}t_{2}\Bigr)\\
 & \quad-\sum_{i=1}^{r}x^{\left\{ i\right\} }\Bigl(\left(xt_{2}\right)t_{1}-t_{1}t_{2}\Bigr)+x^{\left\{ p\right\} }\left(t_{1}t_{2}\right)-x\left(t_{1}t_{2}\right)-t_{1}t_{2};\\
F_{p,\left\{ r\right\} }^{\left[n,1,1\right]}\left(x,y,z\right) & =x^{p}\left(x^{\left\{ r\right\} }\left(yz\right)\right)-\sum_{i=0}^{p-2}x^{\left\{ i\right\} }\Bigl(\left(xy\right)z+\left(xz\right)y-2yz\Bigr)\\
 & \quad+x^{\left\{ p-1\right\} }\left(yz\right)-x^{\left\{ r+1\right\} }\left(yz\right)-yz;\\
G_{p,\left\{ r\right\} }^{\left[n,1,1\right]}\left(x,t_{1},t_{2}\right) & =x^{p}\left(x^{\left\{ r\right\} }\left(\left(xt_{1}\right)t_{2}\right)\right)-\sum_{i=0}^{p-2}x^{\left\{ i\right\} }\Bigl(\left(xt_{1}\right)t_{2}+\left(xt_{2}\right)t_{1}-2t_{1}t_{2}\Bigr)\\
 & \quad-x^{\left\{ r+1\right\} }\left(\left(xt_{1}\right)t_{2}\right)+x^{\left\{ p-1\right\} }\left(t_{1}t_{2}\right)-t_{1}t_{2};
\end{align*}

\medskip{}

Soient $\mathscr{H}$ l'idéal engendré par la famille de polynômes
\[
\left(E_{p,q}^{\left[n\right]},E_{p,q}^{\left[n,1\right]},E_{p,\left\{ r\right\} }^{\left[n,1\right]},E_{\left\{ r\right\} ,p}^{\left[n,1\right]},E_{p,q}^{\left[n,1,1\right]},E_{\left\{ r\right\} ,\left\{ s\right\} }^{\left[n,1,1\right]},E_{p}^{\left[n,1,1\right]},E_{p,\left\{ r\right\} }^{\left[n,1,1\right]},F_{p,\left\{ r\right\} }^{\left[n,1,1\right]},G_{p,\left\{ r\right\} }^{\left[n,1,1\right]}\right)_{\substack{p,q\geq2\\
r,s\geq0
}
}
\]
et $\pi:K\left(x\right)\rightarrow{}^{K\left(x\right)}\!/_{\mathscr{H}}$
la surjection canonique. 

Alors pour tout $n\geq2$ et tout monôme $w\in\frak{M}_{\left[n,1,1\right]}\left(x,y,z\right)$
tel que $w\notin\mathscr{F}$ on a $\pi\left(w\right)=P_{w}$
et pour tout $f\in\bigoplus_{n\geq2}K\left(x,y,z\right)_{\left[n,1,1\right]}$,
le polynôme $f-\pi\left(f\right)$ est une train identité évanescente.
\end{thm}

\begin{proof}
On a montré aux théorèmes \ref{thm:Evanesc_type=00005Bn=00005D_id=0000E9al}
et \ref{thm:Evanesc_type=00005Bn,1=00005D_id=0000E9al} que les
polynômes $E_{p,q}^{\left[n\right]}$, $E_{p,q}^{\left[n,1\right]}$,
$E_{p,\left\{ r\right\} }^{\left[n,1\right]}$ et $E_{\left\{ r\right\} ,p}^{\left[n,1\right]}$
sont évanescents, montrons-le pour les autres polynômes de l'énoncé.

Pour tout $p,q\geq2$ et $r,s\geq0$, on a $\partial_{x}\left(\left(x^{p}y\right)\left(x^{q}z\right)\right)=t\left(\partial_{x}\left(x^{p}y\right)+\partial_{x}\left(x^{q}z\right)\right)$,
puis $\partial_{y}\left(\left(x^{p}y\right)\left(x^{q}z\right)\right)=t\partial_{y}\left(x^{p}y\right)=t^{2}$,
de même $\partial_{z}\left(\left(x^{p}y\right)\left(x^{q}z\right)\right)=t^{2}$
et en utilisant la relation (\ref{eq:Dx(x^p*y)}) on obtient:
\begin{align*}
\partial_{x}\left(\left(x^{p}y\right)\left(x^{q}z\right)\right) & =2\sum_{i=3}^{p}t^{i}+2t^{p+1}+2t^{q+1}, & \partial_{y}\left(\left(x^{p}y\right)\left(x^{q}z\right)\right) & =t^{2}, & \partial_{z}\left(\left(x^{p}y\right)\left(x^{q}z\right)\right) & =t^{2}.
\end{align*}

On a $\partial_{x}\left(\left(x^{\left\{ r\right\} }y\right)\left(x^{\left\{ s\right\} }z\right)\right)=t\left(\partial_{x}\left(x^{\left\{ r\right\} }y\right)+\partial_{x}\left(x^{\left\{ s\right\} }z\right)\right)$,
ensuite $\partial_{y}\left(\left(x^{\left\{ s\right\} }y\right)\left(x^{\left\{ s\right\} }z\right)\right)=t\partial_{y}\left(x^{\left\{ r\right\} }y\right)=t^{2}\partial_{y}\left(x^{\left\{ r-1\right\} }y\right)$
et de même $\partial_{z}\left(\left(x^{\left\{ s\right\} }y\right)\left(x^{\left\{ s\right\} }z\right)\right)=t^{2}\partial_{z}\left(x^{\left\{ s-1\right\} }z\right)$,
on en déduit avec la relation (\ref{eq:Dx(x^=00007Bp=00007D*y)})
et par récurrence que
\begin{align*}
\partial_{x}\bigl(\bigl(x^{\left\{ r\right\} }y\bigr)\bigl(x^{\left\{ s\right\} }z\bigr)\bigr) & =\sum_{i=2}^{r+1}t^{i}+\sum_{i=2}^{s+1}t^{i},\\
\partial_{y}\bigl(\bigl(x^{\left\{ r\right\} }y\bigr)\bigl(x^{\left\{ s\right\} }z\bigr)\bigr) & =t^{r+1},\hspace{10mm}\partial_{z}\bigl(\bigl(x^{\left\{ r\right\} }y\bigr)\bigl(x^{\left\{ s\right\} }z\bigr)\bigr)=t^{s+1}.
\end{align*}

De $\partial_{x}\left(\left(x^{p}y\right)\left(x^{\left\{ r\right\} }z\right)\right)=t\left(\partial_{x}\left(x^{p}y\right)+\partial_{x}\left(x^{\left\{ r\right\} }z\right)\right)$,
$\partial_{y}\left(\left(x^{p}y\right)\left(x^{\left\{ r\right\} }z\right)\right)=t\partial_{y}\left(x^{p}y\right)$,
$\partial_{z}\left(\left(x^{p}y\right)\left(x^{\left\{ r\right\} }z\right)\right)=t\partial_{z}\left(x^{\left\{ r\right\} }z\right)$
et des relations (\ref{eq:Dx(x^p*y)}) et (\ref{eq:Dx(x^=00007Bp=00007D*y)})
on déduit:
\begin{align*}
\partial_{x}\left(\left(x^{p}y\right)\left(x^{\left\{ r\right\} }z\right)\right) & =2t^{p+1}+\sum_{i=3}^{p}t^{i}+\sum_{i=2}^{r+1}t^{i},\\
\partial_{y}\left(\left(x^{p}y\right)\left(x^{\left\{ r\right\} }z\right)\right) & =t^{2},\hspace{10mm}\partial_{z}\left(\left(x^{p}y\right)\left(x^{\left\{ r\right\} }z\right)\right)=t^{r+1}.
\end{align*}

On a $\partial_{x}\left(x^{p}\left(x^{\left\{ r\right\} }\left(yz\right)\right)\right)=t\left(\partial_{x}\left(x^{p}\right)+\partial_{x}\left(x^{\left\{ r\right\} }\left(yz\right)\right)\right)=t\partial_{x}\left(x^{p}\right)+t^{2}\partial_{x}\left(x^{\left\{ r-1\right\} }\left(yz\right)\right)$,
$\partial_{y}\left(x^{p}\left(x^{\left\{ r\right\} }\left(yz\right)\right)\right)=t\partial_{y}\left(x^{\left\{ r\right\} }\left(yz\right)\right)=t^{2}\partial_{y}\left(x^{\left\{ r-1\right\} }\left(yz\right)\right)$
et de même $\partial_{z}\left(x^{p}\left(x^{\left\{ r\right\} }\left(yz\right)\right)\right)=t^{2}\partial_{z}\left(x^{\left\{ r-1\right\} }\left(yz\right)\right)$
. 

De manière analogue, on a $\partial_{x}\left(x^{p}\left(x^{\left\{ r\right\} }\left(\left(xy\right)z\right)\right)\right)=t\left(\partial_{x}\left(x^{p}\right)+\partial_{x}\left(x^{\left\{ r\right\} }\left(\left(xy\right)z\right)\right)\right)=t\partial_{x}\left(x^{p}\right)+t^{2}\partial_{x}\left(x^{\left\{ r-1\right\} }\left(\left(xy\right)z\right)\right)$,
ainsi que $\partial_{y}\left(x^{p}\left(x^{\left\{ r\right\} }\left(\left(xy\right)z\right)\right)\right)=t\partial_{y}\left(x^{\left\{ r\right\} }\left(\left(xy\right)z\right)\right)=t^{2}\partial_{y}\left(x^{\left\{ r-1\right\} }\left(\left(xy\right)z\right)\right)$
et $\partial_{z}\left(x^{p}\left(x^{\left\{ r\right\} }\left(\left(xy\right)z\right)\right)\right)=t^{2}\partial_{z}\left(x^{\left\{ r-1\right\} }\left(\left(xy\right)z\right)\right)$.
On en déduit en appliquant la relation (\ref{eq:Dx(x^k)}) et
par récurrence que
\begin{align*}
\partial_{x}\left(x^{p}\left(x^{\left\{ r\right\} }\left(yz\right)\right)\right) & =2t^{p}+\sum_{i=2}^{p-1}t^{i}+\sum_{i=2}^{r+1}t^{i},\\
\partial_{y}\left(x^{p}\left(x^{\left\{ r\right\} }\left(yz\right)\right)\right) & =t^{r+2},\hspace{10mm}\partial_{z}\left(x^{p}\left(x^{\left\{ r\right\} }\left(yz\right)\right)\right)=t^{r+2}.
\end{align*}

et
\begin{align*}
\partial_{x}\left(x^{p}\left(x^{\left\{ r\right\} }\left(\left(xy\right)z\right)\right)\right) & =2t^{p}+t^{r+3}+\sum_{i=2}^{p-1}t^{i}+\sum_{i=2}^{r+1}t^{i},\\
\partial_{y}\left(x^{p}\left(x^{\left\{ r\right\} }\left(\left(xy\right)z\right)\right)\right) & =t^{r+3},\hspace{10mm}\partial_{z}\left(x^{p}\left(x^{\left\{ r\right\} }\left(\left(xy\right)z\right)\right)\right)=t^{r+2}.
\end{align*}

Avec ces résultats et les relations du lemme \ref{lem:Dx_ds_M=00005Bn,1,1=00005D}
on montre par de simples calculs que les polynômes $E_{p,q}^{\left[n,1,1\right]}$,
$E_{\left\{ r\right\} ,\left\{ s\right\} }^{\left[n,1,1\right]}$,
$E_{p}^{\left[n,1,1\right]}$, $E_{p,\left\{ r\right\} }^{\left[n,1,1\right]}$,
$F_{p,\left\{ r\right\} }^{\left[n,1,1\right]}$ et $G_{p,\left\{ r\right\} }^{\left[n,1,1\right]}$
sont évanescents.

Soit $w\in\frak{M}\left(x,y,z\right)_{\left[n,1,1\right]}$ tel
que $w\notin\mathscr{F}$, montrons par récurrence sur le degré
$n$ en $x$ de $w$ que le polynôme $w-\pi\left(w\right)$ est
évanescent. Le résultat est vrai pour $n=2$ comme on peut le
vérifier sur les générateurs des train polynômes de degré $\left(2,1,1\right)$
donnés ci-dessous. Supposons le résultat vrai pour tous les monômes
de type $\left[p,1,1\right]$ avec $2\leq p<n$. Il existe $u,v\in\mathfrak{M}\left(x,y,z\right)$
tel que $w=uv$ avec $\bigl|u\bigr|_{x},\bigl|v\bigr|_{x}<n$
on a donc $u\in\frak{M}\left(x,y\right)_{\left[n-p,1\right]}$,
$v\in\frak{M}\left(x,z\right)_{\left[p,1\right]}$ ou bien $u\in\frak{M}\left(x\right)_{\left[n-p\right]}$,
$v\in\frak{M}\left(x,y,z\right)_{\left[p,1,1\right]}$. On a
$\partial_{x}\left(w-\pi\left(w\right)\right)=\partial_{x}\left(uv-\pi\left(u\right)\pi\left(v\right)\right)=t\partial_{x}\left(u-\pi\left(u\right)\right)+t\partial_{x}\left(v-\pi\left(v\right)\right)$
et de même $\partial_{y}\left(w-\pi\left(w\right)\right)=t\partial_{y}\left(u-\pi\left(u\right)\right)+t\partial_{y}\left(v-\pi\left(v\right)\right)$
et $\partial_{z}\left(w-\pi\left(w\right)\right)=t\partial_{z}\left(u-\pi\left(u\right)\right)+t\partial_{z}\left(v-\pi\left(v\right)\right)$. 

Dans le cas $u\in\frak{M}\left(x,y\right)_{\left[n-p,1\right]}$,
$v\in\frak{M}\left(x,z\right)_{\left[p,1\right]}$, d'après le
théorème \ref{lem:Dx_ds_M=00005Bn,1,1=00005D} les polynômes
$u-\pi\left(u\right)$ et $v-\pi\left(v\right)$ sont évanescents,
on a donc $\partial_{x}\left(w-\pi\left(w\right)\right)=\partial_{y}\left(w-\pi\left(w\right)\right)=\partial_{z}\left(w-\pi\left(w\right)\right)=0$.

Quand $u\in\frak{M}\left(x\right)_{\left[n-p\right]}$, $v\in\frak{M}\left(x,y,z\right)_{\left[p,1,1\right]}$,
d'après le théorème \ref{thm:Evanesc_type=00005Bn=00005D_id=0000E9al}
et l'hypothèse de récurrence les polynômes $u-\pi\left(u\right)$
et $v-\pi\left(v\right)$ sont évanescents. 

Il est clair que pour tout $w\in\frak{M}\left(x,y,z\right)_{\left[n,1,1\right]}$
tel que $w\notin\mathscr{F}$ on a $\pi\left(w\right)\in\mathbb{Z}\left[\mathscr{F}\right]$
et donc par unicité du polynôme $P_{w}$ on a $\pi\left(w\right)=P_{w}$.
\end{proof}
En utilisant ce théorème on peut donner les générateurs des train
identités évanescentes
\begin{align*}
 & \text{-- de degré }\left(2,1,1\right)\\
 & \;x^{2}\left(yz\right)-\left(xy\right)z-\left(xz\right)y+yz,\\
 & \left(xy\right)\left(xz\right)-\left(xy\right)z-\left(xz\right)y+yz,\\
 & \left(x\left(xy\right)\right)z-x\left(\left(xy\right)z\right)-\left(xy\right)z+x\left(yz\right),\\
 & \left(x\left(xz\right)\right)y-x\left(\left(xz\right)y\right)-\left(xz\right)y+x\left(yz\right),\\
 & \left(x^{2}y\right)z+x\left(x\left(yz\right)\right)-x\left(\left(xy\right)z\right)-x\left(\left(xz\right)y\right)-\left(xy\right)z+x\left(yz\right),\\
 & \left(x^{2}z\right)y+x\left(x\left(yz\right)\right)-x\left(\left(xy\right)z\right)-x\left(\left(xz\right)y\right)-\left(xz\right)y+x\left(yz\right).\\
\\
 & \text{-- de degré }\left(3,1,1\right)\\
 & \;x\left(x^{2}\left(yz\right)\right)-x\left(\left(xy\right)z\right)-x\left(\left(xz\right)y\right)+x\left(yz\right),\\
 & \;x\left(\left(xy\right)\left(xz\right)\right)-x\left(\left(xy\right)z\right)-x\left(\left(xz\right)y\right)+x\left(yz\right),\\
 & \;x\left(\left(x\left(xy\right)\right)z\right)-x\left(x\left(\left(xy\right)z\right)\right)-x\left(\left(xy\right)z\right)+x\left(x\left(yz\right)\right),\\
 & \;x\left(\left(x\left(xz\right)\right)y\right)-x\left(x\left(\left(xz\right)y\right)\right)-x\left(\left(xz\right)y\right)+x\left(x\left(yz\right)\right),\\
 & \;x^{2}\left(x\left(yz\right)\right)-x\left(x\left(yz\right)\right)-\left(xz\right)y-\left(xy\right)z+x\left(yz\right)+yz,\\
 & \;x^{2}\left(\left(xy\right)z\right)-x\left(\left(xy\right)z\right)-\left(xz\right)y-\left(xy\right)z+x\left(yz\right)+yz,\\
 & \;x^{2}\left(y\left(xz\right)\right)-x\left(y\left(xz\right)\right)-\left(xz\right)y-\left(xy\right)z+x\left(yz\right)+yz,\\
 & \left(x\left(xy\right)\right)\left(xz\right)-x\left(\left(xy\right)z\right)+x\left(yz\right)-\left(xy\right)z-y\left(xz\right)+yz,\\
 & \left(x\left(xz\right)\right)\left(xy\right)-x\left(\left(xz\right)y\right)+x\left(yz\right)-\left(xy\right)z-y\left(xz\right)+yz,\\
 & \left(x^{2}\left(xy\right)\right)z+x\left(x\left(yz\right)\right)-2x\left(\left(xy\right)z\right)-x\left(y\left(xz\right)\right)+2x\left(yz\right)-\left(xy\right)z,\\
 & \left(x^{2}\left(xz\right)\right)y+x\left(x\left(yz\right)\right)-2x\left(\left(xz\right)y\right)-x\left(\left(xy\right)z\right)+2x\left(yz\right)-\left(xz\right)y,\\
 & \left(x\left(x\left(xy\right)\right)\right)z-x\left(x\left(\left(xy\right)z\right)\right)+x\left(x\left(yz\right)\right)-x\left(\left(xy\right)z\right)-\left(xy\right)z+x\left(yz\right),\\
 & \left(x\left(x\left(xz\right)\right)\right)y-x\left(x\left(y\left(xz\right)\right)\right)+x\left(x\left(yz\right)\right)-x\left(y\left(xz\right)\right)-y\left(xz\right)+x\left(yz\right),\\
 & \;x^{3}\left(yz\right)+x\left(x\left(yz\right)\right)-x\left(\left(xy\right)z\right)-x\left(y\left(xz\right)\right)+x\left(yz\right)-y\left(xz\right)-z\left(xy\right)+yz,\\
 & \left(x^{2}y\right)\left(xz\right)+x\left(x\left(yz\right)\right)-x\left(\left(xy\right)z\right)-x\left(y\left(xz\right)\right)+x\left(yz\right)-y\left(xz\right)-\left(xy\right)z+yz,\\
 & \left(x^{2}z\right)\left(xy\right)+x\left(x\left(yz\right)\right)-x\left(\left(xy\right)z\right)-x\left(y\left(xz\right)\right)+x\left(yz\right)-y\left(xz\right)-\left(xy\right)z+yz,\\
 & \;x\left(\left(x^{2}y\right)z\right)+x\left(x\left(x\left(yz\right)\right)\right)-x\left(x\left(y\left(xz\right)\right)\right)-x\left(x\left(\left(xy\right)z\right)\right)+x\left(x\left(yz\right)\right)-x\left(\left(xy\right)z\right),\\
 & \;x\left(\left(x^{2}z\right)y\right)+x\left(x\left(x\left(yz\right)\right)\right)-x\left(x\left(y\left(xz\right)\right)\right)-x\left(x\left(\left(xy\right)z\right)\right)+x\left(x\left(yz\right)\right)-x\left(y\left(xz\right)\right),\\
 & \left(x\left(x^{2}y\right)\right)z+x\left(x\left(x\left(yz\right)\right)\right)-x\left(x\left(\left(xy\right)z\right)\right)-x\left(x\left(y\left(xz\right)\right)\right)+x\left(x\left(yz\right)\right)-x\left(\left(xy\right)z\right)\\
 & \;-\left(xy\right)z+x\left(yz\right),\\
 & \left(x\left(x^{2}z\right)\right)y+x\left(x\left(x\left(yz\right)\right)\right)-x\left(x\left(\left(xy\right)z\right)\right)-x\left(x\left(y\left(xz\right)\right)\right)+x\left(x\left(yz\right)\right)-x\left(y\left(xz\right)\right)\\
 & \;-y\left(xz\right)+x\left(yz\right),\\
 & \left(x^{3}y\right)z+x\left(x\left(x\left(yz\right)\right)\right)-x\left(x\left(y\left(xz\right)\right)\right)-x\left(x\left(\left(xy\right)z\right)\right)+2x\left(x\left(yz\right)\right)-x\left(\left(xy\right)z\right)\\
 & \;-x\left(y\left(xz\right)\right)-\left(xy\right)z+x\left(yz\right),\\
 & \left(x^{3}z\right)y+x\left(x\left(x\left(yz\right)\right)\right)-x\left(x\left(y\left(xz\right)\right)\right)-x\left(x\left(\left(xy\right)z\right)\right)+2x\left(x\left(yz\right)\right)-x\left(\left(xy\right)z\right)\\
 & \;-x\left(y\left(xz\right)\right)-\left(xz\right)y+x\left(yz\right).
\end{align*}

\subsubsection{Identités homogènes évanescentes de type $\left[n,1,1\right]$.}

\textcompwordmark{}
\begin{prop}
Pour tout $n\geq2$, l'espace des identités homogènes évanescentes
de type $\left[n,1,1\right]$ est engendré au moins $W_{\left[n,1,1\right]}-3n$
identités homogènes évanescentes.
\end{prop}

\begin{proof}
Pour simplifier les notations on pose $N=W_{\left[n,1,1\right]}$
et $\frak{M}\left(x,y,z\right)_{\left[n,1,1\right]}=\left\{ w_{k};1\leq k\leq N\right\} $.
Soit $f=\sum_{k=1}^{N}\alpha_{k}w_{k}$ on cherche $\left(\alpha_{k}\right)_{1\leq k\leq N}$
tel que $f\left(1,1,1\right)=0$, $\partial_{x}f=\partial_{y}f=\partial_{z}f=0$.
Comme pour tout $w\in\mathfrak{M}_{\left[n,1,1\right]}$ on a
$\left|\partial_{x}w\right|,\left|\partial_{y}w\right|,\left|\partial_{z}w\right|\leq n+1$
et que d'après le lemme \ref{lem:Dx_ds_M=00005Bn,1,1=00005D}
il existe dans $\frak{M}\left(x,y,z\right)_{\left[n,1,1\right]}$
des monômes $w$ tel que $\partial_{x}w$, $\partial_{y}w$ ou
$\partial_{z}w$ soit de degré $n+1$, on en déduit que les polynômes
$\partial_{x}f$, $\partial_{y}f$ et $\partial_{z}f$ sont de
degré $n+1$. Par conséquent les relations $\partial_{x}f=0$,
$\partial_{y}f=0$ et $\partial_{z}f=0$ sont équivalentes à
trois systèmes linéaires de $n+1$ équations d'inconnues $\left(\alpha_{k}\right)_{1\leq k\leq N}$,
la condition $f\left(1,1,1\right)=0$ implique que chacun de
ces systèmes est de rang $\leq n$, il en résulte que le système
d'équations $\partial_{x}f=\partial_{y}f=\partial_{z}f=0$ est
de rang $\leq3n$ et donc l'espace des solutions est de dimension
$\geq W_{\left[n,1,1\right]}-3n$.
\end{proof}
En employant la méthode suivie dans la preuve ci-dessus on explicite
les générateurs des identités homogènes évanescentes
\begin{align*}
 & \text{-- de type }\left[2,1,1\right]\\
 & \;x^{2}\left(yz\right)-\left(xy\right)\left(xz\right),\\
 & \left(x^{2}y\right)z-\left(x\left(xy\right)\right)z-x\left(y\left(xz\right)\right)+x\left(x\left(yz\right)\right),\\
 & \left(x^{2}z\right)y-\left(x\left(xz\right)\right)y-x\left(\left(xy\right)z\right)+x\left(x\left(yz\right)\right).\\
\\
 & \text{-- de type }\left[3,1,1\right]\\
 & \;x^{3}\left(yz\right)-\left(x^{2}y\right)\left(xz\right),\\
 & \;x^{3}\left(yz\right)-\left(xy\right)\left(x^{2}z\right),\\
 & \;x\left(x^{2}\left(yz\right)\right)-x\left(\left(xy\right)\left(xz\right)\right),\\
 & \;x^{2}\left(\left(xy\right)z\right)-\left(x\left(xy\right)\right)\left(xz\right),\\
 & \;x^{3}\left(yz\right)-\left(x^{3}y\right)z+\left(x\left(x^{2}y\right)\right)z-x^{2}\left(\left(xy\right)z\right),\\
 & \left(x^{3}y\right)z+x\left(\left(xy\right)\left(xz\right)\right)-x\left(\left(x^{2}z\right)y\right)-\left(x^{2}\left(xy\right)\right)z,\\
 & \left(x^{3}y\right)z-x\left(\left(x^{2}z\right)y\right)-\left(x\left(x\left(xy\right)\right)\right)z+x\left(x\left(z\left(xy\right)\right)\right),\\
 & \left(x^{3}y\right)z-\left(x\left(x^{2}y\right)\right)z-x\left(\left(x\left(xz\right)\right)y\right)+x\left(x\left(y\left(xz\right)\right)\right),\\
 & \left(x^{3}y\right)z-\left(x\left(x\left(xy\right)\right)\right)z-x\left(\left(x\left(xz\right)\right)y\right)+x\left(x\left(x\left(yz\right)\right)\right),\\
 & \;x\left(\left(x^{2}y\right)z\right)-\left(x\left(x^{2}y\right)\right)z-x\left(\left(x\left(xy\right)\right)z\right)+\left(x\left(x\left(xy\right)\right)\right)z,\\
 & \left(x^{3}y\right)z-\left(x^{3}z\right)y+x\left(\left(x\left(xy\right)\right)z\right)-\left(x\left(x\left(xy\right)\right)\right)z+\left(x\left(x^{2}z\right)\right)y-x\left(yx^{2}zy\right),\\
 & \;x^{3}\left(yz\right)-\left(x^{3}y\right)z+\left(x\left(x\left(xy\right)\right)\right)z-x\left(\left(x\left(xy\right)\right)z\right)+x\left(\left(x^{2}z\right)y\right)-x^{2}\left(\left(xz\right)y\right),\\
 & \;x^{3}\left(yz\right)-\left(x^{3}y\right)z+x\left(\left(x^{2}z\right)y\right)-x\left(\left(x\left(xy\right)\right)z\right)+\left(x\left(x\left(xy\right)\right)\right)z-\left(xy\right)\left(x\left(xz\right)\right),\\
 & \left(x^{3}y\right)z-\left(x^{3}z\right)y+x\left(\left(x\left(xy\right)\right)z\right)-\left(x\left(x\left(xy\right)\right)\right)z+\left(x\left(x\left(xz\right)\right)\right)y-x\left(\left(x\left(xz\right)\right)y\right),\\
 & \left(x^{3}z\right)y-\left(x^{2}\left(xz\right)\right)y+x\left(\left(xy\right)\left(xz\right)\right)-x\left(\left(x\left(xy\right)\right)z\right)+\left(x\left(x\left(xy\right)\right)\right)z-\left(x\left(x^{2}y\right)\right)z,\\
 & \;x^{3}\left(yz\right)-2\left(x^{3}y\right)z+\left(x\left(x\left(xy\right)\right)\right)z-x\left(\left(x\left(xy\right)\right)z\right)+x\left(\left(x^{2}z\right)y\right)+\left(x\left(x^{2}y\right)\right)z-x^{2}\left(x\left(yz\right)\right).
\end{align*}
\begin{center}
\medskip{}
\par\end{center}

\end{document}